\documentclass[a4paper,11pt]{article}
\usepackage{amsmath,amssymb,amsbsy,amsfonts,mathrsfs}
\usepackage{microtype}
\usepackage{graphicx}
\usepackage[arrow,matrix,curve]{xy}
\usepackage{color}
\usepackage{a4wide}
\usepackage[normalem]{ulem}
\usepackage{hyperref}
\usepackage{wasysym}
\usepackage{comment}
\usepackage{rotating}
\definecolor{darkgreen}{rgb}{0,0.6,0}
\definecolor{darkred}{rgb}{0.8,0.1,0.1}
\definecolor{midnightblue}{rgb}{0.1,0.1,0.44}
\hypersetup{
     colorlinks=true,         
     linkcolor=midnightblue,
     citecolor=blue,
}

 1

\usepackage{amsthm}

\theoremstyle{plain}
\newtheorem{Theorem}{Theorem}[section]
\newtheorem{Lemma}[Theorem]{Lemma}
\newtheorem{Proposition}[Theorem]{Proposition}
\newtheorem{Corollary}[Theorem]{Corollary}

\theoremstyle{definition}
\newtheorem{Definition}[Theorem]{Definition}
\newtheorem{Example}[Theorem]{Example}

\newtheorem{Remark}[Theorem]{Remark}

\numberwithin{equation}{section}

\newcommand{\eq}{\begin{equation}}
\newcommand{\eqa}{\begin{eqnarray}}
\newcommand{\en}{\end{equation}}
\newcommand{\ena}{\end{eqnarray}}
\newcommand{\enn}{\nonumber \end{equation}}

\def\sk{\vskip .4cm}
\def\noi{\noindent}

\def\de{\delta}
\def\epsi{{\varepsilon}}
\def\st {\star}
\def\f{{\rm{f}\,}}
\def\of{{\overline{{\rm{f}}\,}}}

\def\D/h{\widehat{\fmslash D}}

\def\al{\alpha}
\def\la{\lambda}
\def\be{\beta}
\def\ga{\gamma}

\def\de{\delta}

\def\5bar{{\overline 5}}

\def\MMM{{\mathscr M}^{}}
\def\RR{{\mathcal R}}

\newcommand{\MMMod}[3]{{}^{#1}{\!}_{#2}\MMM{\!}_{#3}}
\newcommand{\AAAlg}[3]{{}^{#1}{\!}_{#2}\AAA{\!}_{#3} }

\def\R{{R}}
\def\oR{{\overline{\R}}}

\def\gg{{\hat g}}

\def\UU{H}

\def\FF{\mathcal F}

\def\varepsi{\varepsilon}

\def\s'O{\stackrel{_{{\displaystyle\st \footnotesize '}}}{_{^{^{\displaystyle\otimes}}}}}

\def\AAs{{A_\st }}

\def\D{\Delta}
\def\1s{{1_\st }}
\def\3s{{3_\st }}
\def\2s{{2_\st }}
\def\ef1{{1_\FF}}
\def\ef2{{3_\FF}}
\def\ef3{{2_\FF}}

\def\hbar{\lambda}

\def\nn{\nonumber}

\def\ker{{\rm ker}}

\def\bbC{\mathbb{C}}
\def\bbK{\mathbb{K}}
\def\bfK{\mathbb{K}}
\def\bbA{\mathbb A}
\def\bbR{\mathbb R}
\def\AA{A}
\def\AAA{\mathscr A}
\def\dd{{\nabla}}
\def\dif{{\mathrm d}}

\def\gg{\mathfrak{g}}

\def\rep{\mathsf{rep}}

\def\trgl{\triangleright}
\def\btrgl{\blacktriangleright}

\def\op{{op}}
\def\cop{{cop}}
\def\eqv{\mathrm{eqv}}

\def\Hom{\text{\sl Hom}}
\def\Con{\text{\sl Con}}
\def\End{\text{\sl End}}

\def\ra{\triangleright}
\def\RA{\blacktriangleright}
\def\pP{P}
\def\qQ{Q}

\def\id{\mathrm{id}}

\title{\bf Noncommutative connections on bimodules \\ and Drinfeld twist deformation}

\author{
{\bf Paolo Aschieri}$^\text{a}$\thanks{e-mail: \texttt{aschieri@to.infn.it}} ~
and~
{\bf Alexander Schenkel}$^\text{b}$\thanks{e-mail: \texttt{schenkel@math.uni-wuppertal.de}} \vspace{2mm}\\
$^\text{a}$ {\small Dipartimento di Scienze e Innovazione Tecnologica}\\
{\small and INFN Gruppo collegato di Alessandria, }\\
{\small Universit{\`a} del Piemonte Orientale, }\\
{\small Viale T.~Michel~11,~15121~Alessandria,~Italy.}\\~
\hfill \\~
$^\text{b}$ {\small  Fachgruppe Mathematik,} \\
{\small Bergische~Universit\"at~Wuppertal, } \\
{\small Gau\ss stra\ss e~20,~42119~Wuppertal,~Germany. }\\~}

\begin{document}
\maketitle

\begin{abstract}
Given a Hopf algebra $H$,
we study modules and bimodules over  an algebra $A$ that carry an $H$-action,
as well as their morphisms and connections. Bimodules naturally 
arise when considering noncommutative analogues of tensor bundles.
For quasitriangular Hopf algebras and
bimodules with an extra quasi-commutativity property
we  induce connections on the tensor product over $A$ of two bimodules
from connections on the individual bimodules. 
This construction applies to arbitrary connections, i.e.~not
necessarily $H$-equivariant ones, and further 
extends to the tensor algebra generated by a bimodule and its dual.
Examples of these noncommutative structures arise in
deformation quantization via Drinfeld twists
of the commutative differential geometry of a smooth manifold,
where the Hopf algebra $H$ is the universal enveloping algebra of vector fields
{(or a finitely generated Hopf subalgebra)}.

We extend the Drinfeld twist
deformation theory of modules and algebras to
morphisms and connections that are not necessarily $H$-equivariant.
The theory canonically lifts to the tensor product structure. 
 \end{abstract}

\newpage

\tableofcontents


\section{Introduction}
Powerful methods for studying deformations of an algebra $A$  are available if this algebra 
carries a representation of a group (or 
Hopf algebra $H$).  In this case one can
first consider a deformation of the group in a quantum group (or of
the Hopf algebra $H$ in a deformed Hopf algebra), and then use the group
(or Hopf algebra) action in order to induce a deformation of the algebra $A$.

Noncommutative manifolds, i.e.~noncommutative deformations of the
algebra of functions on a manifold, are frequently constructed along
this line,  correspondingly quantum groups and Hopf algebras play a
fundamental role in this field.
Deformation via Drinfeld twists is an example \cite{Drinfeld83,Drinfeld89,GIAQUINTO}.
Here  a twist element $\FF\in H\otimes H$ of the Hopf algebra $H$
induces a new Hopf algebra $H^\FF$ and a deformed algebra $A_\st$. If the algebra $A$
is commutative and $H$ is co-commutative (like the universal
enveloping algebra of a Lie algebra), then the commutation relations in the twist deformed
algebra $A_\st$ are determined by the triangular $\RR$-matrix
$\RR=\FF^{-1}_{21}\FF$,  and $A_\st$ is an example of a
quasi-commutative algebra. 
The  algebra $A_\star$ is typically noncommutative, 
i.e.~it is a quantization of $A$.

This research area benefits from the interplay of different
approaches. Let us consider for example quantum groups:
They are studied as noncommutative algebras
of ``functions on a noncommutative group'' \cite{FRT}, 
but also as quasitriangular Hopf algebras \cite{Drinfeld85} together with 
their associated categories of representations \cite{Res89,Drinfeld89}.
They also originated and provided new methods in deformation
quantization (see \cite{Takhtadzhyan} for an introduction).

\subsubsection*{Noncommutative differential geometry}

Hopf algebra actions also play a central role in the study of the
differential geometry of noncommutative manifolds (the
differential  calculus on quantum groups \cite{Woronowicz}
is a leading example).

The algebraic structures underlying noncommutative differential
geometry are quite rich.  If $A$ is the noncommutative analogue of the
algebra of functions on a manifold $M$,
modules over $A$ are then the noncommutative  analogues of 
(modules of sections of) vector bundles over $M$. An analogue of the fibre-wise 
tensor product of vector bundles is achieved restricting to the
subclass of $A$-bimodules 
(compatible left and right  $A$-modules) and considering their tensor product over
$A$. A notable example of an $A$-bimodule is that of one-forms.
When the algebra $A$ carries a representation of a Hopf algebra $H$, 
we study $\MMMod{H}{A}{}$-modules, which are left $H$-modules and also left
$A$-modules in a compatible way. These are the noncommutative analogues of
vector bundles with a lift of the $H$-action from functions on the
base manifold to sections.
$A$-bimodules compatible with the $H$-action,
i.e.~$\MMMod{H}{A}{A}$-modules, form a tensor algebra over $A$.

Noncommutative differential geometry is the study  of maps between modules
and bimodules, like the exterior derivative, 
connections and their curvatures. In particular, 
$A$-linear maps (left or right) are
relevant because, as in the commutative case,  the curvature
of a connection is an $A$-linear map
and also the difference between two connections is an
$A$-linear map.  This latter property is the affine space structure of connections.

When all modules carry an $H$-action
it would seem natural to
consider also $H$-equivariant maps. On the contrary,
a main theme in this work is the  study of the
general structure of non $H$-equivariant homomorphisms,
connections and curvatures.
This case, for example, arises when one considers the Levi-Civita connection
of a Riemannian manifold and studies deformations of the manifold that
are not isometric (i.e.~when the Hopf algebra $H$ is not related to the Lie algebra
of Killing vector fields). More in general if  the
connection is a dynamical field, like in gauge and gravity theories, 
it is not equivariant under the $H$-action.
When homomorphisms are not $H$-equivariant then they are $H$-covariant,
i.e.~they transform under
the $H$-adjoint action. This is the canonical  lift to linear and (left or right) $A$-linear maps
of the $H$-action on the $\MMMod{H}{A}{A}$-modules.
Thus, 
linear and $A$-linear maps between $\MMMod{H}{A}{A}$-modules form also an
$H$-module and their deformation can be studied via deformation
of the Hopf algebra $H$.  
\sk
We study the Drinfeld twist deformation theory of modules
and algebras in this noncommutative differential geometric context. We develop in
particular a theory of connections and of their twist deformations.

Connections in noncommutative geometry have been introduced in the mid eighties \cite{Connes}
and then investigated further since the mid nineties \cite{Mourad:1994xa,DuViMasson,Madore:2000aq,DuViLecture}.
On right (or left) $A$-modules there is a well-established notion
of connection,  however, these connections present issues when
considered  on $A$-bimodules. In fact, two $A$-bimodules can be tensored into 
another $A$-bimodule but there is no corresponding operation on
connections such that connections on the individual $A$-bimodules  induce a
connection on the tensor product $A$-bimodule. 
One way out is to restrict to a subclass of connections that have
extra properties \cite{Mourad:1994xa,DuViMasson,Madore:2000aq,DuViLecture},  in particular their curvature turns out to be both 
left and right $A$-linear. 
An alternative route we advocate is to restrict to a subclass of
$A$-bimodules with extra properties, so that the usual connections on
right (or left) $A$-modules induce connections on tensor product modules.  

Our results on  the theory of connections on
$\MMMod{H}{A}{A}$-modules and their twist deformation
can be organized according to the extra properties we demand:
{\it  i\/}) We study the deformation of connections on right $A$-modules
and the dual theory on left $A$-modules. Connections form an affine 
space and deformation is an affine space isomorphism. It does not 
preserve flatness: flat connections are deformed in non flat ones and vice versa.
{\it ii\/}) The study of $H$-covariant (not necessarily
$H$-equivariant) homomorphisms on tensor products of
$\MMMod{H}{A}{A}$-modules requires a quasitriangular Hopf algebra
$H$. 
We first study a tensor product of linear maps compatible with
the $H$-action and then show that
there is a canonical way to twist deform this tensor product.
{\it iii\/}) If furthermore the algebra $A$ and the $A$-bimodules are
quasi-commutative,
i.e., if they are compatible with the braiding structure of the
quasitriangular Hopf algebra $H$, the  tensor product of linear
maps induces a tensor product 
over $A$ of right  $A$-linear maps and we 
develop a theory of connections on tensor product modules.
Arbitrary  connections on the individual $A$-bimodules  induce a
connection on the tensor product $A$-bimodule. There is also a canonical
extension of a connection on an  $A$-bimodule to the tensor algebra
generated by the $A$-bimodule and its dual. 
In the special case of $H$-equivariant connections we recover the
usual notion of bimodule connections \cite{Mourad:1994xa,DuViMasson,Madore:2000aq,DuViLecture}.
An early account of our results appeared in the PhD thesis \cite{SchenkelDiss} and the proceedings
articles \cite{SchenkelPro,AschieriPro}.

\sk

In the present work we have been led by the example of deformation quantization
of commutative manifolds. In this case $A=C^\infty(M)[[h]]$ (the
algebra of formal power series in $h$ with coefficients in
$C^\infty(M)$) and we canonically have the Lie algebra of derivations
of $A$ and the associated Hopf algebra $H=U\Xi[[h]]$,  where $U\Xi$ is
the universal  enveloping algebra of vector fields on $M$. 
Vector fields and  one-forms are canonically
$\MMMod{H}{A}{A}$-modules, the $H$-action
on these modules being via the Lie derivative.
The twist deformation of  these modules and of the Lie derivative and
inner derivative homomorphisms has been studied in
\cite{Aschieri:2005yw, Aschieri:2005zs}
in order to formulate a noncommutative gravity theory, see \cite{book}
for a pedagogical introduction. 
It is deforming arbitrary connections on the tensor algebra of vector
bundles over commutative manifolds
that we are led to the general theory of connections on
quasi-commutative bimodules presented in this paper. 
The deformation of commutative differential geometry in the more
general framework of cochain twists, leading also to a nonassociative
geometry, but considering only $H$-equivariant connections and
homomorphisms, has been studied in \cite{BeggsMajid1}.

In the wide class of examples obtained via twist  deformation of commutative differential
geometries  the Hopf algebra $H$ is always triangular.
A caveat is here in order: It can be that there are no nontrivial
examples of truly quasitriangular Hopf algebras 
acting on quasi-commutative algebras and bimodules. If this is the
case, then the theory of connections we present is only suitable for triangular Hopf algebras
and the proofs of the theorems in this paper have the advantage of singling
out the specific passages where triangularity (in the form of quasi-commutativity) is needed.

\subsubsection*{Categorical aspects of twist deformation}
The categorical aspects underlying Drinfeld twist theory emerge also
in the study of homomorphisms and connections. It is useful to formulate some of our findings in this language.

Let us consider the category $\rep_\eqv^H$ of
representations of the Hopf algebra $H$. 
The objects in $\rep_\eqv^H$ are $H$-modules, the morphisms
are $H$-equivariant maps between $H$-modules and their composition is the usual composition $\circ$.
 The tensor product of representations structures $\rep_\eqv^H$ as a
monoidal category,  $\big(\rep_\eqv^H,\otimes\big)$.
 Given a twist $\FF$ of the Hopf algebra $H$,  one can 
 deform $H$ into the Hopf algebra $H^\FF$
and consider  the corresponding  monoidal category
$\big(\rep_\eqv^{H^\FF},\otimes_\star\big)$ of representations of $H^\FF$.
The objects in $\rep_\eqv^{H^\FF}$ are $H^\FF$-modules, the morphisms
are $H^\FF$-equivariant maps between $H^\FF$-modules and their composition is the usual composition $\circ$.
It follows from Drinfeld's work \cite{Drinfeld89} that the two
categories $\big(\rep_\eqv^H,\otimes\big)$ and $\big(\rep_\eqv^{H^\FF},\otimes_\star\big)$
are equivalent as monoidal categories.
Following \cite{Drinfeld89}, Giaquinto and Zhang \cite{GIAQUINTO} studied
the twist deformation of the category  $\big(\MMMod{H}{A}{},{}_A\Hom^\eqv,\circ\big)$. 
In this category the objects are $\MMMod{H}{A}{}$-modules
(i.e.~modules that are both left $H$-modules and left
$A$-modules in a compatible way), the morphisms are $H$-equivariant and left $A$-linear maps
 and their composition is the usual composition $\circ$.
They proved that 
the categories   $\big(\MMMod{H}{A}{},{}_A\Hom^\eqv,\circ\big)$ and 
$\big(\MMMod{H^\FF}{A_\star}{},{}_{A_\star}\Hom^\eqv,\circ\big)$ are equivalent.

\sk
Led by the structures required in noncommutative differential
geometry we investigate 
the category $\big(\MMMod{H}{A}{},{}_A\Hom,\circ\big)$ where now
morphisms are not $H$-equivariant but just left $A$-linear maps. 
In this case the twist does not only deform the Hopf algebra $H$, the algebra $A$ and the
modules, but also the
morphisms. We show that this twist deformation gives an
equivalence of categories. However, this equivalence is not between 
the deformed category $\big(\MMMod{H^\FF}{A_\star}{},{}_{A_\star}\Hom,\circ\big)$ (with morphisms
left $A_\st$-linear maps) and the initial category
 $\big(\MMMod{H}{A}{},{}_A\Hom,\circ\big)$, but between $\big(\MMMod{H^\FF}{A_\star}{},{}_{A_\star}\Hom,\circ\big)$
   and the category  $\big(\MMMod{H}{A}{},{}_A\Hom,\circ_\star\big)$ obtained from  $\big(\MMMod{H}{A}{},{}_A\Hom,\circ\big)$ 
    by deforming just the composition law of morphisms.
If we restrict to $H$-equivariant morphisms
we recover the results of \cite{GIAQUINTO}.
A similar equivalence is found when we consider the category with objects $\MMMod{H}{}{A}$-modules 
(or $\MMMod{H}{A}{A}$-modules)
and morphisms  right $A$-linear maps.

For quasitriangular Hopf algebras $H$ we also study the category $\rep^H$.
The objects in $\rep^H$ are $H$-modules, the morphisms 
 linear maps (not necessarily $H$-equivariant) between $H$-modules 
 and their composition is the usual composition $\circ$. 
This category is an ``almost monoidal'' category because the
tensor product on morphisms that we consider (i.e.~the one compatible with
the lift of the $H$-action from the tensor product of modules to the tensor product of morphisms)
 spoils the bifunctor properties of the tensor product (it is a bifunctor up to braiding). 
We show that twist deformation is however compatible with this tensor structure.
Also the category with objects given by
quasi-commutative  $\MMMod{H}{A}{A}$-modules
and morphisms given by right $A$-linear maps (and usual composition $\circ$) is an ``almost monoidal''
category.
Here too we show that twist deformation is  compatible with this tensor structure.

\subsubsection*{Outline}
We clarify the structure of the paper by outlining its content.
In Section \ref{sec:prelim} we settle our notation and
we recall elementary Hopf algebra notions.
In Section \ref{sec:hopfalg} we first introduce Hopf algebra twists and review how
they induce deformations of algebras that are also $H$-modules 
(i.e.~$H$-module algebras). Then we study algebras that 
transform under an $H$-adjoint action, like the algebra of
linear maps of  an $H$-module.
The twist deformation of such an algebra leads to a deformed algebra that is isomorphic
to the original one and we begin with the detailed study of this isomorphism, 
that we denote by $D_\FF$.
\sk

Section \ref{sec:modhom} is devoted to the deformation of endomorphisms and
homomorphisms between modules that carry both an
$H$ and an $A$-module structure, i.e.~$\MMMod{H}{}{A}$-modules
(like e.g.~the module of one-forms on a manifold).
As already said, we do not restrict ourselves to $H$-equivariant
maps between these modules. We first  just consider linear maps
(this is propaedeutical and will be needed for the study of connections),
then  we consider $A$-linear ones.
Due to the isomorphism $D_\FF$ between  the deformed  algebra of endomorphisms of
a module  and the algebra of endomorphisms of
the deformed module, there is a canonical way to deform right
$A$-linear endomorphisms (right $A$-linear maps) on $\MMMod{H}{}{A}$-modules
into right $A_\star$-linear ones on $\MMMod{H^\FF}{}{A_\star}$-modules. 
Similarly, also the deformation of homomorphisms between two modules is canonical.
In categorical language we have constructed a deformation
functor that maps $\MMMod{H}{}{A}$-modules to  $\MMMod{H^\FF}{}{A_\star}$-modules,
 and that has a nontrivial action, given by $D_\FF$,
on the corresponding right $A$-module homomorphisms. This functor
implies that the category of  $\MMMod{H}{}{A}$-modules with right $A$-linear maps as morphisms and
with a twist deformed composition law is equivalent to that of
$\MMMod{H^\FF}{}{A_\star}$-modules with right $A_\star$-linear maps as morphisms and with the
 usual composition law.  

In this section we also consider the deformation of $\MMMod{H}{A}{}$-modules
and of left $A$-linear homomorphisms (left $A$-linear maps). 
The dual $V^\prime$ of a right $A$-module $V$ is a left $A$-module,
and we show that dualizing a deformed module is
equivalent to deforming the dual module.
This result will be needed later in order to study connections on dual modules and their deformation.
\sk

In Section \ref{sec:tenprod} we study tensor products of modules and of homomorphisms, as well as their
deformation. A tensor product of linear maps is a linear map on the
tensor product of the original modules. 
We require $H$-covariance of this construction, i.e.~that it transforms
according to the  $H$-adjoint action, like the original linear maps.
This is achieved  if the Hopf algebra is quasitriangular, in fact it is lifting the braiding
of modules to linear maps  (via an adjoint action) that 
the tensor product of linear maps is obtained. 
Twist deformation of homomorphisms is compatible with this tensor product,
 hence again there is a canonical way to deform tensor products
of homomorphisms. Otherwise stated, the deformation functor is
compatible with the tensor product structure. We would actually have
an equivalence of monoidal categories if  the tensor product between linear maps would
structure the category of $H$-modules as a monoidal category. This is, however, only almost the case,
since the tensor product is a bifunctor up to a braiding.

In the second part of this section we finally consider left $H$-module
$A$-bimodules, i.e.~$\MMMod{H}{A}{A}$-modules. We focus on the subclass of quasi-commutative ones,
i.e.~of those $A$-bimodules compatible with the braiding structure of the
quasitriangular Hopf algebra $H$. Correspondingly right $A$-linear
homomorphisms (right $A$-linear maps) inherit a braided left $A$-linearity property.
The tensor product structure we have studied before induces a tensor product structure over $A$.
In particular, the tensor product of two right $A$-linear maps is again a right $A$-linear map
on the tensor product module (over $A$).
The deformation is also induced canonically on  tensor
products of quasi-commutative $A$-bimodules and of their right $A$-linear
homomorphisms. Otherwise stated, the deformation functor is compatible with the
tensor product structure over $A$.

We conclude the section by recalling that the universal $\RR$-matrix of a 
quasitriangular Hopf algebra $H$ is an
example of twist of $H$. If we use it in order to twist
deform quasi-commutative $A$-bimodules and their right $A$-linear homomorphisms,
 we obtain an isomorphism between left and right $A$-linear homomorphisms.

\sk
In Section \ref{sec:connections} we consider a differential calculus over the algebra $A$
(a differential graded algebra over $A$) and connections on right $A$-modules.
These are in particular  linear maps between modules, and carry an affine space structure with
respect to right $A$-module homomorphisms. 
We then deform the differential calculus and the connections. 
Using the results of the previous sections we obtain an affine space
isomorphism between connections on right $A$-modules and deformed
connections on right $A_\st$-modules. 
When the Hopf algebra is quasitriangular, and we restrict ourselves to
quasi-commutative $A$-bimodules, arbitrary connections on $A$-bimodules can be summed
to give a connection on the tensor product module (over $A$) of the initial
$A$-bimodules. This operation is again compatible with twist deformation,
in fact we show that the sum of deformed connections equals the
deformation of the sum of connections.

Finally we study connections on dual modules. A connection on a
right $A$-module induces a  connection on the dual left $A$-module (and there is an
affine space isomorphism between the affine spaces of
connections on a right $A$-module  and of
connections on the dual left $A$-module).
Again deformation and  dualization are compatible. 
The dual construction can be applied to connections with right
Leibniz rule on quasi-commutative $A$-bimodules to obtain connections 
with left Leibniz rule on the dual $A$-bimodules.
These can be mapped into right Leibniz rule  connections by extending the left to
right isomorphism of homomorphisms that we have constructed with the
$\RR$-matrix in the end of Section \ref{sec:tenprod}. In this way, from a
connection on a quasi-commutative $A$-bimodule we induce a connection
on the tensor algebra (over $A$) generated by this $A$-bimodule and its dual.
\sk

In Section \ref{sec:curvature} we deform the curvature of a connection and notice that it is
different from the curvature of the deformed connection.
In particular, flat connections are not deformed into flat connections.
Similarly we calculate the curvature of the sum of connections  and find that it 
 differs in general from the sum of the curvatures of the original connections.


\section{\label{sec:prelim}Preliminaries and notation on modules, algebras and Hopf algebras}
In this section we fix the notation and recall some basic facts about Hopf algebras and their
modules.  In deformation quantization the field of complex numbers
$\bbC$ is replaced by the ring $\bbC[[h]]$ of formal power series (in
an indeterminate, say $h$) with
coefficients in $\bbC$.  In order to cover also this example we
shall consider modules and algebras over a commutative ring $\bbK$ with unit element $1\in\bbK$.

A {\bf $\bbK$-module} is an abelian
group $V$ together with a map $\bbK\times V\rightarrow V$, $(\la, v)\mapsto \la v$,
 such that for all $\la, \tilde\la\in \bbK$ and $v,\tilde v\in V$,  
\eq
(\tilde\la \la) v=\tilde\la(\la v)~,~~\la(v+\tilde v)=\la v+\la \tilde v~,~~
(\la+\tilde \la)v=\la v+\tilde\la v~,~~ 1v=v~.
\en
A {\bf $\bbK$-module homomorphism} (or {\bf$\bbK$-linear map}) between two
$\bbK$-modules $V$ and $W$ is a homomorphism $P: V\rightarrow W$ of abelian groups 
 that satisfies, for all $v\in V$ and $\la\in \bbK$,  $P(\la
 v)=\la P(v)$.
The $\bbK$-module of  all $\bbK$-linear maps between $V$ and $W$ is
denoted $\Hom_\bbK(V,W)$.

An {\bf algebra} is a $\bbK$-module $\AA$ together with a 
$\bbK$-linear map $\mu:\AA\otimes\AA\to\AA$ (product), where $\otimes$ is the tensor product over $\bbK$.
We denote by $a\otimes b$ the image of $(a,b)$ under the canonical
$\bbK$-bilinear map $\AA\times\AA \to\AA\otimes\AA$ 
and write for the product $\mu(a\otimes b) = a\,b$.
The algebra $\AA$ is called {\bf associative} if,
for all $a,b,c\in\AA$,
$(a\, b)\, c = a\, (b\, c)$.
It is called {\bf unital} if
there exists a unit element $1\in\AA$ satisfying $1\,a=a\,1=a$, for all $a\in \AA$.
An {\bf algebra homomorphism} between two algebras $A$ and $B$
is a $\bbK$-linear map $\varphi: A\to B$, such that for all $a,\tilde a\in A$,
$\varphi(a\,\tilde a) = \varphi(a)\,\varphi(\tilde a)$. If $A$ and $B$ are unital, then 
$\varphi$ is also required to preserve the unit, i.e.~$\varphi(1) = 1$.
In the following, algebras will always be associative and unital if not otherwise
stated.
\begin{Definition}
A {\bf Hopf algebra} is an algebra $\UU$ together with 
two algebra homomorphisms $\Delta:\UU\to\UU\otimes\UU$ (coproduct), $\varepsilon:\UU\to\bbK$ (counit) and a
$\bbK$-linear map $S:\UU\to\UU$ (antipode) satisfying, for all $\xi\in\UU$,
\begin{subequations}
\label{eqn:Hopfalgebraaxioms}
\begin{flalign}
(\Delta\otimes\id)\Delta(\xi)&= (\id\otimes\Delta)\Delta(\xi)
~,\quad\text{(coassociativity)}\\
(\varepsilon\otimes\id)\Delta(\xi)&=(\id\otimes \varepsilon)\Delta(\xi) = \xi~,\\
\mu\bigl((S\otimes\id)\Delta(\xi)\bigr) &= \mu\bigl((\id\otimes S)\Delta(\xi)\bigr)=\varepsilon(\xi) 1~.
\end{flalign}
\end{subequations}
 The product in the algebra $\UU\otimes\UU$  is defined by, for all $\xi,\zeta,\tilde \xi,\tilde\zeta\in\UU$,
\begin{flalign}
(\xi\otimes \zeta)\, (\tilde\xi\otimes\tilde\zeta)=(\xi\, \tilde\xi)\otimes(\zeta\, \tilde\zeta)~.
\end{flalign}
\end{Definition} 
\sk

It is useful to introduce a compact notation  (Sweedler's notation) for the coproduct, for all $\xi\in\UU$,
$\Delta(\xi)=\xi_{1}\otimes\xi_{2}$ (sum understood). The Hopf algebra
properties (\ref{eqn:Hopfalgebraaxioms}) in this notation read
\begin{subequations}
\begin{flalign}
(\xi_{1})_1\otimes(\xi_1)_2\otimes \xi_{2}&=
\xi_{1}\otimes(\xi_2)_1\otimes (\xi_2)_2
=:\xi_1\otimes\xi_2\otimes\xi_3
~,\\
\varepsilon(\xi_{1})\xi_{2}&=\xi_{1}\varepsilon(\xi_{2})=\xi~,\\
S(\xi_1)\, \xi_2 &= \xi_1\, S(\xi_2)=\varepsilon(\xi) 1~.\label{2.6}
\end{flalign}
\end{subequations}
Likewise we denote the three times iterated application of the
coproduct on $\xi$  by $\xi_1\otimes\xi_2\otimes\xi_3\otimes \xi_4$.
It can be shown that the antipode of a Hopf algebra is unique and satisfies
$S(\xi\, \zeta)=S(\zeta)\, S(\xi)$ (antimultiplicative property), $S(1)=1$, $S(\xi_1)\otimes S(\xi_2)=S(\xi)_2\otimes S(\xi)_1$ 
and $\varepsilon(S(\xi))=\varepsilon(\xi)$, for all $\xi,\zeta\in\UU$.
\begin{Definition}
Let $A$ be an algebra. A {\bf left $A$-module} (or {\bf $\MMMod{}{A}{}$-module}) is a $\bbK$-module $V$ together with a 
$\bbK$-linear map $\cdot:\AA\otimes V\to V$ satisfying, for all $a,b\in\AA$ and $v\in V$,
\begin{flalign}
(a\, b) \cdot v = a\cdot(b\cdot v)~,~~1\cdot v = v~.
\end{flalign} 
The map $\cdot:\AA\otimes V\to V$ is called an action of
$A$ on $V$ or a representation of $A$ on $V$. 

A  $\bbK$-linear map $P : V\to W$ between two $\MMMod{}{A}{}$-modules
$V$ and $W$ is an  {\bf $\MMMod{}{A}{}$-module homomorphism} (or {\bf left $A$-linear
map}) if, for all $a\in A$ and $v\in V$,
$P(a\cdot v) = a\cdot P(v)$. We denote the $\bbK$-module of all left $A$-linear
maps between $V$ and $W$ by ${_A}\Hom(V,W)$.

Similarly a {\bf right $\AA$-module} (or {\bf $\MMMod{}{}{A}$-module}) is a $\bbK$-module $V$ together with a 
$\bbK$-linear map $\cdot:V\otimes\AA\to V$ 
satisfying, for all $a,b\in \AA$ and $v\in V$,
\begin{flalign}
 v\cdot(a\,b)=(v\cdot a)\cdot b~,~~v\cdot 1=v~.
\end{flalign}

A  $\bbK$-linear map $P : V\to W$ between two $\MMMod{}{}{A}$-modules
$V$ and $W$ is an  {\bf $\MMMod{}{}{A}$-module homomorphism} (or {\bf right $A$-linear
map}) if, for all $a\in A$ and $v\in V$,
$P(v\cdot a) = P(v)\cdot a$. We denote the $\bbK$-module of all right $A$-linear
maps between $V$ and $W$ by $\Hom_A(V,W)$. 
\end{Definition}
\sk

A left and a right module structure on $V$ are compatible if left and
right actions commute.
\begin{Definition}
Let $A$ and $B$ be algebras. An {\bf $(A,B)$-bimodule} (or {\bf $\MMMod{}{A}{B}$-module}) 
is a left $A$-module and a right $B$-module $V$
satisfying the compatibility condition,
for all $a\in A$, $b\in B$ and $v\in V$,
\begin{flalign}
(a\cdot v)\cdot b &=a\cdot (v\cdot b)~.
\end{flalign}
In case of $B=A$ we call $V$ an {\bf $A$-bimodule}  (or {\bf $\MMMod{}{A}{A}$-module}).

A  $\bbK$-linear map $P : V\to W$ between two $\MMMod{}{A}{B}$-modules
$V$ and $W$ is an  {\bf $\MMMod{}{A}{B}$-module homomorphism} if it is left $A$-linear and right $B$-linear.
\end{Definition}
\sk

The algebra $A$ can itself be a module over another algebra $H$. If
$H$ is further a Hopf algebra we have the notion of an $H$-module
algebra, expressing covariance of $A$ under $H$.
\begin{Definition}
 Let $\UU$ be a Hopf algebra. A {\bf left $\UU$-module algebra} (or {\bf $\AAAlg{H}{}{}$-algebra})
  is an algebra $\AA$ which is also a left $\UU$-module (we denote
the $H$-action by $\trgl$),
such that for all $\xi\in\UU$ and $a,b\in\AA$,
\begin{flalign}
\xi \trgl(a\,b)=(\xi_1\trgl a)\,(\xi_2\trgl b)~,~~\xi\trgl1 =\varepsilon(\xi)\,1~.
\label{coproductonab}\end{flalign}

An {\bf $\AAAlg{H}{}{}$-algebra homomorphism} between two $\AAAlg{H}{}{}$-algebras $A$ and $B$ is
an algebra homomorphism $\varphi: A\rightarrow  B$ that intertwines 
between the left action of $H$ on $A$ and the left action 
of $H$ on $B$,
for all $\xi\in H,\, a\in A$, $\varphi(\xi\trgl a)=\xi \trgl \varphi(a)$.
This property is also called {\bf $H$-equivariance} of the algebra homomorphism $\varphi$. 
\end{Definition}
\sk

We can now consider $\MMMod{}{A}{B}$-modules $V$, where $A,B$ are $\AAAlg{H}{}{}$-algebras
 and $V$ is also a left $H$-module. Compatibility between the Hopf algebra structure
of $H$ and the $(A,B)$-bimodule structure of $V$  leads to the following
covariance requirement.
\begin{Definition}
Let $H$ be a Hopf algebra and $A,B$ be $\AAAlg{H}{}{}$-algebras.
A {\bf left $\UU$-module $(\AA,B)$-bimodule} (or {\bf $\MMMod{H}{A}{B}$-module}) is an $\MMMod{}{A}{B}$-module $V$ 
which is also a left $\UU$-module, such that
for all $\xi\in\UU$, $a\in A$, $b\in B$ and $v\in V$,
\begin{subequations}
\label{eqn:moduleHcovariance}
\begin{flalign}
\label{eqn:moduleHcovariance2} \xi\trgl(a\cdot v) = (\xi_1\trgl a)\cdot (\xi_2\trgl v)~,\\
\label{eqn:moduleHcovariance1} \xi\trgl(v\cdot b) = (\xi_1\trgl v)\cdot (\xi_2\trgl b)~.
\end{flalign}
\end{subequations}
In case of $B=A$ we say that $V$ is a {\bf left $H$-module $A$-bimodule} (or {\bf $\MMMod{H}{A}{A}$-module}).

An algebra $E$ is a {\bf left  $\UU$-module $(\AA,B)$-bimodule algebra} (or {\bf $\AAAlg{H}{A}{B}$-algebra}), if $E$ as a
module is an $\MMMod{H}{A}{B}$-module and if $E$ is also an $\AAAlg{H}{}{}$-algebra.
\end{Definition}
$\MMMod{H}{A}{}$-modules and $\MMMod{H}{}{B}$-modules are defined similarly
to $\MMMod{H}{A}{B}$-modules, where (\ref{eqn:moduleHcovariance}) is restricted to
(\ref{eqn:moduleHcovariance2}) or (\ref{eqn:moduleHcovariance1}), respectively.  
For a coherent notation we shall call left $H$-modules (where $H$ is a Hopf algebra) 
also {\bf $\MMMod{H}{}{}$-modules}.

\sk

We can consider different classes of maps between $\MMMod{H}{A}{B}$-modules.
The first option is to consider maps that are compatible with all module structures,  
i.e.~$\MMMod{H}{A}{B}$-module homomorphisms.
 Let $V,W$ be $\MMMod{H}{A}{B}$-modules, then  a map $P:V\to W$ is an $\MMMod{H}{A}{B}$-module homomorphisms
  if it is an $H$-equivariant map (or $\MMMod{H}{}{}$-module homomorphism), for
  all $\xi\in H$ and $v\in V$,
$P(\xi\ra v) = \xi\ra P(v)$, and if it is also a  left $A$-linear map
and a right $B$-linear map.
A second option, as motivated  in the introduction, is to consider maps
 that are only compatible with the left $A$-module structure or the
 right $B$-module structure, i.e.~left $A$-linear maps and
 right $B$-linear maps.
A third option is to consider just $\bbK$-linear maps.

\begin{Example}
Consider the universal enveloping algebra
$U\Xi$ associated with the Lie algebra of vector fields $\Xi$ on a
smooth manifold $M$.
This is the tensor algebra (over $\bbR$) generated by the elements of
$\Xi$ and the unit element $1$ modulo the left and right ideal
generated by the elements $uv-vu-[u,v]$, for all $u,v\in \Xi$.
$U\Xi$ has a natural Hopf algebra structure; on the generators $u\in\Xi$
and the unit element $1$ we define
\begin{flalign}
\label{cosclass}
 \nonumber \qquad\qquad \qquad\qquad \Delta(u) &= u\otimes 1 + 1\otimes u~,  &\Delta(1) &= 1\otimes 1~,\qquad \qquad\qquad\qquad\qquad\\
 \varepsilon(u) &= 0~,  &\varepsilon(1) &= 1~,\\
 \nonumber S(u) &= -u~,  &S(1) &= 1~,
\end{flalign}
and extend $\Delta$ and  $\varepsilon$ 
as algebra homomorphisms and $S$ as an antialgebra homomorphism to all $U\Xi$.

Let $V$ be the vector space of one-forms $\Omega$ (vector fields $\Xi$) on 
$M$, and $A=C^\infty (M)$ be the algebra of smooth functions on $M$.   
$\Omega$ ($\,\Xi$) is a $\MMMod{}{C^\infty(M)}{C^\infty(M)}$-module (the right module
structure equals the left module structure because $C^\infty(M)$ is a
commutative algebra). 

$C^\infty(M)$ is a $\AAAlg{U\Xi}{}{}$-algebra; the first of property ({\ref{coproductonab}})
for $\xi$ a vector field is just the Leibniz rule.  
Employing the Lie derivative, we also have that  $\Omega$ ($\,\Xi$) is a
$\MMMod{U\Xi}{C^\infty(M)}{C^\infty(M)}$-module.

Another example is obtained by considering smooth complex valued
functions on $M$, and similarly vector fields and one-forms over the
field $\bbC$ rather than $\bbR$. 
\end{Example}


\section{\label{sec:hopfalg}Hopf algebra twists and deformations}
We first review some well-known results on deformations of a Hopf algebra 
and its modules by Drinfeld twists. We then introduce a deformation isomorphism
which will be of utmost importance in the study of deformations of module homomorphisms and connections.


\subsection{Twist deformation preliminaries}
\begin{Definition}\label{twistdeff}
Let $H$ be a Hopf algebra. A {\bf twist $\FF$} is an element 
$\FF\in \UU\otimes \UU$ that is invertible and that satisfies
\begin{subequations}
\label{propF}
\begin{flalign}
\label{propF1}
\FF_{12}(\Delta\otimes \id)\FF &=\FF_{23}(\id\otimes \Delta)\FF~,\quad
\text{($_{\,}2$-cocycle property)}\\
\label{propF2}
(\varepsilon\otimes \id)\FF &=1=(\id\otimes \varepsilon)\FF~,\quad 
\text{(normalization property)}
\end{flalign}
\end{subequations}
where $\FF_{12}=\FF\otimes 1$ and $\FF_{23}=1\otimes \FF$.
\end{Definition}
\sk

We shall frequently use the notation (sum over $\al$ understood)
\eq\label{Fff}
\FF=\f^\al\otimes\f_\al~~~,~~~~\FF^{-1}=\of^\al\otimes\of_\al~,
\en
where $ \f^\alpha, \f_\alpha, \of^\al,\of_\al$ are elements in  $\UU$. 

In order to get familiar with this notation we rewrite
(\ref{propF1}), (\ref{propF2}) and the inverse of (\ref{propF1}), 
\eq
((\Delta\otimes\id)\FF^{-1}) 
\FF^{-1}_{12} =((\id \otimes \Delta)\FF^{-1})\FF^{-1}_{23}~,\label{Finverse}
\en using the notation (\ref{Fff}). Explicitly,
\begin{subequations}\label{eqn:twistpropsimp}
\begin{flalign}
\f^\beta \f^\alpha_{_1}\otimes \f_\beta \f^\alpha_{_2}\otimes \f_\alpha &=
\f^\alpha\otimes \f^\beta \f_{\alpha_1}\otimes \f_\beta
\f_{\alpha_2}~ ,\label{2.21}\\
\varepsilon(\f^\alpha)\f_\alpha &= 1 =  \f^\alpha\varepsilon(\f_\alpha) ~, \label{2.23}\\
\label{ass}
\of_{_1}^\al\of^\be\otimes \of_{_2}^\al\of_\be\otimes \of_\al&=
\of^\al\otimes {\of_{\al_1}}\of^\be\otimes {\of_{\al_2}}\of_\be~.
\end{flalign}
 \end{subequations}
We next recall how
a twist $\FF$ induces a 
deformation of the Hopf 
algebra $\UU$ into a Hopf algebra  $\UU^\FF$, and of all its
$\MMMod{H}{}{}$-modules
into 
$\MMMod{H^\FF}{}{}$-modules. In particular  
$\AAAlg{H}{}{}$-algebras are deformed into  
$\AAAlg{H^\FF}{}{}$-algebras, and commutative ones are typically deformed
into noncommutative ones. In this respect  $\FF$ induces a quantization.

\begin{Theorem}\label{TwistedHopfAlg}
The twist ${\cal F}$ of the Hopf algebra $\UU$ determines a new 
 Hopf algebra $\UU^\FF$,  given by
\begin{flalign}
(\UU ,\mu, \Delta^{\cal F}, S^{\cal F}, \varepsilon)~.
\end{flalign}
As algebras $\UU^\FF =\UU$ and they also have the same counit 
$\varepsilon^\FF=\varepsi$. 
The coproduct is, for all $\xi \in \UU$,
\begin{flalign}
\Delta^{\FF}(\xi) = {\FF}\Delta(\xi){\FF}^{-1}~ .\label{2.4.1}
\end{flalign}
The antipode is, for all $\xi\in H$,
\begin{flalign}
S^\FF(\xi)=\chi S(\xi)\chi^{-1}~,\label{2.4.3}
\end{flalign}
where
\begin{flalign}
\chi := \f^\al S(\f_\al) \quad,\quad \chi^{-1} = S(\of^\al) \of_\al~ .\label{2.4.2}
\end{flalign}
\end{Theorem}
A proof of this theorem can be found in textbooks on Hopf algebras,
see e.g.~\cite{Majid:1996kd}, Theorem 2.3.4.

\begin{Remark}\label{rem:dequant}
It is easy to show that the Hopf algebra $\UU^\FF$ admits the twist $\FF^{-1}$, indeed
\eq\FF_{12}^{-1}(\Delta^\FF\otimes \id)\FF^{-1}=\FF_{23}^{-1}(\id\otimes \Delta^\FF)\FF^{-1}
\en
is equivalent to (\ref{Finverse}). From (\ref{2.4.1}), (\ref{2.4.3})
and (\ref{2.4.2}) we see
that  the Hopf algebra
$(H^\FF)^{\FF^{-1}}$ is canonically isomorphic to $H$. 
We say that we twist back  $H^\FF$ to $H$ via the
twist $\FF^{-1}$.
\end{Remark}
\sk

\begin{Theorem}\label{Theorem1}
Given a Hopf algebra $H$, 
a twist $\FF\in H\otimes H$ and  an $\AAAlg{H}{}{}$-algebra  $\AA$ (not
necessarily associative or with unit), then there exists
an  $\AAAlg{H^\FF}{}{}$-algebra $\AA_\st$. 
The algebra $\AA_\st$ has the same $\bbK$-module
structure as $\AA$ and
the action of $H^\FF$ on $\AA_\st$
is that of $H$ on $\AA$.
The product in $\AA_\st $ is defined by, for all $a,b\in A$,
\eq
a\st  b :=\mu\circ \FF^{-1}\trgl(a\otimes b)=(\of^\al\trgl
a)\,(\of_\al\trgl b)~.
\en
If $\AA$ has a unit element then $\AAs$ has the same unit element. If 
$\AA$ is associative then $\AA_\st $ is an associative algebra as well.
\end{Theorem}
\begin{proof}
We have to prove  that the product in $\AA_\st$ 
is compatible with the Hopf algebra structure on $H^\FF$, for all $a,b\in A$ and $\xi\in H$,
\eqa
\xi\trgl(a\st b)&=&\xi\trgl\bigl(\mu\circ \FF^{-1}\trgl (a\otimes b)\bigr)\nn\\
&=&
\mu\circ \D(\xi)\trgl \circ \,\FF^{-1}\trgl (a\otimes b)\nn\\
&=&
\mu\circ(\D(\xi)\, \FF^{-1})\trgl  (a\otimes b)\nn\\
&=&\mu\circ\FF^{-1}\trgl\circ\D^\FF(\xi)\trgl(a\otimes b)\nn\\
&=&(\xi_{1_\FF}\trgl a)\st(\xi_{2_\FF}\trgl b)~,
\ena
where we used the notation $\Delta^\FF(\xi)=\xi_{1_\FF}\otimes \xi_{2_\FF}$.

If $\AA$ has a unit element $1$, then  $1\st a=a\st1=a$ follows from the
normalization property (\ref{propF2}) of the twist $\FF$.
If $\AA$ is an associative algebra we also have to prove
associativity of the new product, for all $a,b,c\in A$,
\begin{flalign}
\nn (a\st b)\st c&=\of^\al\trgl \bigl((\of^\be\trgl a)(\of_\be\trgl
b)\bigr)\;(\of_\al\trgl c)\\
&=
(\of^\al\trgl a)\,\of_\al\trgl \bigl((\of^\be\trgl b)(\of_\be\trgl
c)\bigr)=a\st (b\st c)~,
\end{flalign}
where we used the twist cocycle property  (\ref{propF1})  in
the notation  adopted in (\ref{ass}).
\end{proof}

\begin{Theorem}\label{Theorem2}
In the hypotheses of Theorem \ref{Theorem1}, given another $\AAAlg{H}{}{}$-algebra $B$ and an $\MMMod{H}{A}{B}$-module
 $V$, then there exists
an $\MMMod{H^\FF}{A_\star}{B_\star}$-module $V_\st$.
The module $V_\st$ has the same $\bbK$-module
structure as $V$ and the left action of $H^\FF$ on $V_\st$
is that of $H$ on $V$. The $A_\st$ and $B_\st$ action on $V_\st$
are respectively defined by, for all $a\in A$, $b\in B$ and $v\in V$,
\begin{subequations}
\label{modst}
\begin{flalign}
a\st  v &=\cdot\circ \FF^{-1}\trgl(a\otimes v)=(\of^\al\trgl
a)\cdot (\of_\al\trgl v)~, \label{lmodst}\\
v\st  b &=\cdot\circ \FF^{-1}\trgl(v\otimes b)=(\of^\al\trgl
v)\cdot(\of_\al\trgl b)~. \label{rmodst}
\end{flalign}
 \end{subequations}
If $V=E$ is further an $\AAAlg{H}{A}{B}$-algebra, then $E_\st$ is an $\AAAlg{H^\FF}{A_\st}{B_\st}$-algebra,
where the product in the algebra $E_\st$ is given in Theorem \ref{Theorem1}.
\end{Theorem}
\begin{proof}
We give a sketch of the proof.
Left $\AA_\st$-module property: 
\begin{flalign}
\nn (a\st b)\st v&=\of^\al\trgl \bigl((\of^\be\trgl a)(\of_\be\trgl
b)\bigr)\cdot(\of_\al\trgl v)=
(\of^\al\trgl a)\cdot \of_\al\trgl \bigl((\of^\be\trgl b)\cdot (\of_\be\trgl
v)\bigr)\\
\label{astbstv}&=a\st (b\st v)~.
\end{flalign}
The right $B_\st$-module and $(A_\st,B_\st)$-bimodule
properties are similarly proven.\newline
Compatibility between the  left $H^\FF$ and the left $A_\st$-action:
\begin{flalign}
\xi\trgl (a\st v)
&=\cdot\circ\bigl(\Delta(\xi) \FF^{-1}\bigr)\trgl (a\otimes v)=\cdot\circ\FF^{-1}\ra\circ\Delta^\FF (\xi)\trgl (a\otimes v)\nn \\
&=(\xi_{1_\FF}\trgl a)\st(\xi_{2_\FF}\trgl v)\,.
\end{flalign}
Compatibility between the  left $H^\FF$ and the right $B_\st$-action
is similarly shown.
In case $V=E$ is an $\AAAlg{H}{A}{B}$-algebra, then $E_\st$ is an $\AAAlg{H^\FF}{A_\st}{B_\st}$-algebra
 because of Theorem \ref{Theorem1}.
\end{proof}
As in  Theorem \ref{Theorem2} we can deform $\MMMod{H}{A}{}$-modules and $\MMMod{H}{}{B}$ -modules into
$\MMMod{H^\FF}{A_\st}{}$-modules and $\MMMod{H^\FF}{}{B_\st}$-modules by restricting 
(\ref{modst}) to (\ref{lmodst}) or (\ref{rmodst}), respectively. We can also (trivially) deform
$\MMMod{H}{}{}$-modules into $\MMMod{H^\FF}{}{}$-modules.

\begin{Example}
Consider $U\gg[[h]]$, the universal 
enveloping algebra over $\bbC[[h]]$ of 
a Lie algebra $\gg$.  Twists $\FF\in U\gg\otimes U\gg_{\,}[[h]]$ are 
(up to equivalence) in one to one correspondence with
skew-symmetric elements $r\in\gg\otimes \gg$ satisfying the classical
Yang-Baxter equation \cite{Drinfeld83}.

Let $\FF$ be a twist of (a Hopf subalgebra $U\gg[[h]]$ of) $U\Xi[[h]]$, the universal
enveloping algebra of vector fields on a smooth
manifold $M$. Then the $\AAAlg{U\Xi[[h]]}{}{}$-algebra
$A=C^\infty(M)[[h]]$ of smooth function over $M$ with values in
$\bbC[[h]]$ and the 
$\MMMod{U\Xi[[h]]}{A}{A}$-modules of one-forms
$\Omega$ and of vector fields $\Xi$ can be deformed into the
noncommutative ones $A_\star$, $\Omega_\star$ and $\Xi_\star$. Also
the $\AAAlg{U\Xi[[h]]}{}{}$-algebras of tensor fields
$(T,\otimes)$, of exterior forms $(\Omega^\bullet,\wedge)$ and the
Lie algebra of vector fields $(\Xi, [~,~])$ can be deformed into the noncommutative algebras
$(T, \otimes_\star)$,  $(\Omega^\bullet, \wedge_\star)$
and the quantum Lie algebra $(\Xi, [~,~]_\star)$. These
deformations lead to a noncommutative gravity theory \cite{Aschieri:2005zs}.
\end{Example}

\begin{Remark}
The examples presented are in the context of formal deformation
quantization. However, if one considers sufficiently 
regular actions of a group $G$, rather than actions of its Lie algebra
$\gg$, then abelian Drinfeld twists (i.e.~Drinfeld twists associated to
an abelian Lie algebra $\gg$) induce  $\star$-products that can be
implemented non-formally and produce deformations $A_\star$ of
$C^\ast$-algebras \cite{Rieffel}. This construction has further  been
generalized to  the case of  nonabelian Drinfeld twists associated to
solvable Lie algebras $\gg$ (with extra structure) \cite{Bieliavsky}. 

Disregarding the topological aspects, the (non-formal) 
noncommutative algebras of polynomial functions on
multiparametric quantum groups associated with (abelian) Drinfeld twists
were introduced in \cite{Reshetikhin}.  Their differential geometry and that
of their homogeneous spaces was studied in \cite{AC} (and references therein).

Non-formal noncommutative geometries {\`a} la Connes related to (abelian) Drinfeld twists
are, besides the noncommutative torus, the noncommutative spheres
\cite{Connes-Landi} and further noncommutative manifolds (so-called isospectral deformations)
considered in \cite{Connes-Landi, Dubois-Violette}.
This was clarified in \cite{Varilly, Sitarz, Dubois-Violette, AB}. 
\end{Remark}

\subsection{The deformation isomorphism $D_\FF$}
Given an algebra $\bbA$ and an algebra homomorphism $\rho : H\rightarrow \bbA$ we can
consider the adjoint action on the algebra $\bbA$, for
all $\xi\in H$ and $P\in \bbA$, $\xi\RA P:=\rho(\xi_1)P\rho(S(\xi_2))$. 
We denote adjoint actions by the symbol $\RA$
in order to stress that these actions are induced  from 
``fundamental''  actions (in this case the left and right actions
$\rho(\xi)P$ and $P\rho(\xi)$ of the algebra $H$ on $\bbA$).
In this subsection we show that twist deformation of the algebra
$\bbA$ according to Theorem \ref{Theorem1} and using the $H$-adjoint
action $\RA$,  gives an algebra  $\bbA_\star$ that  is  isomorphic to $\bbA$.
This feature has been observed for the special case of the Hopf algebra $H$  in 
\cite{Gurevich:1991nr} and it has been further exploited in \cite{Aschieri:2005zs} 
 (see also \cite{book} Section 8.2.3.1). It has been considered in the more general case of
$H$-adjoint actions induced by an
algebra homomorphism $\rho : H\rightarrow \bbA$ in
\cite{Fiore:2008sj} and \cite{Kulish:2010mr}.

In Theorem \ref{isostar} and Theorem \ref{isostar2} we consider
the algebra isomorphism $D_\FF: \bbA_\star \to \bbA$
under slightly more general assumptions than the existence 
of an algebra homomorphism $\rho :H\rightarrow \bbA$. 
In Theorem \ref{isoAAstDAA} we clarify that the algebra isomorphism 
$D_\FF: \bbA_\star \to \bbA$ preserves the 
$\AAAlg{H^\FF}{}{}$-algebra structures of  $\bbA_\star$ and of $\bbA$.
We also show that $D_\FF$ has a categorical interpretation.

\begin{Theorem}\label{isostar}
Let $\bbA$ be an $\AAAlg{H}{}{}$-algebra (not necessarily  associative
or with unit; the $H$-action is denoted by $\RA$)
and  also a right module with respect to the algebra
$(H, \mu)$ (the right action of $(H,\mu)$ on $\bbA$ is simply denoted by juxtaposition), 
with the compatibility conditions, for all $\pP,\qQ\in\bbA$ and $\xi,\zeta\in \UU$,
\begin{subequations}\label{eqn:comp1}
\eqa
(\pP\qQ)\xi&=&\pP(\qQ\xi)~,\label{abxi}\\
(\pP\xi)\qQ &=&\pP(\xi_1\RA \qQ)\xi_2~,\label{comp19}\\
\xi\RA (\pP\zeta)&=&(\xi_1\RA\pP)  (\xi_2\RA\zeta)~, \label{Halgacrac}
\ena
\end{subequations}
where $\xi\RA\zeta=\xi_1 \zeta S(\xi_2)$ is the adjoint action of
$H$ on $H$.
In this case $\bbA$ and $\bbA_\st$ are isomorphic as algebras.
\end{Theorem}
\begin{proof}
As we shall explain in the text below, this theorem is equivalent to
Theorem \ref{isostar2}.
Therefore the proof follows from the proof of Theorem \ref{isostar2}.
\end{proof}

Notice that $\bbA$ in the hypotheses above is a left module with 
respect to the algebra $(H, \mu)$ by defining, for all $\xi\in H$ and $P\in \bbA$, 
\eq
\xi \pP:= (\xi_1\RA \pP)\,\xi_2~\label{defleftact}.
\en
Condition  (\ref{Halgacrac})  implies that $\bbA$ is an
$(H,\mu)$-bimodule
\eq\label{xiPzetaassoc}
\xi (\pP \zeta)=(\xi_1\RA (\pP\zeta))\xi_2
=((\xi_1\RA \pP) (\xi_2\RA\zeta))\xi_3
=(\xi_1\RA \pP) \,\xi_2\zeta S(\xi_3)\xi_4
=(\xi \pP)\zeta~.
\en
The Hopf algebra action $\RA$ on $\bbA$ is just the adjoint action
with respect to this bimodule structure:
\eq
\xi\RA \pP=\xi_1 \pP S(\xi_2)~.\label{adjxasx}
\en 
Condition (\ref{comp19}) then simply reads 
\eq
(\pP\xi)\qQ=\pP(\xi \qQ)\label{comp192}\\
\en
and together with (\ref{abxi}) and the $\AAAlg{H}{}{}$-algebra
property $\xi\RA (\pP\qQ)= (\xi_1\RA \pP)(\xi_2\RA \qQ)$ we obtain
\eq
\xi(\pP\qQ)=(\xi \pP)\qQ~.\label{xiab}
\en
In case $\bbA$ is unital with $1 \in \bbA$ we also find, for all
$\xi\in \UU$,
\eq\label{xi1AA}
\xi\, 1=(\xi_1\RA 1)\xi_2=1\,\epsi(\xi_1)\xi_2=1\,\xi~.
\en

Vice versa, if $\bbA$ is an algebra and an $(H, \mu)$-bimodule satisfying
(\ref{abxi}), (\ref{comp192}) 
and (\ref{xiab}) (as well as (\ref{xi1AA}) if $\bbA$ is unital), then $\xi\RA \pP:=\xi_1 \pP S(\xi_2)$ defines an
$\AAAlg{H}{}{}$-algebra structure on $\bbA$ that
satisfies (\ref{comp19}) and  (\ref{Halgacrac}).
\sk
Hence,  Theorem \ref{isostar} equivalently reads
\begin{Theorem}\label{isostar2}
Consider a Hopf algebra $H$ and an 
$(H, \mu)$-bimodule $\bbA$ that is also an algebra (not necessarily
associative or with unit). If, for all
$\xi\in H$ and $\pP,\qQ\in \bbA$,  the ``generalized associativity'' conditions
\eq\label{theorem5hp}
(\pP\qQ)\xi=\pP(\qQ\xi)
~,~~
(\pP\xi)\qQ=\pP(\xi \qQ)
~,~~
\xi(\pP\qQ)=(\xi \pP)\qQ~,
\en
and in case $\bbA$ is unital, with $1\in\bbA$, also the condition
\eq
\xi \,1=1\,\xi~,
\en
hold true, 
then the adjoint action (\ref{adjxasx}) structures
$\bbA$ as an $\AAAlg{H}{}{}$-algebra.
Given a twist $\FF$ of the Hopf algebra $\UU$, the twist deformed algebra
$\bbA_\st$ is isomorphic (as an algebra) to $\bbA$ via the map 
\eq\label{eqn:Ddef}
D_\FF:\bbA_\st\to\bbA~,~~P\mapsto D_\FF(\pP):=(\of^\al\RA \pP)\:\of_\al=\of^\al_1 \pP S(\of^\al_2)\of_\al~.
\en
\end{Theorem}
\begin{proof} 
$D_\FF$ is obviously a $\bbK$-linear map. It is also an algebra homomorphism,
for all $\pP,\qQ\in \bbA$,
\eqa
D_\FF(\pP\st \qQ)
&=&D_\FF\bigl((\of^\be\RA \pP)(\of_\be\RA \qQ)\bigr)\nonumber\\
&=&\Bigl(\of^\al\RA \bigl((\of^\be\RA\pP)(\of_\be \RA \qQ)\bigr)\Bigr)\,\,\of_\al\nonumber\\
&=&(\of^\al_{~_1}\of^\be \RA \pP)\,\,(\of^\al_{~_2}\of_\be \RA \qQ)\,
\,\of_\al
\nonumber\\
&=&(\of^\al \RA \pP)\,\,(\of_{\al_1}\of^\be \RA \qQ)\,\,\of_{\al_2}\,\,\of_\be
\nonumber\\
&=&(\of^\al\RA \pP)\,\,\of_{\al}\,\,(\of^{\be}\RA \qQ)\,\,\of_\be\nonumber\\
&=&D_\FF(\pP)\,D_\FF(\qQ)~,
\label{D alg-homo}\label{DPSTQproof}
\ena
where in the fourth line we used (\ref{ass}) and in the fifth line
we used that 
\begin{flalign}
\nn (\of_{\al_1}\of^\be\RA \qQ)\,\,\of_{\al_2}&=
\of_{\al_1}
\,\,(\of^\be\RA \qQ)\,\,S(\of_{\al_2})\,\,\of_{\al_3}\\
&=
\of_{\al_1}
\,\,(\of^\be\RA \qQ)\,\,\varepsi(\of_{\al_2})=
\of_{\al}\,\,(\of^\be\RA \qQ)~.
\end{flalign}
In order to show that $D_\FF$ is invertible we simplify (\ref{eqn:Ddef})  using (\ref{eqn:twistpropsimp}) as follows
\begin{flalign}
 \nn D_\FF(P) &= \of^\al_1 P S(\of^\al_2)\of_\al = 
\of^\al \f^\gamma P S(\of_{\al_1}\of^\beta \f_\gamma)\of_{\al_2}\of_\beta\\
\label{Disinvertible}&=\of^\al \f^\gamma P S(\f_\gamma) S(\of^\beta)\varepsilon(\of_\al)\of_\beta =
 \f^\gamma P S(\f_\gamma) \chi^{-1}~,
\end{flalign}
where $\chi^{-1} = S(\of^\al)\of_\al$.
Therefore, $D_\FF$ is invertible and we have,
for all $\pP\in \bbA$,
\eq\label{Xmap}
{D_\FF}^{-1}(\pP)=\of^\al\pP\chi S(\of_\al)~,
\en
where $\chi =\f^\be S(\f_\be)$.

Finally, if $\bbA$ is unital,   $D_\FF$ maps the unit of $\bbA_\star$ to the unit
of $\bbA$ because of the normalization property of the twist 
$(\varepsi\otimes \id)\FF^{-1}=1$.
\end{proof}

We introduce the triple notation $\big(\bbA,\mu,{\RA}\big)$ in order
to explicitly write the product and the $H$-action of an
$\AAAlg{H}{}{}$-algebra $\bbA$. Then the $\AAAlg{H^\FF}{}{}$-algebra
$\bbA_\star$ is described by the triple
$\big(\bbA,\mu_\star,{\RA}\big)$.

In the hypotheses of Theorem \ref{isostar2}, the Hopf algebra property (\ref{2.6}) immediately implies that 
the algebra $\bbA$ has an $\AAAlg{H^\FF}{}{}$-algebra structure given by the $H^\FF$-adjoint
action,
for all $\xi\in H^\FF$ and $\pP\in \bbA$,
\eq
\xi\RA_\FF \pP:=\xi_{1_\FF}\,\pP\,S^\FF(\xi_{2_\FF})~. \label{adjDFFxia}
\en   
We denote this 
$\AAAlg{H^\FF}{}{}$-algebra by $\big(\bbA,\mu,{\RA_\FF}\big)$.
We have the following
\begin{Theorem}\label{isoAAstDAA}
The algebra isomorphism 
$D_\FF : \bbA_\st\rightarrow \bbA$ of Theorem \ref{isostar2} is also an isomorphism between the 
$\AAAlg{H^\FF}{}{}$-algebras $\big(\bbA,\mu_\star,{\RA}\big)$
and $\big(\bbA,\mu,{\RA_\FF}\big)$, i.e.~$D_\FF$
intertwines between the $H^\FF$-actions  $\RA$ and $\RA_\FF$, for
all $\xi\in \UU^\FF$ and $P\in \bbA$,
\eq\label{DHFintertwiner}
D_\FF(\xi \RA \pP)=\xi\RA_\FF D_\FF(\pP)~.
\en
\end{Theorem}
\begin{proof} Using (\ref{Disinvertible}) we obtain, for all $\xi\in H^\FF$ and $P\in \bbA_\star$,
\eqa
D_\FF(\xi \RA \pP)
&=&\f^\be (\xi \RA \pP) S(\f_\be)\chi^{-1}
=\f^\be\xi_1\pP S(\xi_2)S(\f_\be)\chi^{-1}
=\f^\be\xi_1\of^{\,\ga}\f^\de \pP S(\f_\be\xi_2\of_\ga\f_\de)\chi^{-1}\nn\\
&=&
\xi_{1_\FF}\f^\de \pP S(\f_\de)\chi^{-1}\chi S(\xi_{2_\FF})\chi^{-1}
=\xi_{1_\FF}D_\FF(\pP)
S^\FF(\xi_{2_\FF})\nn\\
&=&
\xi\RA_{\FF}D_\FF(\pP)~,\label{DxitrglP}
\ena
where in the third equality we inserted $1\otimes 1=\FF^{-1}\FF$.
\end{proof}
\begin{Remark}\label{rem:Dinv}
We have discussed in Remark \ref{rem:dequant} that $H^\FF$ admits the twist $\FF^{-1}$ leading to $(H^\FF)^{\FF^{-1}}=H$.
The associated deformation isomorphism $D_{\FF^{-1}}$ is exactly ${D_\FF}^{-1}$ given in (\ref{Xmap}).
This can be shown by using (\ref{Disinvertible}) and a short calculation, for all $P\in\bbA$,
\begin{flalign}
D_{\FF^{-1}}(P)= (\f^\be\RA_\FF P) \f_\be = \of^\be P S^\FF(\of_\be) \chi_\FF^{-1} 
= \of^\be P \chi S(\of_\be)\chi^{-1} \chi= {D_\FF}^{-1}(P)~,
\end{flalign}
where we also have used that $\chi_\FF^{-1} = S^\FF(\f^\be)\f_\be = \chi$.
\end{Remark}
\sk


\begin{Example}
Given an $\AAAlg{H}{}{}$-algebra $A$ (not necessarily associative or with unit) we
consider the crossed product (or smash
product) algebra $A\rtimes \UU$.
By definition the underlying $\bbK$-module structure  is $A\otimes \UU$ and
the product is given by 
$
(a\otimes \xi)(b\otimes \zeta)=a (\xi_1\triangleright b) \otimes\xi_2\zeta
$,
that we simply rewrite as 
\eq
a\xi\,b \zeta=a(\xi_1\triangleright b) \xi_2\zeta~.
\en
The algebra
$A\rtimes \UU$ is an $\AAAlg{H}{}{}$-algebra with the action 
$\xi \RA(a\zeta) :=(\xi_1\trgl a)(\xi_2\RA\zeta)$.
The right $(H, \mu)$-module structure is given by
$(a\xi)\zeta=a(\xi\zeta)$
and the compatibility conditions (\ref{eqn:comp1}) hold true.
Hence the hypotheses of Theorem \ref{isostar} are satisfied.

\begin{Corollary}
Deformation of an $\AAAlg{H}{}{}$-algebra $A$ is the restriction to $A$
of the deformed algebra $(A\rtimes\UU)_\st$ that is isomorphic to  $A\rtimes\UU$
because of Theorem \ref{isostar}.
\end{Corollary}

If $M$ is a smooth manifold, $A=C^\infty(M)$ and $\UU=U\Xi$ is the universal enveloping algebra
 of vector fields on $M$, then $A\rtimes\UU$ is the algebra of differential
operators. Up to considering the
extension of  these algebras to the ring $\bbC[[h]]$, we have that the
map 
$D_\FF\,:\,A\rightarrow D_\FF(A)\subset A\rtimes U\Xi$ realizes a
quantization of $A$ in terms of differential operators \cite{Aschieri:2005yw,Wess:2006cm}.
\end{Example}
\sk

\begin{Example}\label{ex:injection}
Given an algebra $\bbA$ that admits an algebra homomorphism
$\rho:\UU\rightarrow \bbA$, then the hypotheses of Theorem \ref{isostar2} immediately hold.
Just define the $(H, \mu)$-bimodule structure of $\bbA$ by, for all
$\xi\in H, P\in \bbA$,  $\xi \pP :=\rho(\xi) \pP$ and $\pP\xi:=\pP \rho(\xi)$.
A particular case is when $\bbA=\UU$ and we consider the identity
homomorphism. Then we recover the (Hopf algebra) isomorphism
$D\,:\,\UU_\st\rightarrow \UU^\FF$ discussed in \cite{Aschieri:2005zs}.
\end{Example}
\sk

\begin{Example}\label{ex:endos}
Given a Hopf algebra $\UU$  and an $\MMMod{H}{}{}$-module $V$ 
we consider the algebra $\End_\bbK(V)$ of $\bbK$-linear maps
($\bbK$-module homomorphisms) from $V$ to $V$.
Since $H$ is a Hopf algebra the left action of $H$ on $V$ lifts to a left
action of $H$ on $\End_\bbK(V)$, defined by, for all $P\in
\End_\bbK(V)$ and $\xi \in H$,
\eq
\xi\btrgl P:=\xi_1\trgl \circ P \circ S(\xi_2)\trgl ~,\label{ex3.13}
\en
where $\circ$ denotes the usual composition of
morphisms and $\xi\ra \in \End_\bbK(V)$ is the endomorphism $v\mapsto \xi\ra v$. 
The algebra $\End_\bbK(V)$ is thus an $\AAAlg{H}{}{}$-algebra, and we denote it 
also by the triple $\big(\End_\bbK(V),\circ,{\RA}\big)$ in order to explicitly refer to the product and the $H$-action. The
algebra homomorphism $H\rightarrow  \End_\bbK(V)$, $\xi\mapsto \xi\trgl$ 
implies (see  Example \ref{ex:injection}) the isomorphism $D_\FF:
\End_\bbK(V)_\st\rightarrow \End_\bbK(V)$. The composition law in
$\End_\bbK(V)_\st$ is given by the $\star$-composition 
$P\circ_\st Q :=(\of^\al\RA P)\circ (\of_\al\RA Q)$, for all $P, Q\in \End_\bbK(V)_\st$.
Because of Theorem \ref{isoAAstDAA} we further obtain that $D_\FF$ is an $\AAAlg{H^\FF}{}{}$-algebra 
isomorphism between $\big(\End_{\bbK}(V),\circ_\star,{\RA}\big)$ and
$\big(\End_\bbK(V),\circ,{\RA_\FF}\big)$. 

Notice that since as algebras $H^\FF=H$, 
 the deformed $\MMMod{H^\FF}{}{}$-module $V_\star$ of Theorem \ref{Theorem2} (with trivial algebras $A=B=\bbK$) 
 is equal to the $\MMMod{H}{}{}$-module $V$ and henceforth we can
 identify the  $\AAAlg{H^\FF}{}{}$-algebras 
$\big(\End_\bbK(V),\circ,{\RA_\FF}\big)$ and 
$\big(\End_\bbK(V_\star),\circ,{\RA_\FF}\big)$.   Thus,  $D_\FF$ is also an isomorphism 
  between the $\AAAlg{H^\FF}{}{}$-algebras
 $\big(\End_{\bbK}(V),\circ_\star,{\RA}\big)$ and $\big(\End_\bbK(V_\star),\circ,{\RA_\FF}\big)$.
\end{Example}
\sk

\subsubsection*{Categorical point of view}
We provide a generalization of Example \ref{ex:endos}. Instead of
studying a fixed $\MMMod{H}{}{}$-module $V$, let
us consider the category $\rep^H$ of representations of $H$. An object in $\rep^H$ is an $\MMMod{H}{}{}$-module $V$ and
a morphism between two objects $V,W$ in $\rep^H$ is a $\bbK$-linear map
$P: V\to W$, i.e.~$P\in\Hom_\bbK(V,W)$.
Notice that we do not assume the map $P$ to be
$H$-equivariant. The composition
of morphisms is the usual composition $\circ$. Given a twist $\FF\in
H\otimes H$ of the Hopf algebra $H$ we can consider the category ${\rep^H}_{\,\star}$.
The objects and morphisms in ${\rep^H}_{\,\star}$
are the same as the objects and morphisms in $\rep^H$, but the
composition is given by the $\star$-composition $\circ_\star$, where
the $H$-action on morphisms is given by the $H$-adjoint action  $\RA$
(cf. (\ref{ex3.13})), canonically obtained lifting the source and
target $H$-actions.
\begin{Theorem}\label{theo:firstcat}
Let $H$ be a Hopf algebra with twist $\FF\in H\otimes H$. Then there is a functor  
from ${\rep^H}_{\,\star}$ to $\rep^{H}$ that is the identity on
objects, and associates
to any morphism $P: V\to W$ 
the morphism 
 \begin{flalign}
 D_\FF(P) : V \to W ~,~~v\mapsto D_\FF(P)(v) = (\of^\alpha\RA P) (\of_\alpha\ra v) = \of^\alpha_1\ra\Big(P\big(S(\of^\alpha_2)\of_\alpha\ra v\big)\Big) ~.
 \end{flalign}
Furthermore, the categories ${\rep^H}_{\,\star}$ and $\rep^{H}$ are equivalent.
\end{Theorem}
\begin{proof}
$D_\FF(P)=(\of^\alpha\RA P) \circ \of_\alpha\ra$ is obviously a $\bbK$-linear map. We also have that 
$D_\FF(\id_V) = (\of^\alpha\RA \id_V)\circ \of_\alpha \ra\,= \id_{V}$, 
because of the normalization condition (\ref{propF2}) of the
twist. In order to show that $D_\FF$
preserves composition of morphisms, i.e.~that for any two composable morphisms
$P$ and $Q$ in ${\rep^H}_{\,\star}$ we have $D_\FF(P\circ_\star Q) =
D_\FF(P)\circ D_\FF(Q)$, we repeat the same passages as in  
the proof of Theorem \ref{isostar2}.

This functor has a left and right inverse given by the functor that is
the identity on objects and that to any morphism $P : V\to W$ in $\rep^{H}$
associates the morphism $D_{\FF^{-1}}(P) = D_\FF^{-1}(P) : V\to
W$ in $\rep^{H}{}_{\,\star}$ (cf.~Remark \ref{rem:Dinv}). Hence the two categories are
equivalent.
\end{proof}
In the end of Example \ref{ex:endos} we have identified the 
$\AAAlg{H^\FF}{}{}$-algebras 
$\big(\End_\bbK(V),\circ,{\RA_\FF}\big)$ and 
$\big(\End_\bbK(V_\star),\circ,{\RA_\FF}\big)$.  Here we can similarly
identify the category $\rep^H$ with the category $\rep^{H^\FF}$ of
 $\MMMod{H^\FF}{}{}$-modules with $\bbK$-linear maps as morphisms. Indeed any 
 $\MMMod{H}{}{}$-module $V$ can be seen as an
 $\MMMod{H^\FF}{}{}$-module (denoted $V_\star$).  
It follows that  $D_\FF$ provides an equivalence between the categories $\rep^H{}_{\,\star}$ and $\rep^{H^\FF}$.
In particular, any morphism $P : V\to  W $ in $\rep^H{}_{\,\star} $ is
mapped to a morphism $D_\FF(P) : V_\star\to W_\star$ in $\rep^{H^\FF}$. 


\section{\label{sec:modhom}Module homomorphisms}
Let $A, B$ be two $\AAAlg{H}{}{}$-algebras and $V,W$ be two
$\MMMod{H}{A}{B}$-modules. In this section we study the properties of 
$\bbK$-linear maps $\Hom_\bbK(V,W)$ and
right $B$-linear maps $\Hom_B(V,W)$,
and their deformations. 
The $H$-action on the modules
$V$ and $W$ lifts to an $H$-adjoint action on these maps; in general
this adjoint action
is nontrivial because the maps are not $H$-equivariant.
To any map $P: V\to W$ we associate a deformed map
$D_\FF(P):V_\star\to W_\star$, where the deformed $\MMMod{H^\FF}{A_\star}{B_\star}$-modules
$V_\star$ and $W_\star$ are obtained according to Theorem \ref{Theorem2}.
We show that this correspondence is a bijection between $\bbK$-linear
maps
 (i.e.~between $\Hom_\bbK(V,W)$ and   $\Hom_\bbK(V_\star,W_\star)$),
and also between right $B$-linear and right $B_\star$-linear maps
(i.e.~between $\Hom_B(V,W)$ and   $\Hom_{B_\star}(V_\star,W_\star)$).

We further clarify the algebraic structures preserved by 
$D_\FF$. In particular we extend the result obtained in Example \ref{ex:endos}
where $D_\FF$ was shown to be an  isomorphism 
between the $\AAAlg{H^\FF}{}{}$-algebras
 $\big(\End_{\bbK}(V),\circ_\star,{\RA}\big)$ and $\big(\End_\bbK(V_\star),\circ,{\RA_\FF}\big)$.

Finally, for later use,  we consider a mirror construction that deforms left $A$-linear maps of
$\MMMod{H}{A}{B}$-modules into left $A_\star$-linear maps of
$\MMMod{H^\FF}{A_\star}{B_\star}$-modules. 
We conclude with a categorical  description of the results obtained.


\subsection{Deformation of  endomorphisms}
In this subsection we study the algebras $\End_\bbK(V)$ and $\End_B(V)$ of
endomorphisms of an $\MMMod{H}{A}{B}$-module $V$, where $A$ and $B$ are $\AAAlg{H}{}{}$-algebras.
In particular we focus  on the canonical $A$-bimodule structure of
$\End_\bbK(V)$ and $\End_B(V)$ induced by the $\AAAlg{H}{}{}$-algebra homomorphism $l: A\to \End_\bbK(V)$ (see Proposition \ref{Hom_K HAMA}).
This structure will also be relevant in the next subsection, 
where we discuss homomorphisms between
different $\MMMod{H}{A}{B}$-modules $V,W$.

  The behaviour of 
the endomorphism algebras $\End_\bbK(V)$ and $\End_B(V)$ under twist
deformation is studied. 
There are two equivalent  deformations of  $\End_B(V)$:
The $\End_B(V)_\star$ deformation is obtained by considering
endomorphisms as {\sl elements} of the algebra  $\End_B(V)$. The 
$\End_{B_\star}(V_\star)$ deformation  is obtained by considering them 
as right $B$-linear {\sl maps} on $V$ and deforming the module
$V$. The equivalence of the two deformations is provided by the map $D_\FF$ from right
$B$-linear to right $B_\st$-linear endomorphisms.

\begin{Proposition}\label{Hom_K HAMA}
Let $A$ be an $\AAAlg{H}{}{}$-algebra and
$V$ be an $\MMMod{H}{A}{}$-module.
Then  the algebra $\End_\bbK(V)$ of $\bbK$-linear maps
from $V$ to $V$ is an $\AAAlg{H}{A}{A}$-algebra
with the $H$-adjoint action, for all $\xi\in H$ and $P\in \End_\bbK(V)$,
\eq
\xi\btrgl P := \xi_1 \trgl\circ P \circ S(\xi_2)\trgl\label{xiadjactP}~,
\en
and the $A$-bimodule structure given by, for all $a\in A$ and $P\in \End_\bbK(V)$,
\begin{subequations}\label{aPLaP}
\eqa
a\cdot P&:=&l_a\circ P~,\\ 
P\cdot a&:=&P\circ l_a~,
\ena
\end{subequations}
where, for all $v\in V$,  $l_a(v) :=a\cdot v$.
\sk
If $B$ is another $\AAAlg{H}{}{}$-algebra and $V$ is also an $\MMMod{H}{A}{B}$-module, then the subalgebra
$\End_B(V)\subset \End_\bbK(V)$ of right $B$-linear endomorphisms of
$V$,  (i.e.~$P\in \End_B(V)$ if, for all $v\in V$ and $b\in B$, $P(v\cdot b)=P(v)\cdot b$)
is an $\AAAlg{H}{A}{A}$-subalgebra
with $H$ and $A$ actions given in  (\ref{xiadjactP}) and (\ref{aPLaP}), respectively.
\end{Proposition}
\begin{proof}
$\End_\bbK(V)$ is a left $A$-module, for all $a,\tilde a\in A$ and $P\in \End_\bbK(V)$,
\eq
a\cdot(\tilde a \cdot P)=l_a\circ l_{\tilde a}\circ P= l_{a\,\tilde a}\circ P=
(a\,\tilde a)\cdot P~.
\en
Similarly, we have that $\End_\bbK(V)$ is a right $A$-module. It is a
bimodule because $(l_a\circ P) \circ l_{\tilde a}=l_a\circ (P \circ l_{\tilde a})$.

It is straightforward to check that $\End_\bbK(V)$ is an $\AAAlg{H}{}{}$-algebra,
for all $\xi,\zeta\in \UU$ and
$P, Q\in \End_\bbK(V)$,
\begin{subequations}
\eq\xi\btrgl (\zeta\btrgl P)=(\xi\zeta)\btrgl P~,\quad \xi\RA \id_V = \varepsilon(\xi)\,\id_V~,
\en
and
\eq\label{xiPcircQ}
\xi\btrgl
(P\circ Q)=(\xi_1\btrgl P)\circ (\xi_2\btrgl Q)~.
\en
\end{subequations}
We now prove that the algebra homomorphism $l : A\rightarrow \End_\bbK(V)$
given by $a\mapsto l_a$ is also an $\AAAlg{H}{}{}$-algebra
homomorphism, i.e.~for all $a\in A$ and $\xi\in \UU$, $l_{\xi\trgl a}=\xi\btrgl
l_a$. Indeed, for all $v\in V$,
\begin{flalign}\label{4.5prop}
(\xi\btrgl l_a)(v)&=\xi_1\trgl (l_a(S(\xi_2)\trgl v))=\xi_1\trgl
(a\cdot (S(\xi_2)\trgl v))\nn\\
&=\xi_1\trgl a\cdot (\xi_2S(\xi_3)\trgl
v)= (\xi\trgl a)\cdot v=l_{\xi\trgl a}(v)~.
\end{flalign}
Compatibility between the $\UU$-module structure and the  $A$-bimodule
structure, i.e.~for all $\xi\in \UU, a\in A$ and $P\in \End_\bbK(V)$, 
$\xi\btrgl (a\cdot P)=(\xi_1\trgl a)\cdot(\xi_2\btrgl P)$ and $\xi\btrgl (P\cdot a)=
(\xi_1\btrgl P)\cdot(\xi_2\trgl a)$, follows from (\ref{xiPcircQ}) and (\ref{4.5prop}).
\sk
Finally let $V$ be an $\MMMod{H}{A}{B}$-module, then $V$ is in particular an
$(A,B)$-bimodule (i.e.~for all $a\in A, b\in B$ and $v\in V$, we have $ a \cdot (v\cdot b)=(a\cdot v)\cdot b$),
hence $l_a\in \End_B(V)$. Therefore, $a\cdot P
\in \End_B(V)$ and $P\cdot a
\in \End_B(V)$ if $P\in \End_B(V)$. Furthermore, for all $\xi\in
\UU$ and $P\in \End_B(V)$ we have $\xi\btrgl P\in \End_B(V)$, indeed, for all $b\in B$ and $v\in V$,
\eqa
(\xi\btrgl P)(v\cdot b)
&=&\xi_1\trgl \Bigl(P\bigl((S(\xi_3)\trgl v)\cdot(S(\xi_2)\trgl b)\bigr)\Bigr)
\nn\\
&=&\Bigl(\xi_1\trgl \bigl( P(S(\xi_4)\trgl v)\bigr)\Bigr) \cdot \Bigl(\xi_2\trgl
\bigl(S(\xi_3)\trgl b\bigr)\Bigr)\nn\\
&=&(\xi\btrgl P)(v)\cdot b~.
\ena
\end{proof}

Let $\UU$ be a Hopf algebra with twist $\FF\in \UU\otimes \UU$.
Let also $A, B$ be two $\AAAlg{H}{}{}$-algebras
and $V$ an $\MMMod{H}{A}{B}$-module, so that $\End_\bbK(V) $ and $ \End_B(V)$ are $\AAAlg{H}{A}{A}$-algebras.
In order to explicitly write the product,
module structure and $H$-action of these $\AAAlg{H}{A}{A}$-algebras we
use the quadruples
$\big(\End_\bbK(V),\circ,\cdot,{\RA}\big)$ and $\big(\End_B(V),\circ,\cdot,{\RA}\big)$.

We have two deformations of the endomorphisms $\End_\bbK(V)$ and
$\End_B(V)$: 
The first option,
$\End_\bbK(V)_\st$ and $\End_B(V)_\st$,
is to consider the $\AAAlg{H^\FF}{A_\st}{A_\st}$-algebras obtained by applying Theorem
\ref{Theorem2} to $\End_\bbK(V)$ and $\End_B(V)$.
It is characterized by a deformed composition law and a deformed
module structure, for all 
$P,Q \in  \End_\bbK(V)_\st$ and  $a\in A_\st$, 
\eq\label{bimastP}
P\circ_\st
Q:=(\of^\al\btrgl P)\circ (\of_\al\btrgl Q)~~,~~~~
a\st P:=l_a\circ_\st P~~,~~~~
P\st a:= 
P\circ_\st l_a~.
\en
In the quadruple notation  these 
 $\AAAlg{H^\FF}{A_\star}{A_\star}$-algebras are denoted by 
$\big(\End_\bbK(V),\circ_\star,\star,{\RA}\big)$ and  
$\big(\End_B(V),\circ_\star,\star,{\RA}\big)$.
 The second option
is simply to consider the $\bbK$-linear or right $B_\st$-linear endomorphisms of the deformed 
$\MMMod{H^\FF}{A_\star}{B_\star}$-module
$V_\st$.  From Proposition \ref{Hom_K HAMA} we know that
$\,\End_{\bbK}(V_\st)$ and $ \End_{B_\st}(V_\st)$ are $\AAAlg{H^\FF}{A_\st}{A_\st}$-algebras, 
where the product  is the usual composition.
The $A_\star$-bimodule structure is induced by the $\AAAlg{H^\FF}{}{}$-algebra homomorphism 
$l^\star: A_\star \to \End_\bbK(V_\star)$ given by, for all $a\in
A_\star$ and all $v\in V_\star$, $l^\star_a(v):= a\star v$. It explicitly reads, for all $a\in A_\st$ and $P_\st\in \End_\bbK(V_\st)$,
 \eq\label{bimaPst}
a \cdot P_\st:= l^\st_a\circ P_\st~~,~~~~
P_\st \cdot a:= P_\st\circ l^\st_a~.
\en
The $H^\FF$-action is the $H^\FF$-adjoint action given by, for all $\xi\in H^\FF$ and $P_\star\in \End_\bbK(V_\star)$,
\begin{flalign}
\xi\btrgl_\FF P_\star := \xi_{1_\FF}\ra\,\circ\, P_\star\circ S^\FF(\xi_{2_\FF})\ra\,~.
\end{flalign}
In the quadruple notation these $\AAAlg{H^\FF}{A_\star}{A_\star}$-algebras are denoted by 
$\big(\End_\bbK(V_\star),\circ,\cdot,{\RA_\FF}\big)$ 
and  $\big(\End_{B_\star}(V_\star),\circ,\cdot,{\RA_\FF}\big)$.

\begin{Theorem}\label{PDP}
Let $\UU$ be a Hopf algebra with twist $\FF\in \UU\otimes \UU$, and
let $A, B$ be two $\AAAlg{H}{}{}$-algebras and $V$ be an
$\MMMod{H}{A}{B}$-module. The map
\begin{flalign}
\label{DFFEK}
D_\FF ~:~\End_\bbK(V)_\st~~&\longrightarrow~~~~ \End_{\bbK}(V_\st)\nn\\
 P~~~~&\longmapsto~~~~  D_\FF(P):=(\of^\al\btrgl P)\circ \of_\al\trgl~ 
\end{flalign}
is an $\AAAlg{H^\FF}{A_\star}{A_\star}$-algebra isomorphism between
$\big(\End_\bbK(V),\circ_\star,\star,{\RA}\big)$ and $\big(\End_\bbK(V_\star),\circ,\cdot,{\RA_\FF}\big)$.
It restricts to an $\AAAlg{H^\FF}{A_\star}{A_\star}$-algebra isomorphism
\eq
D_\FF ~:~\End_B(V)_\st\longrightarrow \End_{B_\st}(V_\st)~
\en
between 
$\big(\End_B(V),\circ_\star,\star,{\RA}\big)$ and $\big(\End_{B_\star}(V_\star),\circ,\cdot,{\RA_\FF}\big)$.
\end{Theorem}
\begin{proof} 
The $\AAAlg{H^\FF}{}{}$-algebra isomorphism (\ref{DFFEK}) is  given by
 the isomorphism $D_\FF ~:~\End_\bbK(V)_\st\to \End_{\bbK}(V)$ discussed
in Example \ref{ex:endos}.

The map $D_\FF$ is an $\AAAlg{H^\FF}{A_\star}{A_\star}$-algebra isomorphism because
(cf.~Theorem \ref{isostar2}, and recall the bimodule structures
(\ref{bimastP}) and (\ref{bimaPst}))
$D_\FF(a\st P)=D_\FF(l_a\circ_\st P)=D_\FF(l_a)\circ D_\FF(P)$, $D_\FF(P\st a) = D_\FF(P)\circ D_\FF(l_a)$ and
$D_\FF(l_a)=l^\st_a$. The last property follows from a short calculation
\eq
D_\FF(l_a)(v)=(\of^\al\btrgl l_a)(\of_\al \trgl v)= l_{\of^\al\ra a}(\of_\al\ra v)=a\star v = l_a^\star(v)~.
\en

In order to prove that $D_\FF$ restricts to an isomorphism 
 between the $\AAAlg{H^\FF}{A_\star}{A_\star}$-subalgebras
$\End_B(V)_\st$
and $\End_{B_\st}(V_\st)$
we  show that for all $P\in \End_B(V)$
we have $D_\FF(P)\in \End_{B_\st}(V_\st)$, and that for all
$P_\st\in\End_{B_\st}(V_\st)$ we have
$D_\FF^{-1}(P_\st)\in\End_{B}(V)$. Because of Remark \ref{rem:Dinv} it is sufficient to prove
the first statement, since the second follows from twisting back with $\FF^{-1}$.
The proof is short, for all
$P\in \End_B(V)$, $v\in V$ and $b\in B$,
\eqa
D_\FF(P)(v\star b)&=&
(\of^\al\btrgl P)(\of_{\al_1}\of^\be\trgl v\cdot
\of_{\al_2}\of_\be\trgl b)\nn\\
&=&(\of^\al_1\of^\be\btrgl P)(\of^\al_2\of_\be\trgl v) \cdot
(\of_\al\trgl b)\nn\\
&=& \of^\al\ra\bigl( D_\FF(P)(v)\bigr)\cdot (\of_\al\ra b) = D_\FF(P)(v)\st b\label{DPva}~,
\ena
where in the second line we  used the twist cocycle property (\ref{ass})
and the fact that $\xi\RA P\in \End_B(V)$, for all $\xi\in H$ and
$P\in \End_B(V)$.
\end{proof}
We call $D_\FF(P): V_\star\to V_\star$ the  deformation of the
endomorphism $P: V\to V$ because $P$ is a map between undeformed modules
and can be seen as an element of the undeformed $\AAAlg{H}{A}{A}$-algebra 
$\End_\bbK(V)$. From this viewpoint
$D_\FF: \End_\bbK(V)\to \End_\bbK(V_\star)$ is a bijection from (the  
 $\AAAlg{H}{A}{A}$-algebra) 
$\End_\bbK(V)$ to (the  $\AAAlg{H^\FF}{A_\star}{A_\star}$-algebra)
$\End_\bbK(V_\star)$. 
Actually, since 
$\End_\bbK(V)$ and $\End_\bbK(V_\star)$ are $\bbK$-modules, it is
an isomorphism of $\bbK$-modules.

\subsubsection*{Left $A$-linear endomorphisms}
We have so far studied the deformations of the algebra $\End_B(V)$ of right $B$-linear
endomorphisms, but we could as well have studied the deformations of
the algebra ${}_A\End(V)$ of left
$A$-linear endomorphisms  of the $\MMMod{H}{A}{B}$-module $ V$.
These deformations are obtained by a ``mirror'' construction.  
A key point is that  there is an 
isomorphism ${}_A\End(V)\simeq \End_{A^\op}(V^\op)$ 
between left $A$-linear endomorphisms of a module $V$ and right
$A^\op$-linear endomorphisms of the opposite module $V^\op$, where, as
we detail below,
$A^\op$ has the opposite product of $A$, and $V^\op$ the opposite
module structure of $V$. Then, twist deformation of 
left $A$-linear endomorphisms in ${}_A\End(V)$ can be studied via 
twist deformation of right $A^\op$-linear endomorphisms in $\End_{A^\op}(V^\op)$. 
\sk

We recall that given a Hopf algebra 
$\UU=(\UU,\mu,\Delta, S,\epsi)$ 
(where with slight abuse of notation we
also denoted by $\UU$ the  $\bbK$-module structure underlying the Hopf
algebra $\UU$) 
we have the Hopf algebras $\UU^\cop$ and $\UU_\op$ if the antipode $S$
is invertible.
Explicitly, $\UU^\cop=(\UU, \mu, \Delta^\cop,
S^{-1},\epsi)$ is the Hopf algebra with the co-opposite coproduct $\Delta^\cop$
defined by, for all $\xi\in \UU$, $\Delta^\cop(\xi):=\xi_{1^\cop}\otimes
\xi_{2^\cop}:=\xi_2\otimes \xi_1\,,$ where $\Delta(\xi)=\xi_1\otimes \xi_2$.
$\UU_\op=(\UU, \mu^\op, \Delta,
S^{-1},\epsi)$ is the Hopf algebra with the opposite product $\mu^\op$
defined by, for all $\xi, \zeta\in \UU$,
$\mu^\op(\xi\otimes\zeta)=\zeta\xi$.
Even if the antipode $S$ is not invertible we have the Hopf algebra
$\UU^\cop_\op=(\UU,\mu^\op,\Delta^\cop, S,\epsi)$.
For a simpler mirror construction we are going to assume 
invertibility of the antipode in what follows. 
Notice that, in particular, quasitriangular Hopf algebras have an invertible antipode.

We observe that for any $\MMMod{}{A}{}$-module $V$ there is an $\MMMod{}{}{A^\op}$-module  $V^\op$. 
As $\bbK$-modules $V=V^\op$,  the algebra
$A^\op$ is the algebra with
opposite product  and its right action
on $V^\op$ is given by $v\cdot^\op a :=a\cdot v$. Similarly, for any $\MMMod{}{}{A}$-module
  $V$ there is an $\MMMod{}{A^\op}{}$-module $V^\op$, and
$(V^\op)^\op=V$.

Moreover, if we have  an $\AAAlg{H}{}{}$-algebra $A$
 then the opposite algebra $A^\op$ is an $\AAAlg{H^\cop}{}{}$-algebra,
where the Hopf algebra action is unchanged, i.e.~$A^\op$ is the $\AAAlg{H^\cop}{}{}$-algebra
$\big(A,\mu^\op,\ra\big)$.
This easily leads to the following
\begin{Lemma}\label{ABVCop}
Let $A, B$ be two $\AAAlg{H}{}{}$-algebras and $V$ be an $\MMMod{H}{A}{B}$-module. Then 
$A^\op, B^\op$ are $\AAAlg{H^\cop}{}{}$-algebras and $V^\op$ is an $\MMMod{H^\cop}{B^\op}{A^\op}$-module,
where the $H^\cop$-actions on $A^\op, B^\op, V^\op$ are the same actions
as the $H$-actions on $A,B,V$.
Similarly if  $E $ is an $\AAAlg{H}{A}{B}$-algebra,
then $E^\op$ is an $\AAAlg{H^\cop}{B^\op}{A^\op}$-algebra. 
\end{Lemma}
\begin{proof}
We here show  as an illustrative example that the algebra $A^\op$ is an $\AAAlg{H^\cop}{}{}$-algebra.
For all $\xi\in \UU$ and $a,\tilde a\in A$, 
\eqa
\xi\trgl (\mu^\op(a\otimes \tilde a))&=&\xi\trgl (\tilde a  \, a)
=(\xi_1\trgl\tilde a)\, (\xi_2\trgl a)=\mu^\op\big((\xi_2\trgl a)\otimes (\xi_1\trgl
\tilde a)\big)\nn\\
&=&\mu^\op\big((\xi_{1^\cop}\trgl a)\otimes (\xi_{2^\cop}\trgl \tilde a)\big)~.
\ena
The remaining statements are similarly proven.
\end{proof}

We apply these observations to the algebra of endomorphisms of the
module $V$ and obtain 
\begin{Proposition}\label{ABVAEnd}
Let $A, B$ be two $\AAAlg{H}{}{}$-algebras and  $V$ be an $\MMMod{H}{A}{B}$-module.
Then $\big({}_A\End(V)\big)^\op$ is an $\AAAlg{H}{B}{B}$-algebra,
i.e.~more explicitly,
$\big({}_A\End(V), \circ^\op, \cdot^\op, {\RA^\cop}\big)$ is an $\AAAlg{H}{B}{B}$-algebra,
where the left $H$-action, called $\RA^\cop$ adjoint action, is given by, for all
$\xi\in \UU$ and $P\in {}_A\End(V)$, 
\eq
\xi\btrgl^\cop P
:=\xi_2\trgl\circ P\circ S^{-1}(\xi_1)\trgl~.
\en
The $B$-bimodule structure is given by, 
for all $b\in B$ and $P\in {}_A\End(V)$, \eq b\cdot^\op P
=P\circ
r_b~~,~~~~P\cdot^\op b
=r_b\circ P~,
\en 
where for all $v\in V$, $r_b(v)=v\cdot b$.

\end{Proposition}
\begin{proof}
The hypothesis implies that
$A^\op, B^\op$ are $\AAAlg{H^\cop}{}{}$-algebras and $V^\op$ is an $\MMMod{H^\cop}{B^\op}{A^\op}$-module.
Hence, we have the  $\AAAlg{H^\cop}{B^\op}{B^\op}$-algebra
$\big(\End_{A^\op}(V^\op) ,\circ,\cdot,{\RA^\cop}\big)$,
where the 
$H^\cop$-adjoint action is
$\,\xi\btrgl^\cop P
=\xi_{1^\cop}\trgl \circ P\circ
S^\cop(\xi_{2^\cop})\trgl\:$ and the $B^\op$-actions are
\eq
P\cdot b=P\circ l_b^{B^\op}=P\circ r_b~~,~~~~b\cdot P= l_b^{B^\op}\circ P=
r_b\circ P~,
\en
 where by definition
$l^{B^\op}:B^\op\to \End_{A^\op}(V^\op)$, with $l^{B^\op}_b(v)=b\cdot^\op v=v\cdot b$. 
Because of Lemma \ref{ABVCop} 
we equivalently have the opposite $\AAAlg{H}{B}{B}$-algebra
$\big(\End_{A^\op}(V^\op),\circ^\op,\cdot^\op,{\RA^\cop}\big)$.
The thesis holds because there is a canonical $\bbK$-module isomorphism 
\eq\label{identityiso}
\End_{A^\op}(V^\op)\simeq
{}_A\End(V)\
\en
given by the identity map. Indeed, if 
$P\in \End_{A^\op}(V^\op)$ then, for all $a\in A$ and $v\in V$,
\begin{flalign}
P(a\cdot v)=P(v\cdot^\op a)=P(v)\cdot^\op a=a\cdot P(v)~,
\end{flalign}
hence $P\in {}_A\End(V)$, and vice versa. Thus,  $\big({}_A\End(V),\circ^\op,\cdot^\op,{\RA^\cop}\big)$ is an
 $\AAAlg{H}{B}{B}$-algebra.
\end{proof}
The appearance in Proposition \ref{ABVAEnd} of the opposite composition product 
$P\circ^\op Q=Q\circ P$ is naturally explained letting the left $A$-linear
endomorphism $P$ act from the right to the left, $(v)\overleftarrow{P}:=P(v)$. Then we
have $\big((v)\overleftarrow{P}\big) \overleftarrow{Q} = Q\big(P(v)\big)= \big(Q\circ P\big)(v) = \big(P\circ^\op Q\big)(v) = (v)\/\overleftarrow{P\circ^\op Q}$.
\sk
\sk
We now use the isomorphism
${}_A\End(V)\simeq \End_{A^\op}(V^\op)$ between left $A$-linear and
right $A^\op$-linear endomorphisms in order to construct  a
deformation map $P\mapsto D^\cop_\FF(P)$  from left $A$-linear to
left $A_\st$-linear endomorphisms. 
The map $D^\cop_\FF$ is induced from a deformation map 
of right $A^\op$-linear endomorphism 
that is constructed following Theorem \ref{PDP}.

Let $\UU$ be a Hopf algebra with twist $\FF\in \UU\otimes \UU$.
Let also  $A, B$ be two $\AAAlg{H}{}{}$-algebras and
$V$ an $\MMMod{H}{A}{B}$-module, so that $\big({}_A\End(V)\big)^\op$
is an $\AAAlg{H}{B}{B}$-algebra, it reads 
$\big({}_A\End(V),\circ^\op,\cdot^\op,{\RA^\cop}\big)$.
Then the canonical constructions of Theorem \ref{Theorem1}
and of Theorem \ref{Theorem2} lead to the deformed $\AAAlg{H^\FF}{}{}$-algebras
 $A_\st , B_\st$,  the deformed $\MMMod{H^\FF}{A_\star}{B_\star}$-module
$V_\st$ and the deformed $\AAAlg{H^\FF}{B_\st}{B_\st}$-algebra
$\big({}_A\End(V)\big)^\op_{~~\st}$.
This latter one explicitly reads 
$\big({}_A\End(V),(\circ^\op)_\star,(\cdot^\op)_\star,{\RA^\cop}\big)$, where $(\circ^\op)_\star$ and $(\cdot^\op)_\star$
are the $\star$-products constructed from $\circ^\op$ and the twist
$\FF$ (respectively $\cdot^\op$ and the twist $\FF$).
For example, for $P,Q\in {}_A\End(V)$, $P(\circ^\op)_\star Q = (\of^\alpha \RA^\cop P)\circ^\op (\of_\alpha\RA^\cop Q)$.

Another deformation of $\big({}_A\End(V)\big)^\op $ is achieved
by applying the construction of Proposition \ref{ABVAEnd} to the
deformed $\AAAlg{H^\FF}{}{}$-algebras 
$A_\st$, $B_\st$ and the deformed $\MMMod{H^\FF}{A_\star}{B_\star}$-module $V_\st$.
We thus obtain  the $\AAAlg{H^\FF}{B_\st}{B_\st}$-algebra
$\big(\,{}_{A_\st}\End(V_\st)\big)^\op$, or more explicitly
$\big({}_{A_\st}\End(V_\st),\circ^\op,\cdot^\op,{\RA_{\,\!\!\FF}\;\!}^\cop\big)$. Notice in 
 particular that  the $B_\st$-bimodule structure 
is given by, for all $b\in B_\st$ and $P_\star \in {}_{A_\st}\End(V_\st)$,  
\eq \label{astDPast}
b \cdot^\op P_\st=P_\st\circ r^\st_b
~~,~~~~P_\st\cdot^\op
b =r^\st_b \circ P_\st~,
\en
where,  for all $v\in V_\star$, $r^\st_b(v)=v\st b$.

\begin{Theorem}\label{PDPL}
Let $H$ be a Hopf algebra with twist $\FF\in H\otimes H$, and let $A,B$ be two $\AAAlg{H}{}{}$-algebras 
and $V$ be an $\MMMod{H}{A}{B}$-module.
The map 
\begin{flalign} \label{eqn:DFcop}
D^\cop_\FF
~:~\big({}_{A}\End(V)\big)^\op_{~~\st}~~~&\longrightarrow ~~~~
\big({}_{A_\st}\End(V_\st)\big)^\op\nn\\
P~~~~&\longmapsto ~~~~ D^\cop_\FF(P):=(\of_\al\btrgl^\cop P)\circ \of^\al\trgl~
\end{flalign} 
is an $\AAAlg{H^\FF}{B_\star}{B_\star}$-algebra isomorphism between the two deformations of left $A$-linear endomorphisms
${\big({}_A\End(V),(\circ^\op)_\star,(\cdot^\op)_\star,{\RA^\cop}\big)}$ and
$\big({}_{A_\star}\End(V_\star),\circ^\op,\cdot^\op,{\RA_{\,\!\!\FF}\;\!}^\cop\big)$.
\end{Theorem}
\begin{proof}
From Lemma \ref{ABVCop} we have the $\AAAlg{H^\cop}{}{}$-algebras $A^\op, B^\op$ and the
 $\MMMod{H^\cop}{B^\op}{A^\op}$-module $V^\op$.
Now notice that if $\FF$ is a twist of $H$, then $\FF^\cop:=\FF_{21}$ is a
twist of  $H^\cop$.  From Proposition \ref{Hom_K HAMA} and the Theorems
\ref{Theorem1} and \ref{Theorem2} we then have that
\begin{subequations}
\begin{flalign}\label{4.25a}
 \bigl(\End_{A^\op}(V^\op),\circ_{\star^\cop},{\star^\cop},{\RA^\cop}\bigr)
\end{flalign}
and
\begin{flalign}\label{4.25b}
 \big(\End_{(A^\op)_{\star^\cop}}((V^\op)_{\star^\cop}),\circ,\cdot,{\RA^\cop}_{\FF^\cop} \big)
\end{flalign}
\end{subequations}
are $\AAAlg{(H^\cop)^{\FF^\cop}\!\!}{(B^\op)_{\st^\cop}}{(B^\op)_{\st^\cop}}$-algebras.
Here we denoted by $\st^\cop$ the $\st$-product given by  the twist $\FF^\cop$.
In (\ref{4.25b}) the $(B^\op)_{\st^\cop}$-bimodule structure
is the canonical one obtained from the $\AAAlg{(H^\cop)^{\FF^\cop}}{}{}$-algebra homomorphism $l^{(B^\op)_{\star_\cop}}:
{(B^\op)_{\star_\cop}}\to \End_{(A^\op)_{\star^\cop}}((V^\op)_{\star^\cop})$.
We use the short notation $\big(\End_{A^\op}(V^\op)\big)_{\st^\cop}$ for (\ref{4.25a}) and
$\End_{(A^\op)_{_{\st^\cop}}}\!((V^\op)_{\st^\cop})$ for (\ref{4.25b}).  Now
Theorem \ref{PDP} implies the $\AAAlg{(H^\cop)^{\FF^\cop}\!\!}{(B^\op)_{\st^\cop}}{(B^\op)_{\st^\cop}}$-algebra isomorphism
\begin{flalign}
D_{\FF^\cop}
\,:\,\big(\End_{A^\op}(V^\op)\big)_{\st^\cop}\longrightarrow 
\End_{(A^\op)_{_{\st^\cop}}}\!((V^\op)_{\st^\cop})\label{eqn:Dcoptemp}~,~~
 D_{\FF^{\cop}}(P):=(\of_\al\btrgl^\cop P)\circ \of^\al\trgl~.
\end{flalign}
Observe that
\begin{subequations}
\begin{flalign}
 (A^\op)_{\star^\cop} = (A_\star)^\op~,\quad (B^\op)_{\star^\cop} = (B_\star)^\op~,\quad (V^\op)_{\star^\cop} = (V_\star)^\op~,
\end{flalign}
as well as 
\begin{flalign}
 (H^\cop)^{\FF^\cop} = (H^\FF)^\cop~,\quad {\RA^\cop}_{\FF^\cop} = {\RA_{\,\!\!\FF}\;\!}^{\cop}~.
\end{flalign}
\end{subequations}
It follows that  if we consider an $\AAAlg{H}{A}{B}$-algebra  $E$, so that $E^\op$ is an $\AAAlg{H^\cop}{B^\op}{A^\op}$-algebra,
 then the $\AAAlg{(H^\cop)^{\FF^\cop}}{(B^\op)_{\star^\cop}}{(A^\op)_{\star^\cop}}$-algebra
$(E^\op)_{\star^\cop}$
is equal to the $\AAAlg{(H^\FF)^\cop}{(B_\star)^\op}{(A_\star)^\op}$-algebra
$(E_\star)^\op$.
Using this and the canonical isomorphism
$ \End_{A^\op}(V^\op)\simeq {_A}\End(V)$ 
we find
\begin{subequations}\label{eqn:oppotemp}
\begin{flalign}
  \bigl(\End_{A^\op}(V^\op)\bigr)_{\star^\cop} \simeq \bigl({_A}\End(V)\bigr)_{\star^\cop}
= \Bigr(\bigl({_A}\End(V)\bigr)^{\op}\Bigr)_\star^{~\,~\op}~.
\end{flalign}
 We also have that
\begin{flalign}\label{4.29b}
 \End_{(A^\op)_{\star^\cop}}((V^\op)_{\star^\cop}) = \End_{(A_\star)^\op}((V_\star)^\op)\simeq {_{A_\star}}\End(V_\star)~.
\end{flalign}
\end{subequations}
The bimodule structures in the first equality  in (\ref{4.29b})
are identified via  the identification of the maps
$l^{(B^\op)_{\star_\cop}}\!:{(B^\op)_{\star_\cop}}\to \End_{(A^\op)_{\star^\cop}}((V^\op)_{\star^\cop})\,$
and $\,l^{(B_\star)^\op}\!:
(B_\star)^\op\to \End_{(A_\star)^\op}((V_\star)^\op).$ 

The proof of the theorem follows by noting that the isomorphism
defined
in (\ref{eqn:Dcoptemp}) canonically induces on the opposites of  the modules in (\ref{eqn:oppotemp})
the $\AAAlg{H^\FF}{B_\star}{B_\star}$-algebra  isomorphism
$D^\cop_\FF$ defined in  (\ref{eqn:DFcop}).
\end{proof}
We call $D_\FF^\cop(P): V_\star\to V_\star$ the left deformation of
the endomorphism $P:V\to V$.


\subsection{Deformation of homomorphisms}
Let $H$ be a Hopf algebra,  $A, B$ be two $\AAAlg{H}{}{}$-algebras and
$V$ be an $\MMMod{H}{A}{}$-module. Then, 
as seen in the previous subsection, $\End_\bbK(V)$ is an
$\AAAlg{H}{}{}$-algebra, 
where the $H$-action is given by
the $H$-adjoint action $\RA$.
The algebra structure trivially implies an $\AAAlg{H}{\End_\bbK(V)}{\End_\bbK(V)} $-algebra structure on
$\End_\bbK(V)$.
Now because of the $\AAAlg{H}{}{}$-algebra homomorphism 
$l : A\rightarrow \End_\bbK(V)$
(cf.~proof of Proposition \ref{Hom_K HAMA}) we have that
$\End_\bbK(V)$ is an $\AAAlg{H}{A}{A}$-algebra.

We here consider the $\bbK$-module $\Hom_\bbK(V, W)$ 
of $ \bbK$-linear maps from the $\MMMod{H}{A}{}$-module $V$ to the $\MMMod{H}{A}{}$-module $W$. 
In this case we immediately have an
$\MMMod{H}{\End_\bbK(V)}{\End_\bbK(W)}$-module structure on $\Hom_\bbK(V,W)$.
Also here the $H$-action is given by the $H$-adjoint action $\RA$.
Because of the $\AAAlg{H}{}{}$-algebra homomorphisms $l : A\rightarrow \End_\bbK(V)$ and
$l : A\rightarrow \End_\bbK(W)$ we similarly obtain 
an $\MMMod{H}{A}{A}$-module structure on $\Hom_\bbK(V,W)$, explicitly
$\big(\Hom_\bbK(V,W),\cdot,{\RA}\big)$.
Furthermore, similarly to  Proposition \ref{Hom_K HAMA},
we have that if $V,W$ are $\MMMod{H}{A}{B}$-modules, then the
$\bbK$-submodule $\Hom_B(V,W)$ of right $B$-linear homomorphisms forms
an $\MMMod{H}{A}{A}$-submodule of $\Hom_\bbK(V,W)$. The
$\MMMod{H}{A}{A}$-module structure of  $\Hom_B(V,W)$  in the case $B=A$ will
be a key ingredient in order to study tensor products over $A$  of
homomorphisms, see Section \ref{subsecprodonA}.
\begin{Remark}\label{rem:hwlms}
We later also encounter the situation where $V$ is only an
$\MMMod{H}{}{B}$-module 
while $W$ is an $\MMMod{H}{A}{B}$-module.
In this case we similarly have that $\Hom_\bbK(V,W),\Hom_B(V,W)$ are
$\MMMod{H}{A}{}$-modules,
and $\Hom_\bbK(W,V),\Hom_B(W,V)$ are $\MMMod{H}{}{A}$-modules.
\end{Remark}
\sk

The results of the previous subsection concerning the deformation of
endomorphisms generalize to the case of homomorphisms.
We present only the main theorems and omit the proofs that are  
easily obtained following the corresponding ones for endomorphisms.

\begin{Theorem}\label{PDPHom}
Let $\UU$ be a Hopf algebra with twist $\FF\in \UU\otimes \UU$, and
let  $A, B$ be two $\AAAlg{H}{}{}$-algebras and
$V, W$ be two $\MMMod{H}{A}{B}$-modules. The map  
\begin{flalign}\label{DFFEKHom}
D_\FF ~:~\Hom_\bbK(V, W)_\st~~~&\longrightarrow ~~~~ \Hom_{\bbK}(V_\st,W_\st)\nn\\
P~~~~&\longmapsto ~~~~ D_\FF(P):=(\of^\al\btrgl P)\circ \of_\al\trgl~
\end{flalign}
is an $\MMMod{H^\FF}{A_\star}{A_\star}$-module isomorphism between 
$\big(\Hom_\bbK(V,W),\star,{\RA}\big)$ and
$\big(\Hom_\bbK(V_\star,W_\star),\cdot,{\RA_\FF}\big)$.
It restricts to an $\MMMod{H^\FF}{A_\star}{A_\star}$-module isomorphism
\eq
D_\FF ~:~\Hom_B(V,W)_\st\longrightarrow \Hom_{B_\st}(V_\st,W_\st)~
\en
between $\big(\Hom_B(V,W),\star,{\RA}\big)$ and $\big(\Hom_{B_\star}(V_\star,W_\star),\cdot,{\RA_\FF}\big)$.
\end{Theorem}
We call $D_\FF(P): V_\star\to W_\star$ the deformation of the
homomorphism $P: V\to W$.

\sk

\subsubsection*{Left $A$-linear homomorphisms}
Left $\AA$-linear homomorphisms  $\big({}_A\Hom(V,W)\big)^\op$, between 
$\MMMod{H}{A}{B}$-modules $V$ and $W$, 
have an $\MMMod{H}{B}{B}$-module structure,
explicitly $\big({}_A\Hom(V,W),\cdot^\op,{\RA^\cop}\big)$,
  where, for all $b\in B$,
$\xi\in \UU$ and $ P\in {}_A\Hom(V,W)$, 
$
\xi\btrgl^\cop P
:=\xi_2\trgl\circ P\circ S^{-1}(\xi_1)\trgl\,
$, $b\cdot^\op P=P\circ
r_b$ and $P\cdot^\op b=r_b\circ P$.
\sk
 Up to isomorphism
there is just one
deformed module of left $\AA$-linear homomorphisms.
\begin{Theorem}\label{PDPLHom}
Let $H$ be a Hopf algebra with twist $\FF\in H\otimes H$, and let  $A, B$ be two $\AAAlg{H}{}{}$-algebras and
$V , W$ be two $\MMMod{H}{A}{B}$-modules.
The map
\begin{flalign}
D^\cop_\FF
~:~\big({}_{A}\Hom(V,W)\big)^\op_{~~\st}~~~&\longrightarrow ~~~~
\big({}_{A_\st}\Hom(V_\st, W_\st)\big)^\op\nn\\
P~~~~~&\longmapsto ~~~~ D^\cop_\FF(P):=(\of_\al\btrgl^\cop P)\circ \of^\al\trgl~
\end{flalign}
is an $\MMMod{H^\FF}{B_\star}{B_\star}$-module isomorphism between the two deformations
of left $A$-linear homomorphisms
$\big({}_A\Hom(V,W),(\cdot^\op)_\star,{\RA^\cop}\big)$ and
$\big({}_{A_\star}\Hom(V_\star,W_\star),\cdot^\op,{\RA_\FF}^\cop\big)$.
\end{Theorem}
We call $D^\cop_\FF(P):V_\star\to W_\star$ the left deformation of the
homomorphism $P:V\to W$.

\sk

\begin{Example}\label{PDPLex}
Let  $A$ be an $\AAAlg{H}{}{}$-algebra
and $V$ be an $\MMMod{H}{A}{A}$-module. The {\bf dual module of $V$} is
defined by $V^\prime :=\Hom_A(V,A)$. Since $A$ can be regarded as an $\MMMod{H}{A}{A}$-module
we have that  $V^\prime$ is also an $\MMMod{H}{A}{A}$-module.
Let $\FF\in \UU\otimes \UU$ be a twist of $\UU$ and consider the
deformed $\AAAlg{H^\FF}{}{}$-algebra $A_\star $
and the deformed $\MMMod{H^\FF}{A_\star}{A_\star}$-module $V_\st$.
 We have two  possible deformations of the dual module:  
${V^\prime}_{\,\st}=\Hom_A(V,A)_\st$
and ${V_\st}^{\,\prime} =\Hom_{A_\st}(V_\st,A_\st)$. They both are $\MMMod{H^\FF}{A_\st}{A_\st}$-modules.
 Due to Theorem \ref{PDPHom}
there is an $\MMMod{H^\FF}{A_\star}{A_\star}$-module isomorphism
${V^\prime}_{\,\st}\simeq {V_\st}^{\,\prime}$. In words, dualizing the deformed module is (up
to the canonical isomorphism $D_\FF$) equivalent to deforming the dual one. Similar
statements hold true for the left $A$-linear dual
${}^\prime V:=({}_A\Hom(V,A))^\op$ and its deformations according to Theorem \ref{PDPLHom}.
\end{Example}


\subsection{Categorical formulation}
Developing the results of Theorem \ref{theo:firstcat}  we provide a
categorical formulation of the investigations of the present section.
We first define the category $\big(\MMMod{H}{A}{B},\Hom_B,\circ\big)$.
An object in $\big(\MMMod{H}{A}{B},\Hom_B,\circ\big)$ is an $\MMMod{H}{A}{B}$-module $V$
and a morphism between two objects $V,W$ in $\big(\MMMod{H}{A}{B},\Hom_B,\circ\big)$
is a right $B$-linear map $P:V\to W$, i.e.~$P\in \Hom_B(V,W)$. The composition of morphisms 
is the usual composition $\circ$. Given a twist $\FF\in H\otimes H$ of the Hopf algebra
$H$ we then define two other categories, $\big(\MMMod{H}{A}{B},\Hom_B,\circ_\star\big)$
and $\big(\MMMod{H^\FF}{A_\star}{B_\star},\Hom_{B_\star},\circ\big)$.
Objects and morphisms in $\big(\MMMod{H}{A}{B},\Hom_B,\circ_\star\big)$ are the same as
objects and morphisms in $\big(\MMMod{H}{A}{B},\Hom_B,\circ\big)$, but in this category the
composition of morphisms is given by the $\star$-composition $\circ_\star$. An object in
$\big(\MMMod{H^\FF}{A_\star}{B_\star},\Hom_{B_\star},\circ\big)$ is an $\MMMod{H^\FF}{A_\star}{B_\star}$-module
$V_\star$ and a morphism between two objects
$V_\star,W_\star$ in $\big(\MMMod{H^\FF}{A_\star}{B_\star},\Hom_{B_\star},\circ\big)$ is a right
$B_\star$-linear map $P_\star:V_\star\to W_\star$, i.e.~$P_\star\in \Hom_{B_\star}(V_\star,W_\star)$.
The composition of morphisms is the usual composition $\circ$.

\begin{Theorem}\label{functortheoremR}
Let $H$ be a Hopf algebra with twist $\FF\in H\otimes H$. Then there
is a functor
from $\big(\MMMod{H}{A}{B},\Hom_B,\circ_\star\big)$ to $\big(\MMMod{H^\FF}{A_\star}{B_\star},\Hom_{B_\star},\circ\big)$.
It maps any $\MMMod{H}{A}{B}$-module $V$ to  the twist deformed $\MMMod{H^\FF}{A_\star}{B_\star}$-module $V_\star$
(cf.~Theorem \ref{Theorem2}), and any morphism $P\in \Hom_B(V,W)$ to the
morphism $D_\FF(P)\in \Hom_{B_\star}(V_\star,W_\star)$ (cf.~Theorem \ref{PDPHom}).

Furthermore, the categories $\big(\MMMod{H}{A}{B},\Hom_B,\circ_\star\big)$ and
 $\big(\MMMod{H^\FF}{A_\star}{B_\star},\Hom_{B_\star},\circ\big)$ are equivalent.
\end{Theorem}
\begin{proof}
The proof of this theorem follows similar steps as that of Theorem \ref{theo:firstcat}.
\end{proof}
\sk
\subsubsection*{Left $A$-linear morphisms}
For completeness we state without proof the corresponding theorem for left $A$-linear maps. 

We first define the category $\big(\MMMod{H}{A}{B},{}_A\Hom,\circ^\op\big)$.
An object in $\big(\MMMod{H}{A}{B},{}_A\Hom,\circ^\op\big)$ is an $\MMMod{H}{A}{B}$-module $V$
and a morphism between two objects $V,W$ in $\big(\MMMod{H}{A}{B},{}_A\Hom,\circ^\op\big)$
is a left $A$-linear map $P:V\to W$, i.e.~$P\in {}_A\Hom(V,W)$. The
composition of morphisms $\circ^\op$ is described after Proposition \ref{ABVAEnd}.

Given a twist $\FF\in H\otimes H$ of the Hopf algebra
$H$ we then define two other categories, $\big(\MMMod{H}{A}{B},{}_A\Hom,(\circ^\op)_{\star}\big)$
and $\big(\MMMod{H^\FF}{A_\star}{B_\star},{}_{A_\star}\Hom,\circ^\op\big)$.
Objects and morphisms in $\big(\MMMod{H}{A}{B},{}_A\Hom, (\circ^\op)_{\star}\big)$ are the same as
objects and morphisms in $\big(\MMMod{H}{A}{B},{}_A\Hom,\circ^\op\big)$, but in this category the
composition of morphisms is given by  $(\circ^\op)_{\star}$, i.e.~for two composable morphisms
$P (\circ^\op)_{\star} Q = (\of^\alpha\RA^\cop P)\circ^\op (\of_\alpha\RA^\cop Q)$.
 An object in $\big(\MMMod{H^\FF}{A_\star}{B_\star},{}_{A_\star}\Hom,\circ^\op\big)$ is an 
 $\MMMod{H^\FF}{A_\star}{B_\star}$-module
$V_\star$ and a morphism between two objects
$V_\star,W_\star$ in $\big(\MMMod{H^\FF}{A_\star}{B_\star},{}_{A_\star}\Hom,\circ^\op\big)$ is a left
$A_\star$-linear map $P_\star:V_\star\to W_\star$, i.e.~$P_\star\in {}_{A_\star}\Hom(V_\star,W_\star)$.
The composition of morphisms is $\circ^\op$.
\begin{Theorem}\label{functortheoremL}
Let $H$ be a Hopf algebra with twist $\FF\in H\otimes H$. Then there is functor
from $\big(\MMMod{H}{A}{B},{}_A\Hom, (\circ^\op)_{\star}\big)$ to
 $\big(\MMMod{H^\FF}{A_\star}{B_\star},{}_{A_\star}\Hom,\circ^\op\big)$.
It maps any $\MMMod{H}{A}{B}$-module $V$ to the twist deformed $\MMMod{H^\FF}{A_\star}{B_\star}$-module $V_\star$
(cf.~Theorem \ref{Theorem2}) and  any morphism $P\in {}_A\Hom(V,W)$ to the 
morphism $D^\cop_\FF(P)\in {}_{A_\star}\Hom(V_\star,W_\star)$ (cf.~Theorem \ref{PDPLHom}).

Furthermore, the categories $\big(\MMMod{H}{A}{B},{}_A\Hom,(\circ^\op)_{\star}\big)$ and
 $\big(\MMMod{H^\FF}{A_\star}{B_\star},{}_{A_\star}\Hom,\circ^\op\big)$ are equivalent.
\end{Theorem}
\sk

\begin{Remark}
This is a generalization of the equivalence of categories obtained by Giaquinto and Zhang \cite{GIAQUINTO}.
We can recover their equivalence by choosing the algebra $B=\bbK$ to be trivial and
restricting the class of morphisms to $H$-equivariant (respectively
$H^\FF$-equivariant) and left $A$-linear
 (respectively left $A_\star$-linear) maps.
Then the $(\circ^\op)_\star$ composition equals the $\circ^\op$ composition
and the equivalence is between the categories $\big(\MMMod{H}{A}{},{}_A\Hom,\circ^\op\big)$ and
 $\big(\MMMod{H^\FF}{A_\star}{},{}_{A_\star}\Hom,\circ^\op\big)$. Notice
 also that in this case the functor  acts trivially on morphisms.
We remark  that the more general equivalence of 
 Theorem \ref{functortheoremL} is not between the 
 original category $\big(\MMMod{H}{A}{B},{}_A\Hom,\circ^\op\big)$ and  
 $\big(\MMMod{H^\FF}{A_\star}{B_\star},{}_{A_\star}\Hom,\circ^\op\big)$, but between the original category with deformed composition law and the latter one.
\end{Remark}


\section{\label{sec:tenprod}Tensor product structure and its deformation}
Let $H$ be a Hopf algebra.
The tensor product $V\otimes W$ of two $\MMMod{H}{}{}$-modules
$V,W$ is again an $\MMMod{H}{}{}$-module.
The left $H$-action 
is defined using the coproduct,
for all $\xi\in H$, $v\in V$ and $w\in W$,
\begin{flalign}
\label{eqn:productaction}
\xi\ra (v\otimes w) := (\xi_1\ra v)\otimes (\xi_2\ra w)~,
\end{flalign}
and extended by $\bbK$-linearity to all $V\otimes W$.

We begin this section  studying 
the lift of two $\bbK$-linear maps $P: V\rightarrow \widetilde V$, $Q
: W\rightarrow \widetilde W$ to a tensor product map $V\otimes W\rightarrow
\widetilde  V\otimes \widetilde W$. The issue is that  the
construction has to be compatible with the $H$-action even if 
the maps we consider are  in general only $H$-covariant (not $H$-equivariant).
This requires extra structure on the Hopf algebra $H$.
Indeed, in order to introduce a tensor product of $\bbK$-linear maps between $\MMMod{H}{}{}$-modules
that is compatible with the Hopf algebra action we require a braiding
isomorphism on tensor products of $\MMMod{H}{}{}$-modules. We
therefore consider quasitriangular Hopf algebras.
Next, we study the deformation of the tensor product of $\bbK$-linear
maps and show that the deformation procedure is canonical.

In a later subsection we focus on the
restricted class of quasi-commutative $\AAAlg{H}{}{}$-algebras and
$\MMMod{H}{A}{A}$-modules. The tensor product structure previously
studied induces a tensor product structure over $A$. In particular
the tensor product of right $A$-linear maps
between quasi-commutative $\MMMod{H}{A}{A}$-modules is compatible with the Hopf algebra
action and is again a right $A$-linear map. 
Also the deformation of this tensor product is studied and shown to be canonical.

Finally, we consider the deformation map $D_\RR$, that corresponds to
the twist $\FF=\RR$, where $\RR$ is the universal $\RR$-matrix of
the quasitriangular Hopf algebra $H$.  The map $D_\RR$ provides an
isomorphism between right and left $A$-linear maps
on strong quasi-commutative $\MMMod{H}{A}{A}$-modules.

\subsection{Triangular and quasitriangular Hopf algebras}
A {\bf cocommutative Hopf algebra} is a Hopf algebra where the coproduct is
cocommutative, i.e.,~for all $\xi\in H$,
$\Delta^\cop(\xi)=\Delta(\xi)$ or, using Sweedler's notation,  $\xi_2\otimes\xi_1=\xi_1\otimes\xi_2$. 
\begin{Definition}
A {\bf quasi-cocommutative Hopf algebra $(H,\RR)$} is a Hopf algebra $H$ and an invertible
element $\RR\in H\otimes H$ (called {\bf universal $\RR$-matrix}) such that, for all $\xi\in \UU$,
\begin{flalign}\label{12}
{\Delta}^{\cop}(\xi)=\RR\,\Delta(\xi)\,\RR^{-1}~.
\end{flalign}
The Hopf algebra is {\bf quasitriangular} if moreover
\begin{flalign}\label{13}
(\Delta\otimes \id)\RR =\RR_{13}\RR_{23}~,\quad
(\id\otimes \Delta)\RR =\RR_{13}\RR_{12}~.
\end{flalign}
The quasitriangular Hopf algebra $(H,\RR)$ is {\bf triangular} if 
\begin{flalign}
\RR_{21}=\RR^{-1}~\label{14},
\end{flalign}
where $\RR_{21}=\sigma (\RR)\in \UU\otimes\UU$,
with $\sigma $ the transposition map
$\sigma(\xi\otimes\zeta)=\zeta\otimes\xi$, for all $\xi, \zeta\in H$.
\end{Definition}
\sk

For later use we write the property (\ref{12}) and the inverse of
properties (\ref{13}) using
Sweedler's notation and the notations $\RR=R^\al\otimes R_\al$,
$\RR^{-1}=\oR^\al\otimes\oR_\al$ (sum over $\alpha$ understood). For all $\xi\in H$,
\begin{subequations}
\begin{flalign} 
\xi_2\otimes\xi_1 &=R^\al\xi_1\oR^\be\otimes R_\al\xi_2\oR_\be~,~\\
\oR^\al_{~1}\otimes \oR^{\al}_{~2}\otimes
\oR_\al& =\oR^\al\otimes\oR^\be\otimes\oR_{\be}\oR_{\al}~,\label{Ral1Ral21first}\\
\oR^\al\otimes \oR_{\al_1}\otimes\oR_{\al_2}&=\oR^\al\oR^\be\otimes\oR_{\al}\otimes\oR_{\be}~.~~~
\label{Ral1Ral21second}
\end{flalign}
\end{subequations}
The triangular property reads $R_\al\otimes R^\al=\oR^\al\otimes \oR_\al$.

From (\ref{12}) (with $\xi=R^\al$) and (\ref{13})  it follows that 
quasitriangular $\RR$-matrices satisfy the Yang-Baxter equation
\eq
\RR_{12}\RR_{13}\RR_{23}=\RR_{23}\RR_{13}\RR_{12}~.\label{YBE}
\en
Further standard properties of quasitriangular $\RR$-matrices are 
(see e.g.~\cite{Majid:1996kd}, Lemma 2.1.2)
\begin{subequations}\eqa
&(\epsi\otimes \id)\RR=1~~~~~~,&~~~(\id\otimes \epsi)\RR=1\label{Rnorm}~,\\
&~(S\otimes \id)\RR=\RR^{-1}~~,&~~~(\id\otimes S)\RR^{-1}=\RR~.\label{SidR}
\ena
\end{subequations}
Notice that the properties  (\ref{13}), (\ref{YBE}) and (\ref{Rnorm}) imply that $\RR$ is a twist
element of the Hopf algebra $H$. From property (\ref{12}) it then
follows that $H^\RR=H^\cop$. The Hopf algebra $H^\cop$ is 
quasitriangular with $\RR$-matrix $\RR^\cop=\RR_{21}$.

\sk
Given  two $\MMMod{H}{}{}$-modules $V,W$ we have the tensor
product $\MMMod{H}{}{}$-modules $V\otimes W$ and $ W\otimes V$ (see
(\ref{eqn:productaction})).
There is a natural isomorphism, called {\bf braiding}, between these
two tensor products; it is defined by 
\begin{subequations}
\eqa
\nn \tau_{{\RR}\;W,V} :  W\otimes V&\rightarrow& V\otimes W\\
w\otimes v&\mapsto& \tau_{\RR\; W,V}(w\otimes v)= 
(\oR^\alpha\ra v) \otimes (\oR_\alpha\ra w)~,\label{eqn:Rflipmap}
\ena
\eqa
\nn \tau^{-1}_{{\RR}\;W,V}
:  V\otimes W&\rightarrow& W\otimes V\\
v\otimes w&\mapsto& \tau^{-1}_{\RR\; W,V}(v\otimes w)= 
(\R_\alpha\ra w) \otimes (\R^\alpha\ra w)~.\label{eqn:Rflipmapinv}
\ena
\end{subequations}
and  extended by $\bbK$-linearity to all $W\otimes V$ (and respectively $V\otimes W$).
Quasi-cocommutativity of the coproduct (cf.~(\ref{12})) implies
that $\tau_\RR$ and its inverse $\tau_\RR^{-1}$
(for ease of notation we frequently omit the module indices)
 are $\MMMod{H}{}{}$-module isomorphisms, i.e., for all $\xi\in H, 
v\in V, w\in W,\,$ 
$\xi \ra(\tau_\RR(w\otimes v))=\tau_\RR(\xi\ra(w\otimes v))$, 
$\xi \ra(\tau^{-1}_\RR(v\otimes w))=\tau^{-1}_\RR(\xi\ra(v\otimes w)),$
or equivalently 
\begin{flalign}\label{xitauR}
\xi\RA\tau_\RR = \varepsilon(\xi)\,\tau_\RR~,\quad \xi\RA\tau_\RR^{-1} = \varepsilon(\xi)\,\tau_\RR^{-1}~. 
\end{flalign}
From (\ref{13}) it follows 
that on the triple tensor product $V\otimes
W\otimes Z$ of $\MMMod{H}{}{}$-modules  $V,W,Z$  
we have the braid relations
\begin{subequations}\label{hexagon}
\begin{flalign}
\tau_{\RR\;V\otimes W,Z}&=(\tau_{\RR\;V,Z}\otimes \id_W)\circ 
(\id_V\otimes \tau_{\RR\; W,Z})~,~~\\
\tau_{\RR\;V,W\otimes Z}&=(\id_W\otimes \tau_{\RR\;V,Z})\circ (\tau_{\RR\; V,W}\otimes\id_Z)~.
\end{flalign}
\end{subequations}
The first one, for example, states that flipping an element $z$ to the left of the element
$v\otimes w$ is the same as first flipping $z$ to the left of $w$ and
then the result to the left of $v$.
\begin{Example}
The universal enveloping algebra ${U}\gg$ of a Lie algebra $\gg$ is a 
cocommutative Hopf algebra.
Every cocommutative Hopf algebra $H$ has
a triangular structure given by the  $\RR$-matrix $\RR=1\otimes 1$. 
Let $\FF$ be a twist of this cocommutative Hopf algebra $H$, then the
Hopf algebra $H^\FF$ is triangular with $\RR$-matrix
$\RR^\FF = \FF_{21}\FF^{-1}$. 

More in general, if $(H,\RR)$ is a quasitriangular (triangular) Hopf algebra,
then $(H^\FF,\RR^\FF := \FF_{21}\RR\FF^{-1} )$ is a quasitriangular (triangular) Hopf algebra.
\end{Example}

\subsection{\label{subsec:productmodhom}Tensor product of  $\bbK$-linear maps}\label{5.2subsec}
Given two $\bfK$-linear maps $P : V\to \widetilde{V}$ and $Q : W\to\widetilde{W}$
between $\bfK$-modules, the tensor product map
$P\otimes Q : V\otimes W\to \widetilde{V}\otimes \widetilde{W}$
is the $\bbK$-linear map defined by, for all $v\in V$ and $w\in W$,
\begin{flalign}
\label{eqn:Ktensorhom}
(P\otimes Q)(v\otimes w):=P(v)\otimes Q(w)~,
\end{flalign}
and extended to all $V\otimes W$ by $\bbK$-linearity.
If $\widetilde{P}:\widetilde{V}\to\widehat{V}$ and  $\widetilde{Q}:
\widetilde{W}\to\widehat{W}$ are two further $\bbK$-linear maps,
then we have the composition property
\begin{flalign}
\label{eqn:Ktensorhomcomp}
\bigl(\widetilde{P}\otimes \widetilde{Q}\bigr)\circ\bigl(P\otimes Q  \bigr) = (\widetilde{P}\circ P)\otimes (\widetilde{Q}\circ Q)~.
\end{flalign}

Let now $H$ be a Hopf algebra. We consider $\MMMod{H}{}{}$-modules
$V,W,\widetilde{V},\widetilde{W}$ and the associated tensor product
$\MMMod{H}{}{}$-modules $V\otimes W$ and $\widetilde V\otimes \widetilde W$.
The $\bfK$-modules of $\bbK$-linear maps 
$\Hom_\bfK(V,\widetilde{V})$, $\Hom_\bfK(W,\widetilde{W})$ and 
$\Hom_\bfK(V\otimes W,\widetilde{V}\otimes \widetilde{W})$ are
$\MMMod{H}{}{}$-modules  with the $H$-adjoint action.
We study the action of $\xi\in H$ on the tensor product map (\ref{eqn:Ktensorhom}).
Using (\ref{eqn:productaction})  and (\ref{eqn:Ktensorhomcomp}) we obtain
\begin{flalign}
\nn \xi\RA (P\otimes Q) &= (\xi_1\ra\,\otimes\xi_2\ra\,) \circ (P\otimes Q)\circ (S(\xi_4)\ra\,\otimes S(\xi_3)\ra)\\
\nn &= \bigl( \xi_1\ra\,\circ P\circ S(\xi_4)\ra\, \bigr)\otimes \bigl(\xi_2\ra\,\circ Q\circ S(\xi_3)\ra\,\bigr)\\
 \label{eqn:helpprodhom}&=\bigl( \xi_1\ra\,\circ P\circ S(\xi_3)\ra\, \bigr)\otimes  (\xi_2\RA Q )~.
\end{flalign}
For a non-cocommutative Hopf algebra and 
$\bbK$-linear maps $Q$ that are not $H$-equivariant (i.e., $\xi\RA Q\not=\epsi(\xi)
Q$) this expression  in general
differs from  $(\xi_1\RA P)\otimes (\xi_2\RA Q)$. This shows that
the tensor product of $\bfK$-linear maps (\ref{eqn:Ktensorhom}) is in general incompatible with the
$\MMMod{H}{}{}$-module structure.
This incompatibility can be understood as follows: Considering $\bfK$-linear 
maps as acting from left to right, the ordering on the left hand side of
(\ref{eqn:Ktensorhom})  is $P,Q,v,w$, while the ordering on the right hand side is $P,v,Q,w$,
i.e.~$v$ and $Q$ do not appear properly ordered in the definition (\ref{eqn:Ktensorhom}).
For a quasitriangular Hopf algebra $(H,\RR)$ this problem can be solved by
defining a new tensor product of
$\bfK$-linear maps. 

\begin{Definition}\label{defi:Rtensor}
Let $(H,\RR)$ be a quasitriangular Hopf algebra and $V,W,\widetilde{V},\widetilde{W}$ be $\MMMod{H}{}{}$-modules.
 The {\bf $\RR$-tensor product} of $\bfK$-linear maps is defined by, for
all $P\in\Hom_\bfK(V,\widetilde{V})$ and $Q\in\Hom_\bfK(W,\widetilde{W})$,
\begin{flalign}
\label{eqn:Rtensor}
P\otimes_\RR Q := (P\circ \oR^\alpha\ra\,)\otimes (\oR_\alpha\RA Q)\in \Hom_\bfK(V\otimes W,\widetilde{V}\otimes\widetilde{W})~,
\end{flalign}
where $\otimes$ is defined in (\ref{eqn:Ktensorhom}).
\end{Definition}

This definition is related to the $H$-equivariant maps that appeared in  
\cite{Majid:1996kd}, Corollary 9.3.16, and in 
\cite{Fiore:2008sj}, eq. (38). 
The tensor product perspective we consider in this paper leads
to further investigations (starting from Theorem  \ref{RotimesK}).
\sk

From Definition  \ref{defi:Rtensor} it immediately follows that
\eq\label{POQPcircQ}
P\otimes_\RR Q = (P\otimes \id)\circ (
\oR^\alpha\ra\;\otimes \,\oR_\alpha\RA Q)
=(P\otimes_\RR \id)\circ (\id\otimes_\RR Q)~.
\en
We see that the lift of 
$P\in \Hom_\bfK(V,\widetilde{V})$ to 
$\Hom_\bfK(V\otimes\widetilde W,\widetilde{V}\otimes\widetilde{W})$ is simply
$P\otimes_\RR\id=P\otimes \id$, while the lift of $Q$ is 
\eq
\id\otimes_\RR Q=
\oR^\alpha\ra\;\otimes \,\oR_\alpha\RA Q~.\label{idQRRQ}
\en
Use of the braiding map $\tau_{\RR}$ (cf.~({\ref{eqn:Rflipmap}}))
allows us to rewrite the lift $\id\otimes_\RR Q$ acting on $V\otimes W$ in terms of the lift 
$Q\otimes \id$ acting on $W\otimes V$,
\begin{flalign}
\nn \id \otimes_\RR Q &= \oR^\alpha\ra\;\otimes\; \oR_\alpha\RA Q\\
 \nn &= \oR^\alpha\ra\;\otimes \bigl(\oR_{\alpha_1}\ra\,\circ \,Q\circ S(\oR_{\alpha_{2}})\ra\, \bigr)\\
 \nn &= \oR^\alpha\oR^\beta\ra\;\otimes \bigl(\oR_{\alpha}\ra\,\circ\, Q\circ S(\oR_{\beta})\ra\, \bigr)\\
 \nn &= \oR^\alpha R^\beta\ra\;\otimes \bigl(\oR_{\alpha}\ra\,\circ \,Q\circ R_{\beta}\ra\, \bigr)\\
\label{eqn:Rtensorsimple} &= \tau_{\RR}\circ (Q\otimes \id ) \circ \tau^{-1}_{\RR}~.
\end{flalign}

We now show that the $\RR$-tensor product $\otimes_\RR$ is compatible with the 
$\MMMod{H}{}{}$-module structure, that it is associative and
that it satisfies a braided composition law. 
\begin{Theorem}\label{RotimesK}
Let $(H,\RR)$ be a quasitriangular Hopf algebra and $V,W,Z,\widetilde{V},\widetilde{W},\widetilde{Z},\widehat{V},\widehat{W}$ be
$\MMMod{H}{}{}$-modules.
The $\RR$-tensor product is compatible with the $\MMMod{H}{}{}$-module structure, i.e.,~for all
$\xi\in H$, $P\in\Hom_\bfK(V,\widetilde{V})$ and $Q\in\Hom_\bfK(W,\widetilde{W})$,
\begin{subequations}
\begin{flalign}
\label{eqn:RtensorHmod}
\xi\RA (P\otimes_\RR Q) = (\xi_1\RA P)\otimes_\RR (\xi_2\RA Q)~.
\end{flalign}
Furthermore, the $\RR$-tensor product is associative, i.e.,~for all
$P\in\Hom_\bfK(V,\widetilde{V})$, $Q\in\Hom_\bfK(W,\widetilde{W})$ and $T\in\Hom_\bfK(Z,\widetilde{Z})$,
\begin{flalign}
\label{eqn:Rtensorass}
\bigl(P\otimes_\RR Q\bigr)\otimes_\RR T = P\otimes_\RR \bigl(Q\otimes_\RR T\bigr)~,
\end{flalign}
and satisfies the braided composition law, for all $P\in\Hom_\bfK(V,\widetilde{V})$,  $Q\in\Hom_\bfK(W,\widetilde{W})$,
$\widetilde{P}\in\Hom_{\bfK}(\widetilde{V},\widehat{V})$ and $\widetilde{Q}\in\Hom_\bfK(\widetilde{W},\widehat{W})$,
\begin{flalign}
\label{eqn:Rtensorcirc}
\bigl(\widetilde{P}\otimes_\RR\widetilde{Q}\bigr)\circ \bigl(P\otimes_\RR Q\bigr) = \bigl(\widetilde{P}\circ (\oR^\alpha\RA P)\bigr)
\otimes_\RR\bigl((\oR_\alpha\RA \widetilde{Q})\circ Q\bigr)~.
\end{flalign}
\end{subequations}
\end{Theorem}
\begin{proof}
From  (\ref{eqn:helpprodhom}) we have compatibility between the
$H$-action and the lift $P\mapsto P\otimes \id$, for all
$\xi\in H$,  
\eq\label{xiPid}
\xi\RA(P\otimes \id) = (\xi\RA P)\otimes \id\,.\en
 Compatibility between the $H$-action and  the
lift $Q\mapsto\id\otimes_\RR Q$ follows from  (\ref{xitauR}) and (\ref{xiPid}),
\begin{flalign}\label{xiidQ}
\nn \xi\RA (\id\otimes_\RR Q) &= \xi\RA\bigl(\tau_\RR\circ (Q\otimes \id)\circ \tau_\RR^{-1}\bigr)\\
\nn &=  \tau_\RR \circ \xi \RA (Q\otimes \id)\circ \tau_\RR^{-1}\nn\\ &=\tau_\RR \circ \big((\xi \RA Q)\otimes\, \id\big)\circ \tau_\RR^{-1}\nn\\
&=\id\otimes_\RR\xi\RA Q~.
\end{flalign}
Equation (\ref{eqn:RtensorHmod}) then follows 
from 
$ P\otimes_\RR Q = (P\otimes_\RR \id)\circ (\id\otimes_\RR Q)$ and 
the left $H$-action property 
$
\xi\RA(T\circ  \check T)=\xi_1\RA T\circ \xi_2\RA \check T\,,
$
that holds for any two composable $\bbK$-linear maps $T$ and $\check T$.

We now prove (\ref{eqn:Rtensorass}). The left hand side of (\ref{eqn:Rtensorass}) can be expanded
as follows
\begin{subequations}
\begin{flalign}
 \nn \bigl(P\otimes_\RR Q\bigr)\otimes_\RR T &= 
\Bigl(\bigl(P\otimes_\RR Q\bigr)\circ \bigl(\oR^\alpha_1\ra\,\otimes \oR^\alpha_2\ra\,\bigr)\Bigr)\otimes \Bigl(\oR_\alpha\RA T\Bigr)\\
\nn &= \Bigl(P\circ \oR^\beta\oR^\alpha_1 \ra\,\Bigr)\otimes \Bigl((\oR_\beta\RA Q)\circ \oR^\alpha_2\ra\,\Bigr)\otimes \Bigl(\oR_\alpha\RA T\Bigr)\\
&= \Bigl(P\circ \oR^\beta\oR^\alpha \ra\,\Bigr)\otimes \Bigl((\oR_\beta\RA Q)\circ \oR^\gamma\ra\,\Bigr)\otimes\Bigl( \oR_\gamma\oR_\alpha\RA T\Bigr)~,
\end{flalign}
where in the third line we have used (\ref{Ral1Ral21first}). This expression equals the right hand side of
(\ref{eqn:Rtensorass}), indeed
\begin{flalign}
\nn P\otimes_\RR \bigl(Q\otimes_\RR T\bigr) 
\nn &= \Bigl(P\circ \oR^\alpha\ra\,\Bigr)\otimes \Bigl(\oR_{\alpha_1}\RA Q\otimes_\RR \oR_{\alpha_2}\RA T\Bigr)\\
\nn &= \Bigl(P\circ \oR^\alpha\ra\,\Bigr)\otimes \Bigl((\oR_{\alpha_1}\RA Q)\circ \oR^\gamma\ra\,\Bigr)\otimes \Bigl(\oR_\gamma \oR_{\alpha_2}\RA T\Bigr)\\
 &= \Bigl(P\circ \oR^\beta \oR^\alpha\ra\,\Bigr)\otimes \Bigl((\oR_{\beta}\RA Q)\circ \oR^\gamma\ra\,\Bigr)\otimes \Bigl(\oR_\gamma \oR_{\alpha}\RA T\Bigr)~,
\end{flalign}
\end{subequations}
where in the third line we have used 
(\ref{Ral1Ral21second}).

Finally, we show (\ref{eqn:Rtensorcirc}).
From (\ref{POQPcircQ}) we have
\begin{flalign}
\label{eqn:Rcircstep1}
 \bigl(\widetilde{P}\otimes_\RR\widetilde{Q}\bigr)\circ \bigl(P\otimes_\RR
Q\bigr)=
(\widetilde{P}\otimes_\RR \id)\circ (\id\otimes_\RR \widetilde{Q})
\circ (P\otimes_\RR\id)\circ 
 (\id\otimes_\RR Q)~,
\end{flalign}
and therefore  (\ref{eqn:Rtensorcirc}) is proven if 
\eq 
(\id\otimes_\RR \widetilde{Q})
\circ (P\otimes_\RR\id)=\oR^\al\RA( P\otimes_\RR\id)\,\circ\,
\oR_\al\RA (\id\otimes_\RR \widetilde{Q})
\en
or equivalently 
$ (\id\otimes_\RR \widetilde{Q})
\circ (P\otimes_\RR\id)=(\oR^\al\RA P\,\otimes_\RR\id)\circ
(\id\otimes_\RR \,\oR_\al\RA\widetilde{Q})~.
$
 This last equality
holds true because (use (\ref{idQRRQ}) and  (\ref{Ral1Ral21first})),
\begin{flalign}
  (\id\otimes_\RR \widetilde{Q})
\circ (P\otimes_\RR\id)&=(\oR^\al\ra\,\otimes \,\oR_\al\RA \widetilde Q)\circ (P\otimes \id)\nn\\
&= (\oR^\alpha\ra\circ\, P)\otimes \bigl(\oR_\alpha\RA\widetilde{Q}\bigr)~\nn\\ 
&= \bigl((\oR^\alpha_1\RA P)\circ \oR^\alpha_2\ra\,\bigr)\otimes \bigl(\oR_\alpha\RA\widetilde{Q}\bigr)~\nn\\
\nn &= \bigl((\oR^\alpha\RA P)\circ \oR^\beta\ra\,\bigr)\otimes \bigl(\oR_\beta\oR_\alpha\RA\widetilde{Q}\bigr)~\\
\label{eqn:Rcircstep2}&=(\oR^\alpha\RA P)\otimes_\RR (\oR_\alpha\RA \widetilde{Q})~.
\end{flalign}
\end{proof}

Let us now consider the case where $V,\widetilde{V}$ are $\MMMod{H}{}{}$-modules and
$W,\widetilde{W}$ are $\MMMod{H}{}{A}$-modules, with $A$  an $\AAAlg{H}{}{}$-algebra.
Then we can equip $V\otimes W$ (as well as $\widetilde{V}\otimes \widetilde{W}$) 
with a right $A$-module structure by defining $(v\otimes w)\cdot a := v\otimes(w\cdot a)$,
for all $v\in V$, $w\in W$ and $a\in A$. This right $A$-action is
extended by $\bbK$-linearity to all $V\otimes W$.  Moreover we have that
$V\otimes W$  and $\widetilde{V}\otimes \widetilde{W}$ are  $\MMMod{H}{}{A}$-modules,
where the left $H$-action is given in (\ref{eqn:productaction}).
\begin{Proposition}\label{RALOR}
 Let $(H,\RR)$ be a quasitriangular Hopf algebra, $A$ be an $ \AAAlg{H}{}{}$-algebra, 
 $V,\widetilde{V}$ be two $\MMMod{H}{}{}$-modules
and $W,\widetilde{W}$ be two $\MMMod{H}{}{A}$-modules. Then we have, for all
$P\in\Hom_\bfK(V,\widetilde{V})$ and $Q\in\Hom_A(W,\widetilde{W})$,
\begin{flalign}\label{PQALinMap}
 P\otimes_\RR Q \in\Hom_A(V\otimes W,\widetilde{V}\otimes\widetilde{W})~.
\end{flalign}
\end{Proposition}
\begin{proof}
For all $a\in A$, $v\in V$ and $w\in W$, 
\begin{flalign}
\nn (P\otimes_\RR Q)\bigl( (v\otimes w)\cdot a \bigr) &=(P\otimes_\RR Q)\bigl( v\otimes (w\cdot a) \bigr) 
= P\bigl(\oR^\alpha\ra v\bigr)\otimes (\oR_\alpha\RA Q)(w\cdot a)\\
\nn &=P\bigl(\oR^\alpha\ra v\bigr)\otimes (\oR_\alpha\RA Q)(w)\cdot a\\
&=\bigl((P\otimes_\RR Q)(v\otimes w)\bigr)\cdot a~,
\end{flalign}
where in the second line we have used that
$\xi\RA Q\in\Hom_A(W,\widetilde{W})$, for all $\xi\in H$.
\end{proof}

\subsection{Deformation}\label{5.3subsec}
We study the twist deformation of  tensor
products of $\MMMod{H}{}{}$-modules 
and $\bbK$-linear maps.
\sk
Since as algebras $H$ and $H^\FF$ coincide, we have that any
$\MMMod{H}{}{}$-module $V$ is equivalently an
$\MMMod{H^\FF}{}{}$-module. 
It is however convenient to distinguish between these two module
structures and hence we denote by $V_\star$ the $\bbK$-module $V$ with the
$\MMMod{H^\FF}{}{}$-module structure. This notation agrees with
that of Theorem \ref{Theorem2}
where we deformed an $\MMMod{H}{A}{B}$-module $V$ into an
$\MMMod{H^\FF}{A_\star}{B_\star}$-module $V_\star$ (just consider
trivial algebras $A=B=\bbK$). 

Given two $\MMMod{H^\FF}{}{}$-modules $V_\star,W_\star$ we denote their tensor product
 by $V_\star\otimes_\star W_\star$. By definition
 $V_\star\otimes_\star W_\star$ equals $V\otimes W$ as $\bbK$-module;
the $\MMMod{H^\FF}{}{}$-module structure is canonically given by
the $H^\FF$-coproduct (cf.  (\ref{eqn:productaction})), 
for all $\xi\in H^\FF$, $v\in V_\star$ and $w\in W_\star$,
 $\xi\ra_\FF (v\otimes_\st w):=\xi_{1_\FF}\ra v\otimes_\st
 \xi_{2_\FF}\ra w$ (and extended by $\bbK$-linearity to all elements in $V_\star\otimes_\star W_\star$).

We now compare the $\MMMod{H^\FF}{}{}$-modules $V_\star\otimes_\star W_\star$ and
$(V\otimes W)_\star$; in this latter the $H^\FF$-action is just the
$H$-action on $V\otimes W$, hence it is obtained using the $H$-coproduct.

It is easy to show that the $\bfK$-linear map
\eqa\label{phiVWVW}
\varphi_{V,W}:=\FF^{-1}\trgl\,: V_\star\otimes_\star W_\star &\to &(V\otimes W)_\star\nn\\
v\otimes_{\st} w&\mapsto &\varphi_{V,W}(v\otimes_{\star} w)
=(\of^\alpha\ra v)  \otimes (\of_\alpha\ra w)
\ena
provides an isomorphism between the $\MMMod{H^\FF}{}{}$-modules
$V_\star\otimes_\st W_\star$ and $(V\otimes W)_\star$.
Indeed it intertwines between the two $H^\FF$-actions,
\begin{flalign} \label{varphixivw}
\varphi_{V,W}\bigl(\xi\ra_\FF(v\otimes_\star w)\bigr) &= \varphi_{V,W}\bigl((\xi_{1_\FF}\ra v)\otimes_\star (\xi_{2_\FF}\ra w)  \bigr)
= (\xi_{1}\of^\alpha\ra v)\otimes (\xi_{2}\of_\alpha\ra w) \nn\\
&= \xi\ra\varphi_{V,W}(v\otimes_\star w)~.
\end{flalign}
The inverse of $\varphi_{V,W}$ is 
\eqa
\varphi^{-1}_{V,W}:=\FF \ra\,: (V\otimes W )_\star &\to &V_\star\otimes_\st W_\star\nn\\
v\otimes w&\mapsto &\varphi^{-1}_{V,W}(v\otimes w)
=(\f^\alpha\ra v)  \otimes_\st ( \f_\alpha\ra w)~.
\ena
Consider now three  $\MMMod{H}{}{}$-modules $V,W,\Omega$. 
The twist cocycle property (\ref{ass}) immediately implies the following
commutative diagram of $\MMMod{H^\FF}{}{}$-module isomorphisms
 \begin{flalign}\label{eqn:higheriotast}
 \xymatrix{
V_\star\otimes_{\star} W_\star\otimes_{\star} \Omega_\star
\ar[d]_-{\id_{V_\st}\otimes_\RR
  \varphi_{W,\Omega}}\ar[rrr]^-{\varphi_{V,W}\otimes_\RR\id_{\Omega_\st}}
& & &(V\otimes W)_\star \otimes_{\star} \Omega_\star
\ar[d]^-{\varphi_{(V\otimes W),\Omega}}\\
V_\star\otimes_{\star} (W\otimes \Omega)_\star \ar[rrr]_-{\varphi_{V,(W\otimes\Omega)}} & & &(V\otimes W\otimes \Omega)_\star
}
\end{flalign}

\sk
Concerning the tensor product of $\bbK$-linear maps, 
on one hand the deformation of $P\otimes_\RR Q \in\Hom_\bbK(V\otimes
W,\widetilde{V}\otimes\widetilde{W})$,
according to Theorem \ref{PDPHom},  is given by
$D_\FF(P\otimes_\RR Q) \in\Hom_\bbK((V\otimes  
W)_\star,(\widetilde{V}\otimes\widetilde{W})_\star)$.
On the other hand, we recall that if $(H,\RR)$ is a quasitriangular Hopf
algebra with twist $\FF\in H\otimes H$,  then
$(H^\FF,\RR^\FF=\FF_{21}\,\RR\,\FF^{-1})$ is a quasitriangular Hopf
algebra. We therefore have the tensor product  $\otimes_{\RR^\FF}$
of $\bbK$-linear maps between $\MMMod{H^\FF}{}{}$-modules.
In particular, the lift of a $\bbK$-linear map $P_\star :V_\star\rightarrow
\widetilde{V}_\star$
to the $\MMMod{H^\FF}{}{}$-module $V_\star\otimes_\st W_\star$ is the $\bbK$-linear map $P_\star\otimes_{\RR^\FF}
\id:=P_\star\otimes_\st \id$ that as usual
is defined by, for all $v\in V_\star, w\in W_\star$, 
\eq
(P_\star\otimes_{\RR^\FF}
\id) (v\otimes_\st w):=(P_\star\otimes_\st \id) (v\otimes_\st
w):=P_\star(v)\otimes_\st w~.\label{TRFID}
\en
\begin{Theorem}\label{theo:promodhomdef}
Let $(H,\RR)$ be a quasitriangular Hopf algebra with twist $\FF\in
H\otimes H$ and  $V,W,\widetilde{V},\widetilde{W}$ be $\MMMod{H}{}{}$-modules.
Then for all $P\in \Hom_\bfK(V,\widetilde{V})$
and $Q\in\Hom_\bfK(W,\widetilde{W})$
the following diagram of $\bbK$-linear maps commutes:
\begin{flalign}\label{eqn:promodhomdef}
 \xymatrix{
V_\star\otimes_\star W_\star \ar[d]_-{\varphi^{}_{V,W}}\ar[rrrr]^-{D_\FF(P)\otimes_{\RR^\FF}D_\FF(Q)} & & & &\widetilde{V}_\star\otimes_\star\widetilde{W}_\star \ar[d]^-{^{}\varphi_{\widetilde{V},\widetilde{W}}}\\
(V\otimes W)_\star \ar[rrrr]_-{D_\FF\bigl((\of^\alpha\RA P)\otimes_\RR (\of_\alpha\RA Q)\bigr)}& & & &(\widetilde{V}\otimes\widetilde{W})_\star
}
\end{flalign}
i.e.
\begin{flalign}
\label{eqn:prodhomdef}
D_\FF\bigl((\of^\alpha\RA P)\otimes_\RR (\of_\alpha\RA Q)\bigr)= \varphi_{\widetilde{V},\widetilde{W}} \circ\big( D_\FF(P)\otimes_{\RR^\FF}D_\FF(Q)\big) \circ \varphi^{-1}_{V,W}~.
\end{flalign}
\end{Theorem}
\begin{proof}
Use of (\ref{POQPcircQ}) and compatibility between the $H$-action and
the lifts of $P$ and $Q$ (cf.~(\ref{xiPid}) and (\ref{xiidQ}))
shows that (\ref{eqn:prodhomdef}) is equivalent to
\eq
D_\FF\bigl((P\otimes_\RR \id)\circ_\st (\id\otimes_\RR Q)\bigr)= \varphi_{\widetilde{V},\widetilde{W}} \circ
\big(D_\FF(P)\otimes_{\RR^\FF} \id\big)\circ \big(\id \otimes_{\RR^\FF} D_\FF(Q) \big)\circ \varphi_{V,W}^{-1}~.
\en
Because of the algebra isomorphism
(\ref{D alg-homo}), see also Theorem \ref{theo:firstcat}, the thesis  (\ref{eqn:prodhomdef}) is equivalent to
\eq\label{DFPDFQ}
D_\FF(P\otimes_\RR \id)\circ D_\FF (\id\otimes_\RR Q)= \varphi_{\widetilde{V},\widetilde{W}} \circ
\big(D_\FF(P)\otimes_{\RR^\FF} \id\big)\circ \big(\id \otimes_{\RR^\FF} D_\FF(Q) \big)\circ \varphi_{V,W}^{-1}~.
\en
The deformation of $P\otimes_\RR \id=P\otimes \id: V\otimes
\widetilde{W}\to \widetilde{V}\otimes\widetilde{W}$ (and also of $Q\otimes \id$) can be simplified as follows
\begin{flalign}
\nn D_\FF(P\otimes \id) &= \bigl((\of^\alpha\RA P)\circ \of_{\alpha_1}\ra\,\bigr)\otimes \of_{\alpha_2}\ra\,\\
\nn &=\bigl((\of^\alpha_1\of^\beta\RA P)\circ \of^\alpha_2\of_\beta \f^\gamma \ra\,\bigr)\otimes \of_{\alpha}\f_\gamma\ra\,\\
\nn&= \bigl(\of^\alpha\ra\,\circ D_\FF(P)\circ \f^\gamma\ra\,\bigr)\otimes \of_\alpha \f_\gamma\ra\, \\
&= \varphi_{\widetilde{V},\widetilde{W}}\circ (D_\FF(P)\otimes_\st \id)\circ\varphi^{-1}_{V,\widetilde{W}}~,
\end{flalign}
where in the second line we inserted on the right $\id\otimes
\id=\FF^{-1}\FF\trgl$ and then used  the twist cocycle property (\ref{ass}).
In the last line we used (\ref{TRFID}).

The deformation of 
\begin{flalign}
\id\otimes_\RR Q=\tau_\RR\circ (Q\otimes \id) \circ
\tau_\RR^{-1}=\tau_\RR\circ_\st (Q\otimes \id) \circ_\st \tau_\RR^{-1}
\end{flalign}
(where in the last equality we used
that $\tau_\RR$ and $\tau_\RR^{-1}$ are $H$-equivariant, i.e.~$\xi\RA\tau_\RR =\varepsilon(\xi)\,\tau_\RR$
and $\xi\RA\tau_\RR^{-1} =\varepsilon(\xi)\,\tau_\RR^{-1}$, for all $\xi\in H$)
is given by
\begin{flalign}
D_\FF(\id\otimes_\RR Q)
&=\tau_\RR \circ\varphi_{\widetilde{W},V}\circ  (D_\FF(Q)\otimes_\st \id)\circ\varphi_{W,V}^{-1}\circ\tau_\RR^{-1}\nn\\
&=\varphi_{V,\widetilde{W}} \circ \tau_{\RR^\FF} \circ  (D_\FF(Q)\otimes_\st \id) \circ \tau_{\RR^\FF}^{-1}\circ \varphi_{V,W}^{-1}~,
\end{flalign}
where in the first equality we again used that $\tau_\RR$ is
$H$-equivariant. In the last equality we have defined
$\tau_{\RR^\FF}: \widetilde{W}_\star\otimes_\st V_\star\rightarrow V_\star\otimes_\st \widetilde{W}_\star$ by $\tau_{\RR^\FF}
:= \varphi_{V,\widetilde{W}}^{-1}\circ \tau_\RR
\circ\varphi_{\widetilde{W},V}$
and similarly  $\tau^{-1}_{\RR^\FF}: V_\star\otimes_\st W_\star\to W_\star\otimes_\st V_\star$ 
by $\tau^{-1}_{\RR^\FF}
:= \varphi_{W,{V}}^{-1}\circ \tau^{-1}_\RR
\circ\varphi_{{V},W}$.

Equality (\ref{DFPDFQ}) holds because $\tau_{\RR^\FF} : W_\star\otimes_\st
V_\star\to V_\star\otimes_\st W_\star$ defined by
$\tau_{\RR^\FF}:=\varphi_{V,W}^{-1}\circ \tau_\RR \circ\varphi^{}_{W,V}$ is
easily seen to be the braiding map for the twist deformed tensor product, for all
$v\in V_\star, w\in W_\star$, 
\begin{flalign}\label{tauRFmap}
\tau_{\RR^\FF}(w\otimes_\st v) = (\oR^{\FF\alpha}\ra v) \otimes_\st (\oR^\FF_{\;\alpha}\ra w)~,
\end{flalign}
so that, as in (\ref{eqn:Rtensorsimple}), we have,
$
\tau_{\RR^\FF} \circ  (D_\FF(Q)\otimes_\st \id) \circ \tau_{\RR^\FF}^{-1}=\id
\otimes_{\RR^\FF} D_\FF(Q).
$
\end{proof}
\sk
\begin{Remark}\label{rem:tensorcat}
We provide a  categorical description of the results in subsections
\ref{5.2subsec} and \ref{5.3subsec}, and show that, because of commutativity of the diagrams (\ref{eqn:higheriotast}) and (\ref{eqn:promodhomdef}), the equivalence of the categories $\rep^H{}_{\,\star}$ and
$\rep^{H^\FF}$ proven after Theorem \ref{theo:firstcat} extends to the
tensor product structures that can be considered on these categories.

We recall that an object in $\rep^H$ is an $\MMMod{H}{}{}$-module $V$
and a morphism between two objects $V,W$ in $\rep^H$ is a
$\bbK$-linear map $P\in \Hom_\bbK(V,W)$ (not necessarily $H$-equivariant). 
Let us consider the association $\otimes_\RR :\rep^H\times\rep^H\to \rep^H$ given on 
objects $(V,W)$ by the tensor product $V\otimes W$ and on morphisms
$(P,Q)$ by the $\RR$-tensor product
of $\bbK$-linear maps $P\otimes_\RR Q$ (cf. Definition \ref{defi:Rtensor}).
Because of the braided composition law (\ref{eqn:Rtensorcirc}) this is
not a bifunctor. We say that it is  ``almost'' a bifunctor and refer to $(\rep^H,\otimes_\RR)$ as an ``almost monoidal'' category. 

We similarly have the ``almost monoidal'' category
$(\rep^{H^\FF},\otimes_{\RR^\FF})$, where the association
 $\otimes_{\RR^\FF} :\rep^{H^\FF}\times\rep^{H^\FF}\to \rep^{H^\FF}$
 is defined on objects $(V_\star,W_\star)$ by $V_\star\otimes_\star
 W_\star$ and on morphisms $(P_\star,Q_\star)$ by $P_\star \otimes_{\RR^\FF} Q_\star$. 
Notice that $\bbK$-linear maps between
$\MMMod{H^\FF}{}{}$-modules are denoted with a $\star$-index, like
$P_\star: V_\star\to W_\star$,  because they carry the $H^\FF$-adjoint action, rather than the
$H$-adjoint action carried by $\bbK$-linear maps between
$\MMMod{H}{}{}$-modules $P:V\to W$. This adjoint action enters the 
definition of the tensor product morphism $P_\star \otimes_{\RR^\FF} Q_\star$. 

We also consider the  ``almost monoidal''  category
$({\rep^H}_{\,\star}, {\otimes_{\RR}}_\star )$,
where the association  
${\otimes_{\RR}}_\star : {\rep^{H}}_{\,\star}\times
{\rep^{H}}_{\,\star}\to {\rep^{H}}_{\,\star}$ is defined on 
objects $(V,W)$ by $V\otimes W$ and on morphisms $(P,Q)$ 
by $P{\otimes_\RR}_\star Q = (\of^\alpha\RA P )\otimes_{\RR}
(\of_\alpha\RA Q)$ (the braided composition law reads as
(\ref{eqn:Rtensorcirc}) with composition of morphisms given by  the
$\star$-composition $\circ_\star$, and tensor product of morphisms
given by ${\otimes_\RR}_\star$).

The equivalence of the categories $\rep^H{}_{\,\star}$ and
$\rep^{H^\FF}$ shown after Theorem  \ref{theo:firstcat} extends to an
equivalence of
$({\rep^H}_{\,\star},{\otimes_\RR}_\star)$ and $(\rep^{H^\FF},\otimes_{\RR^\FF})$
as ``almost monoidal'' categories. Indeed, the collection of maps
 $\varphi_{V,W}:V_\star\otimes_\star W_\star \to (V\otimes W)_\star$
 provides a natural isomorphism between the ``almost bifunctors'' $\otimes_{R^\FF}$ and $\otimes_{\RR_\star}$. This is so because the
 $\varphi_{V,W}$ maps satisfy the commutative diagrams
(\ref{eqn:higheriotast}) and (\ref{eqn:promodhomdef}).

Notice that when we restrict to $H$-equivariant morphisms, 
the $\RR$-tensor product reduces to the usual tensor product
 (\ref{eqn:Ktensorhom}) and the	``almost monoidal'' category
 $\big(\rep^H,\otimes_\RR)$ restricts to the monoidal category
 $\big(\rep_\eqv^H,\otimes)$.
 In this case we recover the results of Drinfeld \cite{Drinfeld89} on the equivalence of the monoidal categories
 $\big(\rep_\eqv^H,\otimes)$ and
 $\big(\rep_\eqv^{H^\FF},\otimes_{\star})$.
\end{Remark}

\subsection{Quasi-commutative algebras and bimodules
(tensor product over $A$)\label{subsecprodonA}}
Let $A$ be an $\AAAlg{H}{}{}$-algebra and consider an $\MMMod{H}{}{A}$-module 
$V$ and an $\MMMod{H}{A}{A}$-module $W$.
We have that $V\otimes W$ is an $\MMMod{H}{}{A}$-module, where the module
structure is given by, for all $\xi\in H, v\in V, w\in W, a\in A$,  
$\xi\ra(v\otimes w) = (\xi_1\ra v)\otimes(\xi_2\ra w)$ and $(v\otimes w)\cdot a = v\otimes (w\cdot a)$.
\sk
We now consider $V\otimes_A W$, i.e.,  the tensor product over $A$ of  $V$ and $W$.
We recall that it can be defined as 
the quotient  of the $\bbK$-module $V\otimes W$ via the
$\bbK$-submodule $\mathcal{N}_{V,W}$ generated by the elements 
$v\cdot a\otimes w-v\otimes a\cdot w$, for all 
$a\in A,v\in V, w\in W$.
The image of $v\otimes w$ under the canonical projection $\pi :
V\otimes W\to V\otimes_A W$ is denoted by $v\otimes_A w$. 
Since the $\bbK$-submodule $\mathcal{N}_{V,W}$ is also an $\MMMod{H}{}{A}$-submodule
 of $V\otimes W$,
an $\MMMod{H}{}{A}$-module structure on
$V\otimes_A W$ is canonically induced from the one 
on $V\otimes W$.
Explicitly we have, for all $\xi\in H, v\in V, w\in W, a\in A$,   $\xi\ra(v\otimes_A w) = (\xi_1\ra
v)\otimes_A(\xi_2\ra w)$ and
$(v\otimes_A w)\cdot a = v\otimes_A(w\cdot a)$. 

\sk
Given two right $A$-linear maps 
$P\in\Hom_A(V,\widetilde{V})$,
$Q\in\Hom_A(W,\widetilde{W})$,
where
$V,\widetilde{V}$ are two $\MMMod{H}{}{A}$-modules and $W,\widetilde{W}$ are two $\MMMod{H}{A}{A}$-modules,
following Proposition \ref{RALOR} we can construct the $\RR$-tensor product
 $P\otimes_\RR Q\in \Hom_A(V\otimes W,\widetilde V\otimes\widetilde W)$. 
We study if $P\otimes_\RR Q$ induces a map in $\Hom_A(V\otimes_A W,\widetilde V\otimes_A\widetilde W)$. 
In other words, we study if the $\RR$-tensor product of $\bfK$-linear maps  induces 
an $\RR$-tensor product between right
$A$-linear maps that is compatible with the tensor product of
modules over $A$.  This requires a quasi-commutative structure on  $\MMMod{H}{A}{A}$-modules.

\begin{Definition} 
Let $(H, \RR)$ be a quasitriangular Hopf algebra. 
Then an $\AAAlg{H}{}{}$-algebra $A$  is called
{\bf quasi-commutative} if, for all $a,\tilde a\in A$,
\eq
a\,\tilde a=(\oR^\al\trgl \tilde a) (\oR_\al\trgl  a)~.\label{QCalg}
\en
In this case, an $\MMMod{H}{A}{A}$-module $V$
is called  {\bf quasi-commutative} if, for all $a\in A, v\in V$,
\eq\label{QCrightact}
v\cdot a=(\oR^\al\trgl a)\cdot(\oR_\al\trgl v)~.
\en
\end{Definition}
\sk

It is not difficult to prove that every  $\MMMod{H}{A}{}$-module $V$, with
 $A$ being a quasi-commutative $\AAAlg{H}{}{}$-algebra, is a quasi-commutative $ \MMMod{H}{A}{A}$-module 
 with the right $A$-action defined  by
(\ref{QCrightact}).

The quasi-commutativity conditions (\ref{QCalg}) and
(\ref{QCrightact}) are equivalent to
\eq
\tilde a\, a=(\R_\al\trgl  a) (\R^\al\trgl \tilde a)~,~~a\cdot v=(\R_\al\trgl v)\cdot(\R^\al\trgl a)~.
\en
In the algebra case we can therefore use both $\RR^{-1}$ and $\RR_{21}$ (being the inverses of the two quasitriangular structures
$\RR$ and $\RR_{21}^{-1}$ on $H$) to
commute products of algebra elements. This suggests the following
definition for the case of modules.
\begin{Definition} Let $(H, \RR)$ be a quasitriangular Hopf algebra,
$A$ be a quasi-commutative $\AAAlg{H}{}{}$-algebra and $V$ be a quasi-commutative $\MMMod{H}{A}{A}$-module.
We say that $V$
is {\bf strong quasi-commutative}  if in addition, for all $a\in A, v\in V$,
\eq\label{QCrightactstrong}
v\cdot a=(\R_\al\trgl a)\cdot(\R^\al\trgl v)~.
\en
Equivalently, $V$ is strong
quasi-commutative if it is quasi-commutative with respect to both
quasitriangular structures $\RR$ and $\RR^{-1}_{21}$ on $H$.
\end{Definition} 
\sk
\begin{Example}
The universal enveloping algebra ${U}\gg$ of a Lie algebra $\gg$ is a 
cocommutative Hopf algebra.
Every cocommutative Hopf algebra $H$ has
a triangular structure given by the  $\RR$-matrix $\RR=1\otimes 1$. 
Let $\FF$ be a twist of this cocommutative Hopf algebra $H$, then the
Hopf algebra $H^\FF$ is triangular with $\RR$-matrix
$\RR^\FF=\FF_{21}\FF^{-1}$. 

Commutative $\AAAlg{U\gg}{}{}$-algebras $A$ and commutative
 $\MMMod{U\gg}{A}{A}$-modules $V$ are (strong) quasi-commutative, and so are
 their twist deformations. In particular, the modules and algebras of
 noncommutative gravity \cite{Aschieri:2005zs} are (strong) quasi-commutative.
\end{Example}
\begin{Example}
The twist deformation $H^\FF$ of any quasitriangular (triangular) Hopf algebra $(H, \RR)$
is quasitriangular (triangular) with $\RR$-matrix $\RR^\FF=\FF_{21}\RR\FF^{-1}$.

If $A$ is a quasi-commutative $\AAAlg{H}{}{}$-algebra, then $A_\st$
is a quasi-commutative $\AAAlg{H^\FF}{}{}$-algebra. If $V$ is a (strong) quasi-commutative $\MMMod{H}{A}{A}$-module, then
$V_\st$ is a  (strong) quasi-commutative $\MMMod{H^\FF}{A_\st}{A_\st}$-module. 

For a triangular Hopf algebra we have $\RR=\RR_{21}^{-1}$. Thus, every quasi-commutative
$\MMMod{H}{A}{A}$-module is automatically strong quasi-commutative.
\end{Example}

\begin{Remark}
We presently do not  have a nontrivial example of a proper quasitriangular (i.e.~not triangular)
Hopf algebra $H$ which has a quasi-commutative $\AAAlg{H}{}{}$-algebra.
\end{Remark}

Tensor products over $A$ of (strong) quasi-commutative $\MMMod{H}{A}{A}$-modules are
again (strong) quasi-commutative $\MMMod{H}{A}{A}$-modules.
\begin{Proposition}\label{qc1qc2thenqc12}
Let $(H,\RR)$  be a quasitriangular Hopf algebra, 
$A$ be a quasi-commutative $\AAAlg{H}{}{}$-algebra
and $V,W$ be two quasi-commutative  $\MMMod{H}{A}{A}$-modules.
Then $V\otimes_AW$ is a quasi-commutative $ \MMMod{H}{A}{A}$-module.

If moreover $V$ and $W$ are strong quasi-commutative $ \MMMod{H}{A}{A}$-modules, then $V\otimes_A W$
is also a strong quasi-commutative  $\MMMod{H}{A}{A}$-module.
\end{Proposition}
\begin{proof}
We show quasi-commutativity of $V\otimes_A W$. For all $v\in V, w\in W, a\in A$,
\begin{flalign}
\nn v\,\otimes_Aw \cdot a&=v\otimes_A (\oR^\al\ra a)\cdot
(\oR_\al\ra w)=(\oR^\be\oR^\al\ra a)\cdot (\oR_\be\ra
v)\otimes_A(\oR_\al\ra w)\\
& =(\oR^\be\ra a)\cdot (\oR_\be\ra (v\otimes_A w))~.
\end{flalign}
The strong quasi-commutativity of $V\otimes_A W$, in case of $V$ and $W$ strong quasi-commutative,
 is proven by considering the equation above with the alternative $\RR$-matrix
$\RR^{-1}_{21}$.
\end{proof}

\begin{Proposition}\label{Prop5.14}
Let $(H,\RR)$  be a quasitriangular Hopf algebra, 
$A$ be a quasi-commutative $\AAAlg{H}{}{}$-algebra
and $V,W$ be two  quasi-commutative  $\MMMod{H}{A}{A}$-modules. Then
$\Hom_A(V,W)$ and $\big({}_A\Hom(V,W)\big)^\op$
are quasi-commutative $\MMMod{H}{A}{A}$-modules. \vspace{1mm}\newline 
If moreover $V$ and $W$ are strong quasi-commutative $\MMMod{H}{A}{A}$-modules, then so are
$\Hom_A(V,W)$ and $\big({}_A\Hom(V,W)\big)^\op$.
\end{Proposition}
\begin{proof}
Recall that the $A$-bimodule structure of $\Hom_A(V,W)$
 is defined by, for all $a\in A, P\in
\Hom_A(V,W),$ 
$P\cdot a=P\circ l_a$ and $a\cdot P=l_a\circ P$. We prove that for $A$, $V$ and $W$ quasi-commutative
\eq\label{PaRaRP}
P\cdot a=(\oR^\al\trgl
a)\cdot(\oR_\al\btrgl P)~.
\en
Indeed, for all $v\in V$ we have
\eqa
(P\cdot a)(v)&=&P(a\cdot v)=P\left( (R_\al\trgl v)\cdot (\R^\al\trgl
  a)\right)
=P(R_\al\trgl v)\cdot (\R^\al\trgl a)\nn\\
&=&(\oR^\be\R^\al\trgl a)\cdot \oR_\be\trgl (P(\R_\al\trgl v))\nn\\
&=&(\oR^\be\oR^\al\trgl a)\cdot \oR_\be\trgl (P(S(\oR_\al)\trgl
v))\nn\\
&=&(\oR^\al\trgl a)\cdot \oR_{\al_1}\trgl (P(S(\oR_{\al_2})\trgl
v))\nn\\
&=&(\oR^\al\trgl a)\cdot (\oR_{\al}\btrgl P)(v)~,
\ena
where in the second equality we used that $a\cdot v=(\R_\al\trgl v)\cdot
(\R^\al\trgl a)$, which is equivalent to (\ref{QCrightact}). Then we have
used the $\RR$-matrix properties  (\ref{Ral1Ral21second}) and (\ref{SidR}).
Quasi-commutativity of the $\MMMod{H}{A}{A}$-module $\big({}_A\Hom(V,W)\big)^\op$
 is similarly proven.

The strong quasi-commutativity of $\Hom_A(V,W)$ and $\big({}_A\Hom(V,W)\big)^\op$ in the
case of strong quasi-commutative $\MMMod{H}{A}{A}$-modules  $V$ and
$W$  is similarly proven considering the alternative $\RR$-matrix $\RR_{21}^{-1}$.
\end{proof}

\begin{Remark}\label{EndQC}
A particular case is when $V=W$,  then $\End_A(V)$
 is a quasi-commutative $\MMMod{H}{A}{A}$-module. However 
the $\AAAlg{H}{A}{A}$-algebra structure of $\End_A(V)$
is in general not quasi-commutative. Similarly 
$({}_A\End(V))^\op$ is quasi-commutative as an $\MMMod{H}{A}{A}$-module but in general not
as an $\AAAlg{H}{A}{A}$-algebra.
\end{Remark}

\sk
\begin{Theorem}\label{propo:RtensorAhom}
Let $(H,\RR)$ be a quasitriangular Hopf algebra, $A$ be a quasi-commutative $\AAAlg{H}{}{}$-algebra, 
$V,\widetilde{V}$ be two $\MMMod{H}{}{A}$-modules
and $W,\widetilde{W}$ be two quasi-commutative $\MMMod{H}{A}{A}$-modules. 
For all $P\in\Hom_A(V,\widetilde{V})$ and $Q\in\Hom_A(W,\widetilde{W})$ 
the map $P\otimes_\RR Q \in \Hom_A(V\otimes W,\widetilde{V}\otimes \widetilde{W})$
induces a well-defined right $A$-linear map (denoted by the same symbol) 
$P\otimes_\RR Q: V\otimes_A W\to \widetilde{V}\otimes_A \widetilde{W}$ 
on the quotient modules. Explicitly we have, for all $v\in V$ and $w\in W$,
\begin{flalign}\label{eqn:Rtensorexplicit}
P\otimes_\RR Q(v\otimes_A w) = P\big(\oR^\alpha\ra v\big)\otimes_A \big(\oR_\alpha\RA Q\big)(w)~.
\end{flalign}
\end{Theorem}
\begin{proof}
Remember that the $\MMMod{H}{}{A}$-module $V\otimes_A W$ was defined as the quotient of the
$\MMMod{H}{}{A}$-module $V\otimes W$ via the $\MMMod{H}{}{A}$-submodule $\mathcal{N}_{V,W}$ generated
by the elements $v\cdot a\otimes w-v\otimes a\cdot w$, for all 
$a\in A,v\in V, w\in W$.
The map $P\otimes_\RR Q :V\otimes W \to \widetilde{V}\otimes \widetilde{W}$ induces a well-defined
map on the quotient modules $V\otimes_A W = V\otimes W/\mathcal{N}_{V,W}$ and 
$\widetilde{V}\otimes_A\widetilde{W} = \widetilde{V}\otimes\widetilde{W}/\mathcal{N}_{\widetilde{V},\widetilde{W}}$,
if the image of $\mathcal{N}_{V,W}$ under $P\otimes_\RR Q$ lies in $\mathcal{N}_{\widetilde{V},\widetilde{W}}$, 
i.e.~$P\otimes_\RR Q \big(\mathcal{N}_{V,W}\big)\subseteq \mathcal{N}_{\widetilde{V},\widetilde{W}}$.
This is indeed the case because, for all $v\in V$, $w\in W$ and $a\in A$,
\begin{flalign}
\nn &P\otimes_\RR Q(v\cdot a\otimes w - v\otimes a\cdot w) \\
\nn &\quad ~=P(\oR^\alpha_1 \ra v\cdot \oR^\alpha_2\ra a)\otimes (\oR_\alpha\RA Q)(w) - P(\oR^\alpha\ra v)\otimes (\oR_\alpha\RA Q)(a\cdot w)\\
\nn & \quad ~=P(\oR^\alpha \ra v)\cdot \oR^\beta\ra a \otimes (\oR_\beta\oR_\alpha\RA Q)(w) - P(\oR^\alpha\ra v)\otimes \oR^\beta \ra a \cdot (\oR_\beta\oR_\alpha\RA Q)( w)\\
& \quad ~\in \mathcal{N}_{\widetilde{V},\widetilde{W}}~,
\end{flalign}
where in the last equality we used Proposition \ref{Prop5.14}.
The right $A$-linearity of the map $P\otimes_\RR Q: V\otimes_A W\to \widetilde{V}\otimes_A \widetilde{W}$
follows from Proposition \ref{RALOR} and the fact that $V\otimes_A W$ (and $\widetilde{V}\otimes_A\widetilde{W}$)
are equipped with the right $A$-module structure canonically induced from $V\otimes W$ 
(and $\widetilde{V}\otimes\widetilde{W}$  respectively).
\end{proof}
Theorem \ref{propo:RtensorAhom} immediately implies that the following definition
  is well-given.
\begin{Definition}\label{defi:RtensorA}
Let $(H,\RR)$ be a quasitriangular Hopf algebra, $A$ be a quasi-commutative $\AAAlg{H}{}{}$-algebra, 
$V,\widetilde{V}$ be two $\MMMod{H}{}{A}$-modules
and $W,\widetilde{W}$ be two quasi-commutative $\MMMod{H}{A}{A}$-modules.
The {\bf $\RR$-tensor product} of right $A$-linear maps is defined by
the $\bbK$-bilinear map
\eqa
 \otimes_\RR:\, \Hom_A(V,\widetilde{V})\times
 \Hom_A(W,\widetilde{W})& \longrightarrow & \Hom_A(V\otimes_A
 W,\widetilde{V}\otimes_A\widetilde{W})\nn\\
(P,Q)&\longmapsto & P\otimes_\RR Q~,
\ena
where $P\otimes_\RR Q$ is explicitly given in (\ref{eqn:Rtensorexplicit}).
\end{Definition}
\sk
\begin{Theorem}\label{RotimesA}
Let $(H,\RR)$ be a quasitriangular Hopf algebra, $A$ be a quasi-commutative $\AAAlg{H}{}{}$-algebra,
 $V,\widetilde V,\widehat V $ be $\MMMod{H}{}{A}$-modules and
 $W,Z,\widetilde{W},\widetilde{Z},\widehat{W}$ be quasi-commutative $\MMMod{H}{A}{A}$-modules.
The $\RR$-tensor product of Definition \ref{defi:RtensorA}
is  compatible with the $\MMMod{H}{}{}$-module structure, i.e.,~for all
$\xi\in H$, $P\in\Hom_A(V,\widetilde{V})$ and $Q\in\Hom_A(W,\widetilde{W})$,
\begin{subequations}
\begin{flalign}
\label{eqn:RtensorHmodA}
\xi\RA (P\otimes_\RR Q) = (\xi_1\RA P)\otimes_\RR (\xi_2\RA Q)~.
\end{flalign}
Furthermore, the $\RR$-tensor product is associative, i.e.,~for all
$P\in\Hom_A(V,\widetilde{V})$, $Q\in\Hom_A(W,\widetilde{W})$ and $T\in\Hom_A(Z,\widetilde{Z})$,
\begin{flalign}
\label{eqn:RtensorassA}
\bigl(P\otimes_\RR Q\bigr)\otimes_\RR T = P\otimes_\RR \bigl(Q\otimes_\RR T\bigr)~,
\end{flalign}
and satisfies the braided composition law, for all $P\in\Hom_A(V,\widetilde{V})$,  $Q\in\Hom_A(W,\widetilde{W})$,
$\widetilde{P}\in\Hom_{A}(\widetilde{V},\widehat{V})$ and $\widetilde{Q}\in\Hom_A(\widetilde{W},\widehat{W})$,
\begin{flalign}
\label{eqn:RtensorcircA}
\bigl(\widetilde{P}\otimes_\RR\widetilde{Q}\bigr)\circ \bigl(P\otimes_\RR Q\bigr) = \bigl(\widetilde{P}\circ (\oR^\alpha\RA P)\bigr)
\otimes_\RR\bigl((\oR_\alpha\RA \widetilde{Q})\circ Q\bigr)~.
\end{flalign}
\end{subequations}
\end{Theorem}
\begin{proof}
These properties follow immediately from Theorem \ref{RotimesK} because all maps in the present theorem are canonically
induced from the maps in Theorem \ref{RotimesK}.
Alternatively, one can repeat the calculations in the proof
of Theorem  \ref{RotimesK} using the explicit expression (\ref{eqn:Rtensorexplicit})
and acting on generating elements $v\otimes_A w\in V\otimes_A W$ or $v\otimes_A w\otimes_A z\in V\otimes_A W\otimes_A Z$.
\end{proof}


\subsection{Deformation}
We study the twist deformation of the tensor 
product over $A$ of  $\MMMod{H}{}{A}$-modules
with $\MMMod{H}{A}{A}$-modules, 
and of  right $A$-linear maps $P\otimes_\RR Q$.
Given  $\MMMod{H}{}{A}$-modules $V,\widetilde{V}$ and  $\MMMod{H}{A}{A}$-modules $W,\widetilde{W}$,  Theorem \ref{Theorem2} implies  that we have the twist deformed $\MMMod{H^\FF}{}{A_\st}$-modules
$V_\st,\widetilde{V}_\st$
and the twist deformed $\MMMod{H^\FF}{A_\st}{A_\st}$-modules $W_\st,\widetilde{W}_\st$, with $A_\st$ being the
twist deformed $\AAAlg{H^\FF}{}{}$-algebra.
  We can therefore consider 
$V_\st\otimes_{A_\st} W_\st$, that is  
the quotient  of the $\bbK$-module $V_\st\otimes_\st W_\st$ via the $\bbK$-submodule $\mathcal{N}^\star_{V_\st,W_\st}$
 generated by the elements $v\star a\otimes_\st w-v\otimes_\st a\st w$, for all 
$a\in A_\st,v\in V_\st, w\in W_\st$.
The image of $v\otimes_\st w$ under the canonical projection $\pi_\st :
V_\st\otimes_\st W_\st\to V_\st\otimes_{A_\st} W_\st$ is denoted by $v\otimes_{A_\st} w$. 
Since the $\bbK$-submodule
$\mathcal{N}^\st_{V_\st,W_\st}$ is also an $\MMMod{H^\FF}{}{A_\star}$-submodule, an
$\MMMod{H^\FF}{}{A_\star}$-module structure on 
$V_\st\otimes_{A_\st}W_\st$ is canonically induced from the one on $V_\st \otimes_\st W_\st$.
Explicitly we have, for all $\xi\in H^\FF, 
a\in A_\st,v\in V_\st, w\in W_\st$, $(v\otimes_{A_\st}w )\star a = v\otimes_{A_\st}(w\star a)$ and 
$\xi\ra_\FF (v\otimes_{A_\st}w) = (\xi_{1_\FF}\ra v)\otimes_{A_\st}(\xi_{2_\FF}\ra w)$.

\begin{Lemma}\label{lemma4}
Let $A$ be an $\AAAlg{H}{}{}$-algebra, $V$ be  an $\MMMod{H}{}{A}$-module
and $W$ be an $\MMMod{H}{A}{A}$-module.
The $\bfK$-linear map  $\varphi_{V,W}=\FF^{-1}\ra: V_\st\otimes_\st
W_\st\to(V\otimes W)_\st$ defined in (\ref{phiVWVW})
is in this case an $\MMMod{H^\FF}{}{A_\star}$-module isomorphism and induces a 
map on the quotients 
according to the following commutative diagram:
\begin{flalign}\label{eqn:lemma4}
\xymatrix{ 
V_\st\otimes_{\star} W_\st \ar[d]_-{\varphi^{}_{V,W}}\ar[rrrr]^-{\pi_\st} & & & &{V_\st}\otimes_{A_\star}{W_\st} \ar[d]^-{^{}\varphi_{{V_\st},{W_\st}}}\\
(V\otimes W)_\st \ar[rrrr]_-{\pi}& & & &({V}\otimes_A{W})_{\st}
}
\end{flalign}
Explicitly, \eqa\label{varphiUVA}
\varphi_{V_\st,W_\st}=\FF^{-1}\trgl\,: V_\st\otimes_{A_\star} W_\st &\to &\,(V\otimes_A W)_\st\nn\\
~v\otimes_{A_\st} w~~&\mapsto &\varphi_{V_\st,W_\st}(v\otimes_{A_\star} w)
=(\of^\alpha\ra v)  \otimes_A (\of_\alpha\ra w)~.
\ena
This map is an isomorphism between the $\MMMod{H^\FF}{}{A_\star}$-modules
$V_\st\otimes_{A_\st} W_\st$ and $(V\otimes_A W)_\st$.
\end{Lemma}
\begin{proof}
The map $\varphi_{V,W}:V_\st\otimes_\st W_\st \to
(V\otimes W)_\st$ was already shown to be an $\MMMod{H^\FF}{}{}$-module homomorphism in
(\ref{varphixivw}).
It is an $\MMMod{H^\FF}{}{A_\star}$-module  homomorphism because, for all $v\in V_\st$, $w\in W_\st$ and $a\in A_\st$,
\eqa
\nn\varphi_{V,W}(v\otimes_\star w\st a)&=&
\of^\alpha \ra v\otimes (\of_{\alpha_1}\of^\beta\ra w)\cdot
(\of_{\alpha_2}\of_\beta\ra a)\\
 &=&\of^{\alpha}{_1} \of^\be\ra v\otimes (\of_{\alpha_2}\of_\beta\ra w)\cdot
(\of_{\alpha}\ra a)\nn\\ &=&\varphi_{V,W}(v\otimes_\st w)\st a~,
\ena
where we used the twist cocycle property (\ref{ass}). 
It canonically induces the map $\varphi_{V_\st, W_\st}$ in
(\ref{eqn:lemma4}) 
because it maps $\mathcal{N}^\st_{V_\st,W_\st}$ 
into  $\mathcal{N}_{V,W}$, i.e.~$\varphi_{V,W}\big(\mathcal{N}^\st_{V_\st,W_\st}\big)\subseteq \mathcal{N}_{V,W}$.
Indeed, for all $v\in V$, $w\in W$ and $a\in A$,
\begin{flalign}
\nn &\varphi_{V,W}(v\star a \otimes_\star w - v\otimes_\st a\star w) \\
&\quad~
=(\of^\alpha_1\of^\beta\ra v)\cdot (\of^\alpha_2\of_\beta\ra a)\otimes \of_\alpha\ra w - \of^\alpha \ra v\otimes (\of_{\alpha_1}\of^\beta\ra a)\cdot (\of_{\alpha_2}\of_\beta\ra w) \in \mathcal{N}_{V,W}~,
\end{flalign}
because of the twist cocycle property (\ref{ass}). The induced map $\varphi_{V_\st,W_\st}$ is an
$\MMMod{H^\FF}{}{A_\star}$-module 
homomorphism because so is $\varphi_{V,W}$, and $\mathcal{N}^\st_{V_\st,W_\st}$,
$(\mathcal{N}_{V,W})_\st$ are $\MMMod{H^\FF}{}{A_\star}$-submodules. (The 
$\MMMod{H^\FF}{}{A_\star}$-module $(\mathcal{N}_{V,W})_\st$ is
obtained by applying Theorem \ref{Theorem2} 
to the $\MMMod{H}{}{A}$-module $\mathcal{N}_{V,W}$. 
The corresponding projection $D_\FF(\pi)=\pi$ is 
$H^\FF$-equivariant and  right $A_\st$-linear.)

 Finally, the map (\ref{varphiUVA}) is obviously invertible with inverse
$\varphi^{-1}_{V_\st,W_\st}:=\FF\trgl$.
\end{proof}
\begin{Remark}
If in the above lemma 
$V$ is an $\MMMod{H}{A}{A}$-module, then
the map $\varphi_{V,W} : V_\st\otimes_\st
W_\st\to(V\otimes W)_\st$ is an $\MMMod{H^\FF}{A_\star}{A_\star}$-module isomorphism.
Furthermore $V_\st\otimes_{A_\st}W_\st$ is an $\MMMod{H^\FF}{A_\st}{A_\st}$-module
with left $A_\st$-module structure explicitly given by, for all $a\in A_\st, v\in V_\st,
w\in W_\st, a\st(v\otimes_{A_\st}w)=(a\st v)\otimes_{A_\st}w$.
In this case the induced map  $\varphi_{V_\st,W_\st}: V_\st\otimes_{A_\st}
W_\st\to(V\otimes_A W)_\st$ is also an $\MMMod{H^\FF}{A_\star}{A_\star}$-module isomorphism.
\end{Remark}
\sk
Let now $A$, $W$ and $\widetilde{W}$ be also quasi-commutative. Then
the commutative diagram (\ref{eqn:promodhomdef}) induces a commutative
diagram on the corresponding quotient modules.

\begin{Theorem}\label{theo:promodhomdefA}
Let $(H,\RR)$ be a quasitriangular Hopf algebra with twist $\FF\in
H\otimes H$,  $A$ be a quasi-commutative $\AAAlg{H}{}{}$-algebra, 
$V,\widetilde{V}$ be two $ \MMMod{H}{}{A}$-modules and
$W,\widetilde{W}$ be two quasi-commutative $\MMMod{H}{A}{A}$-modules. 
Then for all $P\in \Hom_A(V,\widetilde{V})$
and $Q\in\Hom_A(W,\widetilde{W})$
the following diagram of right $A_\st$-module homomorphisms commutes:
\begin{flalign}\label{eqn:promodhomdefA}
\xymatrix{
V_\st\otimes_{A_\star} W_\st \ar[d]_-{\varphi^{}_{V_\st,W_\st}}\ar[rrrr]^-{D_\FF(P)\otimes_{\RR^\FF}D_\FF(Q)} & & & &\widetilde{V_\st}\otimes_{A_\star}\widetilde{W_\st} \ar[d]^-{^{}\varphi_{\widetilde{V_\st},\widetilde{W_\st}}}\\
(V\otimes_A W)_\st \ar[rrrr]_-{D_\FF\bigl((\of^\alpha\RA P)\otimes_\RR (\of_\alpha\RA Q)\bigr)}& & & &(\widetilde{V}\otimes_A\widetilde{W})_{\st}
}
\end{flalign}
\end{Theorem}
\begin{proof}
Theorems \ref{propo:RtensorAhom} and \ref{PDPHom} imply that the horizontal arrows in
the diagram are well-defined right $A_\st$-module homomorphisms. Lemma
\ref{lemma4} states that the vertical arrows are $\MMMod{H^\FF}{}{A_\star}$-module isomorphisms. 
The commutativity of the diagram (\ref{eqn:promodhomdefA})  follows from
the commutativity of the diagram
(\ref{eqn:promodhomdef}) because the maps in (\ref{eqn:promodhomdefA}) 
are all canonically induced
by the maps in (\ref{eqn:promodhomdef}), (cf. (\ref{eqn:lemma4})).
\end{proof}
\begin{Remark}
Theorem
\ref{theo:promodhomdefA} can be interpreted as providing an
equivalence of categories that have a tensor product over $A$ structure.
Indeed, let us consider the category
$\big({\MMMod{H}{A}{A}}^{\mathrm{qc}},\Hom_A,\circ\big)$, where
objects are  quasi-commutative $\MMMod{H}{A}{A}$-modules, morphisms
are right $A$-linear maps and their composition is the usual
composition. We equip this category with an ``almost monoidal'' structure $\otimes_{\RR}$ given on objects
$(V,W)$ by the tensor product over $A$, $V\otimes_A W$, and on morphisms $(P,Q)$ by $P\otimes_\RR Q$.
Lemma \ref{lemma4} implies that the isomorphisms
$\varphi_{V_\star,W_\star}:V_\star\otimes_{A_\star}W_\star\to
(V\otimes_A W)_\star$ satisfy a commutative diagram as in
(\ref{eqn:higheriotast}). Then Theorem
\ref{theo:promodhomdefA} implies that
the equivalence of categories found in Theorem \ref{functortheoremR} extends, as in Remark  \ref{rem:tensorcat},
to an equivalence of the corresponding ``almost monoidal'' categories
$\big({\MMMod{H}{A}{A}}^{\mathrm{qc}},\Hom_A,\circ_\star ,{\otimes_\RR}_\star\big)$ 
and $\big({\MMMod{H^\FF}{A_\star}{A_\star}}^{\mathrm{qc}_\star},\Hom_{A_\star},\circ ,\otimes_{\RR^\FF}\big)$.
\end{Remark}


\subsection{From right to left $A$-linear homomorphisms}

Left $A$-linear homomorphisms and endomorphisms of commutative $A$-bimodules over a commutative algebra 
$A$ are automatically also right $A$-linear and vice versa. If the Hopf algebra $(H, \RR)$ is quasitriangular, 
$A$ is a quasi-commutative $\AAAlg{H}{}{}$-algebra and the $\MMMod{H}{A}{A}$-modules are strong quasi-commutative, we
similarly have an isomorphism between right and left  $A$-linear
homomorphisms and endomorphisms. This can be shown by deforming
$\AAAlg{H}{}{}$-algebras and $\MMMod{H}{A}{A}$-modules 
with the  twist $\FF=\RR$.
\begin{Lemma}\label{HAVop}
Let $(H,\RR)$ be a quasitriangular Hopf algebra with twist $\FF=\RR$,
$A$ be a quasi-commutative $\AAAlg{H}{}{}$-algebra and $V$ be a  strong
quasi-commutative $\MMMod{H}{A}{A}$-module.
Then $\UU^\RR=\UU^\cop$, $A_{\st_\RR}=A^\op$ and
$V_{\st_\RR}=V^\op$.
Here $\st_\RR$ denotes the deformation associated with the twist $\FF=\RR$.
\end{Lemma}
\begin{proof}
The bialgebras $H^\RR$ and $H^\cop$ are the same because the $\RR$-twisted
coproduct is the coopposite coproduct (cf. (\ref{12})). Uniqueness of
the antipode implies that they are the same Hopf algebra.
 
$A_{\st_{\RR_{}}}=A^\op$ as algebras (and hence as $\AAAlg{H^\cop}{}{}$-algebras) because, for all $a, \tilde a\in A$,
$a\st_\RR\tilde a=(\oR^\al\trgl a)\,(\oR_\al\trgl\tilde a)=\tilde a\,
a=\mu^\op(a\otimes \tilde a)$.

Similarly $V_{\st_{\RR_{}}}=V^\op$ as 
$A^\op$-bimodules (and hence as $\MMMod{H^\cop}{A^\op}{A^\op}$-modules)  because,  for all $a\in A, v\in V$, 
\begin{subequations}
\begin{flalign}
 a\st_{\RR_{}} v&=\oR^\al(a)\cdot \oR_\al(v)=v\cdot a=a\cdot^\op v~,\\
v\st_\RR a&=\oR^\al(v)\cdot \oR_\al(a)^{}=a\cdot v=v\cdot^\op a~,
\end{flalign}
\end{subequations}
where in the first equality we have used the quasi-commutativity condition (\ref{QCrightact}) and in the second equality 
 the inverse of the strong quasi-commutativity condition (\ref{QCrightactstrong}).  
\end{proof}
\begin{Remark}\label{rem:otherRop}
Lemma \ref{HAVop} holds true also if we use the alternative $\RR$-matrix $\RR_{21}^{-1}$ as a twist, 
i.e.~if we use $\FF=\RR_{21}^{-1}$.
\end{Remark}
\sk
\begin{Theorem}\label{theo:qclriso} 
Let $(H,\RR)$ be a quasitriangular Hopf algebra with twist $\FF=\RR$,
 $A$ be a quasi-commutative $\AAAlg{H}{}{}$-algebra and $V, W$ be two
strong quasi-commutative $\MMMod{H}{A}{A}$-modules. Then there is an isomorphism
(that with slight abuse of notation we denote)
\begin{flalign}
D_\RR ~:~\big({{\End_A(V)}_{\,\st_\RR}}\big)^\op~~~&\longrightarrow~~~~
\big({}_A\End(V)\big)^\op \nn\\
P~~~~&\longmapsto~~~~ D_\RR(P):=(\oR^\al\btrgl P)\circ \oR_\al\trgl~
\end{flalign} 
between the $\AAAlg{H}{A}{A}$-algebras 
$\big(\End_A(V), {\circ_{\star_\RR}}^\op,{\star_\RR}^\op,{\RA}\big)$
and 
$\big({}_A\End(V),\circ^\op,\cdot^\op,{\RA^\cop}\big)$.

Similarly there is an isomorphism (denoted by the same symbol)
\begin{flalign}
D_\RR ~:~{\Hom_A(V,W)}~~~&\longrightarrow~~~~
\big({}_A\Hom(V,W)\big)^\op \nn\\
P~~~~&\longmapsto ~~~~ D_\RR(P):=(\oR^\al\btrgl P)\circ
\oR_\al\trgl~\label{righttoleftiso}
\end{flalign}
between the $\MMMod{H}{A}{A}$-modules
$\big(\Hom_A(V,W),\cdot,{\RA}\big)$ and
$\big({}_A\Hom(V,W),\cdot^\op,{\RA^\cop}\big)$.
\end{Theorem}
\begin{proof}
From  $H^\RR=H^\cop$ we have equality of the corresponding adjoint
actions $\btrgl_\RR=\btrgl^\cop$. Then 
from $A_{\st_{\RR_{}}}=A^\op$ and $V_{\st_{\RR_{}}}=V^\op$ we have
$\End_{A_{\st_\RR}}(V_{\st_\RR})=\End_{A^\op}(V^\op)$, i.e., more explicitly, 
$(\End_{A_{\st_\RR}}(V_{\st_\RR}), \circ, \cdot, \RA_\RR)=
\big(\End_{A^\op}(V^\op),\circ,\cdot,{\RA^\cop}\big)$ as $\AAAlg{H^\cop}{A^\op}{A^\op}$-algebras,
  where the $A^\op$-bimodule  structure is
the usual one obtained with the left multiplication map, that in this
case is $l^{A^\op}:A^\op\to \End_{A^\op}(V^\op)$.
Recalling the canonical isomorphism $\End_{A^\op}(V^\op)\simeq
{}_A\End(V)$ (cf.  (\ref{identityiso})) we therefore have 
$
(\End_{A_{\st_\RR}}(V_{\st_\RR}), \circ, \cdot, \RA_\RR)\simeq
\big({}_A\End(V),\circ,\cdot,{\RA^\cop}\big)
$.
Use of Lemma \ref{ABVCop} leads to the $\AAAlg{H}{A}{A}$-algebra
isomorphism 
$
\big(\End_{A_{\st_\RR}}(V_{\st_\RR})\big)^\op\simeq({}_A\End(V))^\op 
$.
The isomorphism
$D_\RR  \,:~\End_A(V)_{\,\st_\RR}
\to
\End_{A_{\st_\RR}}(V_{\st_\RR})$
of Theorem \ref{PDP} between the $\AAAlg{H^\RR}{A_{\star_\RR}}{A_{\star_\RR}}$-algebras
$\big(\End_A(V),\circ_{\star_\RR},\star_\RR,{\RA}\big)$
and
$\big(\End_{A_{\star_\RR}}(V_{\star_\RR}),\circ,\cdot,{\RA_\RR}\big)$
then induces the $\AAAlg{H}{A}{A}$-algebra isomorphism 
\begin{flalign}
D_\RR  \,:~\big(\End_A(V)_{\,\st_\RR}\big)^\op\longrightarrow
({}_A\End(V))^\op~
\end{flalign}
between 
$\big(\End_A(V),{\circ_{\star_\RR}}^\op,{\star_\RR}^\op,{\RA}\big)$
and
${\big({}_A\End(V),\circ^\op,\cdot^\op,{\RA^\cop}\big)}$.

The construction of the isomorphism $D_\RR$ for homomorphisms is
similar and leads to an $\MMMod{H}{A}{A}$-module isomorphism
\begin{flalign}
D_\RR  \,:~\big(\Hom_A(V,W)_{\,\st_\RR}\big)^\op\longrightarrow
({}_A\Hom(V,W))^\op ~.
\end{flalign}
We conclude the proof by showing that as $\MMMod{H}{A}{A}$-modules
\begin{flalign}
\big(\Hom_A(V,W)_{\,\st_\RR}\big)^\op=\Hom_A(V,W)~.
 \end{flalign}
We apply Lemma \ref{HAVop} to the strong quasi-commutative $\MMMod{H}{A}{A}$-module
$\Hom_A(V,W)$, and obtain 
$\Hom_A(V,W)_{\st_\RR}=\Hom_A(V,W)^\op$. Therefore,
$\big(\Hom_A(V,W)_{\st_\RR}\big)^\op={\Hom_A(V,W)^\op\,}^\op ={\Hom_A(V,W)}$.
\end{proof}
In Theorems  \ref{PDPHom}  and \ref{PDPLHom} we defined the
deformation maps $D_\FF$ and $D^\cop_\FF$
between  homomorphisms of
$\MMMod{H^\FF}{A_\star}{A_\star}$-modules (see also the corresponding
theorems for endomorphisms). In Theorem
\ref{theo:qclriso} with slight abuse of notation we  have still denoted
by $D_\RR$ the  deformation map that is now  between homomorphisms of
$\MMMod{H}{A}{A}$-modules (rather than
$\MMMod{H^\RR}{A_{\star_\RR}}{A_{\star_\RR}}$-modules). 
\sk
The right to left isomorphism is compatible with twist deformation.
\begin{Theorem}\label{theo:rldefcomp}
Let $(H,\RR)$ be a quasitriangular Hopf algebra with twist $\FF\in H\otimes H$,
 $A$ be a quasi-commutative $\AAAlg{H}{}{}$-algebra and $V, W$ be two
strong quasi-commutative $\MMMod{H}{A}{A}$-modules. Then the following diagram of 
$\MMMod{H^\FF}{A_\star}{A_\star}$-module isomorphisms commutes:
\begin{flalign}\label{eqn:rldefcomp}
\xymatrix{
\Hom_{A}(V,W)_\star\ar[d]_-{D_\RR} \ar[rr]^-{D_\FF} & & \Hom_{A_\star}(V_\star,W_\star)\ar[d]^-{D_{\RR^\FF}}\\
\big({}_A\Hom(V,W)\big)^\op_{~~\star} \ar[rr]^-{D_{\FF}^\cop}& & \big({}_{A_\star}\Hom(V_\star,W_\star)\big)^\op
}
\end{flalign}
\end{Theorem}
\begin{proof}
$H$-equivariance of $D_\RR$  in  (\ref{righttoleftiso}) 
implies that 
$D_\RR:
\Hom_A(V,W)_\star\to \big({}_A\Hom(V,W)\big)^\op_{~~\star} $
is an $\MMMod{H^\FF}{A_\star}{A_\star}$-module  isomorphism ($D_\RR=D_\FF(D_\RR)=D^\cop_\FF(D_\RR)$).

Since all maps in the diagram are $\MMMod{H^\FF}{A_\star}{A_\star}$-module isomorphisms, its
commutativity is proven if we prove the equality 
$D_{\RR^\FF} = D_\FF^\cop \circ D_\RR\circ D_\FF^{-1} $ as
$\bbK$-linear maps.
This immediately follows from  $\RR^\FF = \FF_{21}\,\RR\,\FF^{-1}$  and
the equality
$D_{\FF^\prime\!\FF}=D_{\FF^\prime}\circ D_\FF$, where $\FF$ is a twist of $H$, $\FF'$
is a twist of $H^\FF$, and hence, as it is easily seen, $\FF^{\;\!\prime}\!\FF$ is a
twist of $H$.  We obtain, for all $P\in \Hom_A(V,W)$,
\begin{flalign}
\nn D_{\FF^\prime\!\FF}(P)&=(\of^\al\overline{\f^{\!\prime}}^\be\!\RA P)\circ
\of_\al\overline{\f^{\!\prime}}_\beta\ra=D_\FF(\overline{\f^{\!\prime}}^\be\!\RA P)\circ
\overline{\f^{\!\prime}}_\beta\ra\\
&=\big(\overline{\f^{\!\prime}}^\be\!\RA_\FF
D_\FF(P)\big)\circ\overline{\f^{\!\prime}}_\beta\ra=D_{\FF^\prime}(D_\FF(P))~,
\end{flalign}
where we  used that $D_{\FF}$ intertwines  between the $\RA$
and $\RA_\FF$ action.
\end{proof}

\begin{Example}
 Let $(H,\RR)$ be a quasitriangular Hopf algebra,  $A$ be a quasi-commutative  $\AAAlg{H}{}{}$-algebra
 and $V$ be a strong quasi-commutative $\MMMod{H}{A}{A}$-module.
 Then by Theorem \ref{theo:qclriso} there is an $\MMMod{H}{A}{A}$-module
 isomorphism between the right dual $V^\prime = \Hom_A(V,A)$
and the left dual $^\prime V = \bigl({_A}\Hom(V,A)\bigr)^\op$.
\end{Example}



\section{\label{sec:connections}Connections}
We review the notion of connections on $\MMMod{}{}{A}$-modules and $\MMMod{}{A}{}$-modules.
For $\AAAlg{H}{}{}$-algebras $A$ equipped with a suitable $H$-covariant differential
calculus and $\MMMod{H}{}{A}$-modules (and also
$\MMMod{H}{A}{}$-modules or 
$\MMMod{H}{A}{A}$-modules) we prove that
there is a bijective correspondence between connections on the undeformed and 
deformed modules.
It is an isomorphism between the undeformed and deformed affine spaces
of connections.
We then investigate the problem of constructing connections on tensor products of $\MMMod{H}{A}{A}$-modules
from connections on the individual factors. Assuming,
as  in the study of tensor product module homomorphisms, quasi-commutativity of the $\AAAlg{H}{}{}$-algebras and
$\MMMod{H}{A}{A}$-modules, 
we define a sum of arbitrary (i.e.~not necessarily $H$-equivariant) connections which yields a connection
on the tensor product module. The twist deformation of this sum is investigated in detail. 
As in the case of module homomorphisms we can use the twist $\FF=\RR$ in order to identify
right with left connections on strong quasi-commutative $\MMMod{H}{A}{A}$-modules. Finally, the construction
and twist deformation of connections on the dual module is studied and
shown to be canonical. An extension 
of connections to the tensor algebra of a module and its dual is given.

\subsection{\label{subsec:conbas}Connections on right and left modules}
We briefly review the notion of a connection on an $\MMMod{}{}{A}$-module  or $\MMMod{}{A}{}$-module, 
see e.g.~\cite{Madore:2000aq,DuViLecture} for an introduction.

\sk
\begin{Definition}
Let $A$ be an algebra. 
A {\bf differential calculus} $\bigl(\Omega^\bullet,\wedge,\dif\bigr)$
over $A$ (or an {\bf{$\mathbb{N}^0$-differential graded algebra}}) is an $\mathbb{N}^0$-graded algebra 
$\bigl(\Omega^\bullet \!= \!\bigoplus_{n\geq0}
\Omega^n,\wedge\bigr)$, where $\Omega^0=A$ has degree
zero, together with a
$\bfK$-linear map $\dif:\Omega^\bullet \to \Omega^\bullet$ of degree
one, satisfying $\dif\circ\dif=0$ and the graded Leibniz rule
\begin{flalign}
 \dif(\omega\wedge\omega^\prime) = (\dif\omega)\wedge\omega^\prime + (-1)^{\deg(\omega)}\,\omega\wedge(\dif\omega^\prime)~,
\end{flalign}
for all $\omega,\omega^\prime\in\Omega^\bullet$ with $\omega$ of
homogeneous degree. 
\end{Definition}
\sk
The differential $\dif$ and the product $\wedge$ give rise to
$\bfK$-linear maps (denoted 
by the same symbols) 
$\dif:\Omega^n\to\Omega^{n+1}$ and $\wedge:\Omega^n\otimes \Omega^m\to\Omega^{n+m}$.
Note that in the hypotheses above the $\bfK$-modules $\Omega^n$ are 
$\MMMod{}{A}{A}$-modules.
As in commutative differential geometry  we call $\Omega^n$ the module of
$n$-forms, notice however that our wedge product $\wedge$ is not
necessarily graded commutative. 
We also assume that any $1$-form $\theta\in \Omega :=\Omega^1$ can be written as
$\theta=\sum_i a_i\dif b_i$, with $a_i,b_i\in A$, i.e.~that
exact $1$-forms generate $\Omega$ as an $\MMMod{}{A}{}$-module.

\begin{Example}\label{DGAEX}
 Let $M$ be a $D$-dimensional smooth (second countable) manifold and let $A=C^\infty(M)$ be the smooth 
and complex (or real) valued functions on $M$.
The exterior algebra of differential forms $\bigl(\Omega^\bullet\!=\!\bigoplus_{n\geq 0}\Omega^n,\wedge\bigr)$ 
is an $\mathbb{N}^0$-graded algebra over $\bbC$ (or $\bbR$),
where $\Omega^0=A$ and $\Omega^{n}=0$, for all $n>D$. 
The exterior differential $\dif$ is a differential on $\bigl(\Omega^\bullet,\wedge\bigr)$,
leading to the de Rham differential calculus $\bigl(\Omega^\bullet,\wedge,\dif\bigr)$.
In this special case $\Omega^\bullet$ is graded commutative.

 Another example is given by the twist deformed differential calculus
$\bigl(\Omega^\bullet[[h]],\wedge_\star,\dif\bigr)$ \cite{Aschieri:2005zs}.
There, the algebra
$\bigl(\Omega^\bullet[[h]],\wedge_\star\bigr)$ over
$\bfK=\bbC[[h]]$ is graded quasi-commutative, i.e., for all $\omega,\omega^\prime\in
\Omega^\bullet$ of homogeneous degree,
\begin{flalign}
 \omega\wedge_\star \omega^\prime = 
(-1)^{\deg(\omega)\,\deg(\omega^\prime)} ~(\oR^\alpha\ra \omega^\prime)\wedge_\star (\oR_\alpha\ra \omega)~.
\end{flalign}
\end{Example}
\sk

\begin{Definition}\label{defi:connection}
 Let $A$ be an algebra and $\bigl(\Omega^\bullet,\wedge,\dif\bigr)$
be a differential calculus over $A$.
A {\bf connection on an $\MMMod{}{}{A}$-module} $V$ is a $\bfK$-linear map $\dd:V\to V\otimes_A\Omega$, satisfying
the right Leibniz rule, for all $v\in V$ and $a\in A$,
\begin{flalign}
\label{eqn:rightcon}
 \dd(v\cdot a) = (\dd v)\cdot a + v\otimes_A \dif a~.
\end{flalign}
Similarly, a {\bf connection on an $\MMMod{}{A}{}$-module} $V$ is a $\bfK$-linear map 
$\dd:V\to \Omega\otimes_A V$, satisfying the left Leibniz rule, for all $v\in V$ and $a\in A$,
\begin{flalign}
\label{eqn:leftcon}
\dd(a\cdot v) = a\cdot (\dd v)  + \dif a\otimes_A v~.
\end{flalign}
In case $V$ is an $\MMMod{}{A}{A}$-module we say that a $\bfK$-linear map
$\dd:V\to V\otimes_A\Omega$ is a {\bf right connection} on $V$  if (\ref{eqn:rightcon})
is satisfied. Similarly, we say that a $\bfK$-linear map $\dd: V\to \Omega\otimes_A V$
is a {\bf left connection} on $V$ if (\ref{eqn:leftcon}) is satisfied.
\end{Definition}
 \sk
 
We denote by $\Con_A(V)$ the set of all connections on an $\MMMod{}{}{A}$-module  $V$
and by ${_A}\Con(V)$ the set of all connections on an $\MMMod{}{A}{}$-module $V$.
We also denote by $\Con_A(V)$ and ${_A}\Con(V)$, respectively, the set of all right and left
connections on an $\MMMod{}{A}{A}$-module $V$.
Note that given any connection $\dd \in \Con_A(V)$ and any right $A$-linear map
$P\in\Hom_A(V,V\otimes_A\Omega)$, then $\widetilde \dd = \dd+ P \in\Con_A(V)$ is again a connection. Indeed, for all $a\in A$
and $v\in V$,
\begin{flalign}
 \widetilde \dd(v\cdot a) = \dd(v\cdot a) + P(v\cdot a) = (\dd v)\cdot a + P(v)\cdot a + v\otimes_A\dif a
= (\widetilde \dd v)\cdot a + v\otimes_A \dif a~.
\end{flalign}
The  action $\dd\mapsto \dd +P$ is free and transitive and
hence $\Con_A(V)$ is an affine space over the $\bbK$-module $\Hom_A(V,V\otimes_A\Omega)$.
Similarly, the space of left connections ${_A}\Con(V)$ is an affine space over 
the $\bbK$-module $_A\Hom(V,\Omega\otimes_A V)$. 


\subsection{Deformation of connections}
We now consider differential calculi that are also $H$-covariant
and study the twist deformation of connections on $\MMMod{H}{A}{}$-modules and
$\MMMod{H}{}{A}$-modules.

Let $H$ be a Hopf algebra and $A$ be an $\AAAlg{H}{}{}$-algebra.
Let further $\bigl(\Omega^\bullet,\wedge,\dif\bigr)$ be a {\bf left $H$-covariant differential calculus} over $A$,
i.e.,$\;\Omega^\bullet$ is an $\AAAlg{H}{}{}$-algebra,
the $H$-action $\ra$ is degree preserving and
the  differential is $H$-equivariant, for all $\xi\in H$ and $\omega\in \Omega^\bullet$,
\begin{flalign}
\label{eqn:equivar}
 \xi\ra (\dif\omega) = \dif(\xi\ra \omega)~.
\end{flalign}
Since the $H$-action is degree preserving, $\Omega^n$ are $\MMMod{H}{A}{A}$-modules, for all $n\in\mathbb{N}^0$.

We now show that a left $H$-covariant
differential calculus can be deformed to yield a left $H^\FF$-covariant differential calculus.
\begin{Proposition}\label{lem:dcdef}
Let $H$ be a Hopf algebra with twist $\FF\in H\otimes H$, $A$ be an $\AAAlg{H}{}{}$-algebra
and $\bigl(\Omega^\bullet,\wedge,\dif\bigr)$ be a left $H$-covariant differential calculus over $A$.
Then $\bigl(\Omega^\bullet,\wedge_\star,\dif\bigr)$ is a left $H^\FF$-covariant differential calculus
over the $\AAAlg{H^\FF}{}{}$-algebra $A_\star$.
\end{Proposition}
\begin{proof}
 By Theorem \ref{Theorem1} $\bigl(\Omega^\bullet,\wedge_\star\bigr)$ is an $\AAAlg{H^\FF}{}{}$-algebra.
It is $\mathbb{N}^0$-graded and we have $(\Omega^0,\wedge_\star) = A_\star$.
Due to the $H$-equivariance of the differential, $\dif$ is also a differential on $\bigl(\Omega^\bullet,\wedge_\star\bigr)$.
\end{proof}
Notice that due to $H$-equivariance of the differential $\dif$ the de Rham complex is invariant 
under twist deformation.

\subsubsection*{Right modules}
Let $V$ be an $\MMMod{H}{}{A}$-module.
Since $\Con_A(V)\subseteq \Hom_\bfK(V,V\otimes_A\Omega)$ we can act with the $H$-adjoint 
action $\RA$ on $\dd\in \Con_A(V)$, for all $\xi\in H$,
\begin{flalign}
\label{eqn:conadjoint}
 \xi\RA\dd := \xi_1\ra\,\circ \dd \circ S(\xi_2)\ra\,~.
\end{flalign}
The element $\xi\RA\dd\in \Hom_\bfK(V,V\otimes_A\Omega)$ satisfies,
for all $v\in V$ and $a\in A$,
\begin{flalign}
 \nn(\xi\RA\dd) (v\cdot a) &= \xi_1\ra\bigl(\dd(S(\xi_2)_1\ra v\cdot S(\xi_2)_2\ra a)\bigr)\\
\nn&= \xi_1\ra\Bigl(\bigl(\dd (S(\xi_3)\ra v)\bigr) \cdot (S(\xi_2)\ra a) + (S(\xi_3)\ra v)\otimes_A \dif (S(\xi_2)\ra a)\Bigr)\\
\label{eqn:Hcon}&=\bigl((\xi\RA \dd) (v)\bigr)\cdot a + \varepsilon(\xi)\,v\otimes_A \dif a~.
\end{flalign}
In particular we see that if $\varepsilon(\xi)=0$ then
$\xi\RA \dd \in\Hom_A(V,V\otimes_A \Omega)$, while if  $\varepsilon(\xi)=1$ then $\xi\RA\dd\in\Con_A(V)$.
\sk
We now show that given a twist $\FF\in H\otimes H$ of the Hopf algebra
$H$, then there is an isomorphism
$\Con_A(V)\simeq \Con_{A_\star}(V_\star)$, 
between connections on the 
undeformed $\MMMod{H}{}{A}$-module $V$ and on the deformed $\MMMod{H^\FF}{}{A_\st}$-module
$V_\star$.

We first observe that by composing any $\bbK$-linear map
$V_\star
\to (V\otimes_A \Omega)_\star$
with the $\MMMod{H^\FF}{}{A_\star}$-module isomorphism 
\begin{flalign}\label{varphiagain}
\varphi^{-1}~:~(V\otimes_A\Omega)_\star \longrightarrow V_\star\otimes_{A_\star}\Omega_\star
\end{flalign}
studied in Lemma \ref{lemma4}  (for a simpler notation we drop the module indices on $\varphi$)
 we obtain the $\MMMod{H^\FF}{}{}$-module  isomorphism (that with
abuse of notation we still denote by $\varphi^{-1}$)
\begin{flalign}
\label{isoHomtens} 
\varphi^{-1} ~:~
\Hom_\bbK(V_\star,\, (V\otimes_A\Omega)_\star) \longrightarrow
\Hom_{\bbK}(V_\star,\,V_\star\otimes_{A_\star}\Omega_\star)  ~.
\end{flalign}
Composition of the deformation map 
$D_\FF ~:~\Hom_\bbK(V, \,V\otimes_{A}\Omega)_\star\longrightarrow
\Hom_{\bbK}(V_\star,\, (V\otimes_A\Omega)_\star)$ with
 this isomorphism gives the  $\MMMod{H^\FF}{}{}$-module isomorphism
\begin{flalign}\label{DtildeFHOM}
\widetilde{D}_\FF :=\varphi^{-1}\!\circ\! D_\FF~:~
\Hom_\bbK(V, V\otimes_A \Omega)_\star\longrightarrow
\Hom_{\bbK}(V_\star, V_\star\otimes_{A_\star} \Omega_\star)~.
\end{flalign}

Next we define $\Con_A(V)_\star$ to be the same set of connections as 
$\Con_A(V)$, but with affine space structure over the $\MMMod{H^\FF}{}{}$-module
$\Hom_{A}(V, V\otimes_A\Omega)_\star$ rather than over the $\MMMod{H}{}{}$-module
$\Hom_{A}(V, V\otimes_A\Omega)$ (we recall that 
they coincide as $\bbK$-modules).

\begin{Theorem}\label{theo:condef}
Let $H$ be a Hopf algebra with twist $\FF\in H\otimes H$, $A$ be an
$\AAAlg{H}{}{}$-algebra, $V$ be an $\MMMod{H}{}{A}$-module
and $\bigl(\Omega^\bullet,\wedge,\dif\bigr)$ be a left $H$-covariant differential calculus over $A$. The $\MMMod{H^\FF}{}{}$-module isomorphism (\ref{DtildeFHOM})
restricts to the $\MMMod{H^\FF}{}{}$-module   isomorphism
\begin{flalign}
\nn \widetilde{D}_\FF~:~
\Hom_A(V, V\otimes_A \Omega)_\star&\longrightarrow
\Hom_{A_\star}(V_\star, V_\star\otimes_{A_\star} \Omega_\star)~\\
P~&\longmapsto\varphi^{-1} \circ (\of^\al\btrgl P)\circ \of_\al\ra\,~\label{tildeDFHOMA}
\end{flalign}
and to the affine space  isomorphism 
\begin{flalign}
\nn\widetilde{D}_\FF ~:~
\Con_A(V)_\star&\longrightarrow
\Con_{A_\star}(V_\star)~\\
\label{tildeDFHOMAproof}\dd~&\longmapsto\varphi^{-1} \circ (\of^\al\btrgl \dd)\circ \of_\al\ra\,~,
\end{flalign}
where
$\Con_A(V)_\star$ and $\Con_{A_\star}(V_\star)$ are
respectively affine spaces over the isomorphic $\MMMod{H^\FF}{}{}$-modules 
$\Hom_A(V, V\otimes_A \Omega)_\st$ and 
$\Hom_{A_\star}(V_\star, V_\star\otimes_{A_\star} \Omega_\star)$.
\end{Theorem}
\begin{proof}
From Lemma \ref{lemma4}  we know that the map (\ref{varphiagain})
is an $\MMMod{H^\FF}{}{A_\star}$-module  isomorphism. It then follows that the $\MMMod{H^\FF}{}{}$-module isomorphism
$\varphi^{-1}$ in  (\ref{isoHomtens})  restricts  to an $\MMMod{H^\FF}{}{}$-module isomorphism
(still denoted by $\varphi^{-1}$),
\begin{flalign}
\varphi^{-1} ~:~ 
\Hom_{{A_\star}} (V_\star,\, (V\otimes_A\Omega)_\star) \longrightarrow
\Hom_{{A_\star}} (V_\star, \,V_\star\otimes_{A_\star}\Omega_\star) ~.
\end{flalign}
Henceforth, as in the case of the isomorphism $D_\FF$ (cf. Theorem \ref{PDP}),  the isomorphism 
$\widetilde{D}{_\FF}$ in (\ref{DtildeFHOM}) restricts
to the $\MMMod{H^\FF}{}{}$-module isomorphism (still denoted $\widetilde{D}_\FF$) 
\begin{flalign}
\widetilde{D}_\FF
~:~
\Hom_{A}(V, V\otimes_A \Omega)_\star\longrightarrow
\Hom_{A_\star}(V_\star, V_\star\otimes_{A_\star} \Omega_\star)~.
\end{flalign}

The proof that  $\widetilde{D}_\FF$ in (\ref{tildeDFHOMAproof})
maps connections into connections is similar to the proof
of the right $A_\star$-linearity of the map $D_\FF(P)$, when $P$ is a right
$A$-linear map (cf.~Theorem \ref{PDP}):
Let $\dd\in\Con_A(V)_\star$, then, for all $v\in V$ and $a\in A$,
\begin{flalign}
\nn D_\FF(\dd)(v\star a) 
&= (\of^\alpha\RA\dd)\Bigl((\of_{\alpha_1}\of^\beta\ra
v)\cdot (\of_{\alpha_2}\of_\beta\ra a)\Bigr)\\
\nn &=\bigl((\of^\alpha\RA\dd)(\of_{\alpha_1}\of^\beta\ra v)\bigr) \cdot (\of_{\alpha_2}\of_\beta\ra a)
+\varepsilon(\of^\alpha)\,(\of_{\alpha_1}\of^\beta\ra v) \otimes_A \dif(\of_{\alpha_2}\of_\beta\ra a)~\\
\nn &=\bigl((\of^\alpha_1\of^\beta\RA\dd)(\of^{\alpha}_2\of_\beta\ra v)\bigr) \cdot (\of_{\alpha}\ra a)
+(\of^\beta\ra v) \otimes_A \dif(\of_\beta\ra a)\\
&= D_\FF(\dd)(v)\star a +(\of^\beta\ra v) \otimes_A (\of_\beta\ra \dif a)~.\label{Dnabvsta}
\end{flalign}
In the second line we have used (\ref{eqn:Hcon}),
in the third line the twist cocycle property (\ref{ass}) and in the last line (\ref{eqn:equivar}) and (\ref{eqn:conadjoint}).
Applying $\varphi^{-1}$ we obtain
\begin{flalign}
 \widetilde{D}_\FF(\nabla)(v\star a) = \widetilde{D}_\FF(\dd)(v)\star a +v\otimes_{A_\star} \dif a~.
\end{flalign}
The property $\widetilde{D}_\FF^{-1}(\dd_\star)\in\Con_A(V)_\star$,
for all $\dd_\star\in\Con_{A_\star}(V_\star)$, follows from twisting
back  $H^\FF$ and all its modules to the original undeformed structures
via the twist $\mathcal{F}^{-1}$.

Finally $\widetilde{D}_\FF$ is an affine space isomorphism because its
$\bbK$-linearity
implies that, for all
$\dd\in \Con_A(V)_\star$, $P\in \Hom_A(V,\,V\otimes_A\Omega)_\star$,
 $\,\widetilde{D}_\FF(\dd+P)= \widetilde{D}_\FF(\dd)+\widetilde{D}_\FF(P)\in \Con_{A_\star}(V_\star)$.
\end{proof}

If we forget the $\MMMod{H}{}{}$-module and $\MMMod{H^\FF}{}{}$-module structures,
the $\bbK$-modules $\Hom_{A}(V, V\otimes_A\Omega)$  
and $\Hom_{A}(V, V\otimes_A\Omega)_\star$  coincide, and henceforth $Con_A(V)$ and
$Con_A(V)_\star$ coincide as affine spaces.
Theorem \ref{theo:condef} then implies the 
isomorphism $\Con_A(V)\simeq \Con_{A_\star}(V_\star)$ between the
affine space of connections $\Con_A(V)$ over 
the $\bbK$-module $\Hom_{A}(V, V\otimes_A\Omega)$ 
and the affine space of connections $\Con_{A_\star}(V_\star)$ over
the $\bbK$-module $\Hom_{A_\star}(V_\star,V_\star\otimes_{A_\star}\Omega_\star)$.

\subsubsection*{Left modules}
As in Theorem \ref{theo:condef} we have an isomorphism
${}_A\Con(V)\simeq{}_{A_\star}\Con(V_\star)$ between the affine spaces
of connections on an $\MMMod{H}{A}{}$-module
$V$ and on the  deformed 
$\MMMod{H^\FF}{A_\star}{}$-module $V_\star$.
In this case we consider the $\MMMod{H^\FF}{A_\star}{}$-module isomorphism
$
\varphi^{-1}:(\Omega \otimes_A V)_\star \longrightarrow
\Omega_\star \otimes_{A_\star}V_\star
$
and its lift
$
\varphi^{-1} : \big(\Hom_\bbK(V_\star,\, (\Omega \otimes_A V)_\star)\big)^\op \longrightarrow
\big(\Hom_{\bbK}(V_\star,\,\Omega_\star \otimes_{A_\star}V_\star)  \big)^\op
$. Using also the $D_\FF^\cop$ isomorphism of Theorem \ref{PDPLHom} and
denoting by ${}_A\Con(V)_\star$ the set of connections
${}_A\Con_{A}(V)$ when seen as an affine space over the $\MMMod{H^\FF}{}{}$-module $
\big({}_A\Hom(V,V\otimes_A\Omega)\big)^\op_{~~\star} $, rather than over the $\MMMod{H}{}{}$-module  $
\big({}_A\Hom(V,V\otimes_A\Omega)\big)^\op$,
we obtain
\begin{Theorem}\label{theo:condefleft}
 Let $H$ be a Hopf algebra with twist $\FF\in H\otimes H$, $A$ be an $\AAAlg{H}{}{}$-algebra, $V$
 be an $\MMMod{H}{A}{}$-module and
$\bigl(\Omega^\bullet,\wedge,\dif\bigr)$ be a left $H$-covariant differential calculus 
over $A_{\:\!}$.
Then the $\MMMod{H^\FF}{}{}$-module isomorphism 
\begin{flalign}
\nn \widetilde{D}^\cop_\FF :=\varphi^{-1}\!\circ\! D^\cop_\FF~:~
\big(\Hom_\bbK(V, \Omega \otimes_A V)\big)^\op_{~~\star}&\longrightarrow
\big(\Hom_{\bbK}(V_\star, \Omega_\star\otimes_{A_\star} V_\star)\big)^\op~\\
P~&\longmapsto\varphi^{-1} \circ (\of_\al\btrgl^\cop P)\circ \of^\al\ra\,~
\end{flalign}
restricts to the $\MMMod{H^\FF}{}{}$-module isomorphism 
\begin{flalign}
\widetilde{D}^\cop_\FF~:~
\big({}_A\Hom(V, \Omega \otimes_A V)\big)^\op_{~~\star}\longrightarrow
\big({}_{A_\star}\Hom(V_\star, \Omega_\star \otimes_{A_\star} V_\star)\big)^\op~\label{tildeDFHOMAleft}
\end{flalign}
and to the affine space  isomorphism 
\begin{flalign}
\widetilde{D}^\cop_\FF ~:~
{}_A\Con(V)_\star\longrightarrow
{}_{A_\star}\Con(V_\star)~,
\end{flalign}
where
${}_A\Con(V)_\star$ and ${}_{A_\star}\Con(V_\star)$ are
respectively affine spaces over the isomorphic $\MMMod{H^\FF}{}{}$-modules
$\big({}_A\Hom(V, \Omega\otimes_A V)\big)^\op_{~~\star}$ and 
$\big({}_{A_\star}\Hom(V_\star, \Omega_\star\otimes_{A_\star} V_\star)\big)^\op$.
\end{Theorem}

\subsubsection*{Bimodules}
Let $V$ be an $\MMMod{H}{A}{A}$-module, then $\Hom_A(V,V\otimes_A\Omega)$ and
$\big({}_A\Hom(V,\Omega\otimes_A V)\big)^\op$ are $\MMMod{H}{A}{A}$-modules.
Consequently, $\Con_A(V)_\star$, $\Con_{A_\star}(V_\star)$, ${}_A\Con(V)_\star$ and ${}_{A_\star}\Con(V_\star)$
are affine spaces over the $\MMMod{H^\FF}{A_\star}{A_\star}$-modules
$\Hom_A(V,V\otimes_A\Omega)_\star$, $\Hom_{A_\star}(V_\star,V_\star\otimes_{A_\star}\Omega_\star)$,
$\big({}_A\Hom(V,\Omega\otimes_A V)\big)^\op_{~~\star}$ 
and $\big({}_{A_\star}\Hom(V_\star,\Omega_\star\otimes_{A_\star} V_\star)\big)^\op$, respectively.
The isomorphisms (\ref{tildeDFHOMA}) and (\ref{tildeDFHOMAleft}) are $\MMMod{H^\FF}{A_\star}{A_\star}$-module
isomorphisms, and hence $\widetilde{D}_\FF:\Con_A(V)_\star\to \Con_{A_\star}(V_\star)$ (cf.~Theorem \ref{theo:condef})
 and $\widetilde{D}_\FF^\cop:{}_A\Con(V)_\star\to {}_{A_\star}\Con(V_\star)$
 (cf.~Theorem \ref{theo:condefleft}) are affine space isomorphisms compatible with these extra 
 $\MMMod{H^\FF}{A_\star}{A_\star}$-module structures.


\subsection{Connections on tensor product modules (sum of connections)}
We now investigate the extension of connections to tensor product
modules. We mention that
this is relevant in the formulation of noncommutative gravity and gauge
theories, since it provides a construction principle for connections
on deformed tensor fields in terms of a fundamental connection on 
deformed vector fields, and a minimal coupling prescription for
noncommutative fields with different charges.

The other ingredient in order to construct connections on arbitrary
covariant and contravariant tensor fields is the extension of a
connection on a module to the dual module. 
This will be discussed later in Subsection \ref{sec:leftrightcon}.

\sk
Let $(H,\RR)$ be a quasitriangular Hopf algebra.
A  left $H$-covariant differential calculus 
$\bigl(\Omega^\bullet,\wedge,\dif\bigr)$  over an $\AAAlg{H}{}{}$-algebra $A$ is called 
{\bf graded quasi-commutative} if the $\AAAlg{H}{}{}$-algebra $\Omega^\bullet$ is
graded quasi-commutative, i.e., for all $\omega,\omega^\prime\in \Omega^\bullet$ of
homogeneous degree,
\eq\label{gradedqc}
\omega\wedge\omega^\prime=(-1)^{\deg(\omega)\deg(\omega^\prime)} (\oR^\al\ra\omega^\prime)\wedge
(\oR_\al\ra \omega)~.
\en
Notice that graded quasi-commutativity of the algebra
$(\Omega^\bullet,\wedge)$ in particular implies quasi-commutativity of the $\AAAlg{H}{}{}$-algebra $A$ and 
strong quasi-commutativity of the $\MMMod{H}{A}{A}$-module
of one-forms $\Omega$ (and also of all other $\MMMod{H}{A}{A}$-modules of $n$-forms $\Omega^n$, $n\geq 1$). 

Vice versa, it can be  shown that a left $H$-covariant differential calculus 
$\bigl(\Omega^\bullet,\wedge,\dif\bigr)$ over a quasi-commutative
$\AAAlg{H}{}{}$-algebra $A$ is graded quasi-commutative if the $\MMMod{H}{A}{A}$-module of one-forms $\Omega$ is 
quasi-commutative and generates $\Omega^n$ for all $n> 1$.\footnote{ 
Hint: 1) Show that $\Omega$ is strong quasi-commutative. 
2) Show that (\ref{gradedqc}) holds for $\omega=\dif a$ and $\omega^\prime=\dif b$ with
$a,b\in A$, then extend the result to arbitrary $\omega,\omega^\prime\in \Omega$.
3) Show that if (\ref{gradedqc})  holds for $\omega,\omega^\prime\in\Omega^\bullet$ with 
$\deg(\omega),\deg(\omega^\prime)\leq n$, then it also holds for 
$\omega,\omega^\prime\in\Omega^\bullet$ with 
$\deg(\omega),\deg(\omega^\prime)\leq n+1$.
}
 \sk
Examples of  graded quasi-commutative left $H$-covariant differential calculi
are presented in  Example \ref{DGAEX}.
\begin{Lemma}\label{nablaaw}
Let $A$ be a quasi-commutative $\AAAlg{H}{}{}$-algebra, $W$ be a quasi-commutative $\MMMod{H}{A}{A}$-module and 
$\Omega$ be a strong quasi-commutative $\MMMod{H}{A}{A}$-module. 
Then the braiding map  $\tau_\RR: W\otimes\Omega\rightarrow
\Omega\otimes W$ 
(defined in (\ref{eqn:Rflipmap})) canonically induces an $\MMMod{H}{A}{A}$-module
isomorphism
on the quotient (that with slight abuse of notation we still call $\tau_\RR$),
\begin{flalign}
\tau_\RR: W\otimes_A\Omega\longrightarrow \Omega\otimes_A W ~.
\end{flalign}
Explicitly, for all $w\in W$ and $\theta\in \Omega$,
\begin{flalign}\label{tauRAexplicit}
\tau_\RR(w\otimes_A\theta) =(\oR^\alpha\ra \theta) \otimes_A (\oR_\alpha\ra w)~.
\end{flalign}

We call the isomorphism $\tau_\RR$ a braiding map because it braids
the strong quasi-commutative $\MMMod{H}{A}{A}$-module
$\Omega$ with the quasi-commutative $\MMMod{H}{A}{A}$-modules
$V,W$,
\eq\label{hexagonA1}
\tau_{\RR\;V\otimes_A W,\Omega}=(\tau_{\RR\;V,\Omega}\otimes_\RR \id_W)\circ 
(\id_V\otimes_\RR \tau_{\RR\; W,\Omega})~.
\en
Furthermore, 
\eq\label{hexagonA2}
\tau_{\RR\;V,W\otimes_A \Omega}=(\id_W\otimes_\RR \tau_{\RR\;V,\Omega})\circ (\tau_{\RR\; V,W}\otimes_\RR\id_\Omega)~,
\en
if also the $\MMMod{H}{A}{A}$-module $W$ is strong quasi-commutative.
\end{Lemma}
\begin{proof} 
Since  $\tau_\RR:W\otimes \Omega \to\Omega\otimes W$ satisfies
$\tau_\RR\big(\mathcal{N}_{W,\Omega}\big)\subseteq
\mathcal{N}_{\Omega,W}$ the induced map 
$\tau_\RR:W\otimes_A \Omega
\to\Omega\otimes_A W$
is well-defined.
Since  $\tau_\RR:W\otimes \Omega \to\Omega\otimes W$  is an $\MMMod{H}{}{}$-module
 homomorphism and
$\mathcal{N}_{W,\Omega}\subset W\otimes \Omega$,
$\mathcal{N}_{\Omega,W}\subset \Omega\otimes W$ are  $\MMMod{H}{}{}$-submodules,
also  the induced map $\tau_\RR:W\otimes_A \Omega
\to\Omega\otimes_A W$ is an  $\MMMod{H}{}{}$-module homomorphism.
It is an $\MMMod{H}{A}{A}$-module homomorphism as it can be easily checked 
 using (\ref{tauRAexplicit}),  quasi-commutativity of $W$ and strong quasi-commutativity of $\Omega$.

Consider now the map
$
\tau_\RR^{-1} : \Omega\otimes_A W \rightarrow W\otimes_A \Omega$,
which is canonically induced by $\tau_\RR^{-1}: \Omega \otimes W \to W\otimes \Omega$.
Explicitly, for all $w\in W$ and $\theta\in \Omega$,
$\tau_\RR^{-1}(\theta\otimes_A w)=(\R_\alpha\ra w)\otimes_A (\R^\alpha\ra \theta)$.
With a similar argument as above, one shows that $\tau_\RR^{-1}$  is a well-defined $\MMMod{H}{A}{A}$-module
homomorphism.
The easily proven equalities $\tau_\RR\circ \tau_\RR^{-1}=\id_{\Omega\otimes_AW}$ and
  $\tau_\RR^{-1}\circ\tau_\RR=\id_{W\otimes_A\Omega}$ show that
  $\tau_\RR:W\otimes_A\Omega\rightarrow \Omega\otimes_A W$ is an
 $\MMMod{H}{A}{A}$-module isomorphism.

The braiding properties (\ref{hexagonA1}) and (\ref{hexagonA2}) can be
written with the $\RR$-tensor product $\otimes_\RR$ replaced by the
usual tensor product because all maps are $H$-equivariant.
Then these properties follow from the $\RR$-matrix properties (\ref{13}).
\end{proof}

\begin{Proposition}\label{twistedleft connection}
Let $(H,\RR)$ be a quasitriangular Hopf algebra,
$A$ be a quasi-commutative $\AAAlg{H}{}{}$-algebra,  $W$ be a quasi-commutative $\MMMod{H}{A}{A}$-module and
 $\bigl(\Omega^\bullet,\wedge,\dif\bigr)$ be a graded
 quasi-commutative left $H$-covariant differential calculus over $A$. 
Then any right connection $\dd\in\Con_A(V)$ is also a quasi-left connection in the sense that
 we have the braided left Leibniz rule, for all $a\in A$ and $w\in W$, 
\eq
\nabla(a\cdot w)=(\oR^\al\ra a)\cdot (\oR_\al\RA
\nabla)(w)+(\R_\al\ra w)\otimes_A(\R^\al\ra \dif a)~.
\en

 \noi More generally, for all $\xi\in H$, $a\in A$ and $w\in W$,
\eq
(\xi\RA \nabla)(a\cdot w)=(\oR^\al\ra a)\cdot (\oR_\al\xi\RA\nabla)(w)+\epsi(\xi)(\R_\al\ra w)\otimes_A(\R^\al\ra \dif a)~.
\en
\end{Proposition}
\begin{proof} For all $a\in A$ and $w\in W$,
\eqa
\nabla(a\cdot w)&=&\nabla\big((\R_\al\ra w)\cdot(\R^\al\ra
a)\big)=\nabla(\R_\al \ra w)\cdot (\R^\al\ra a) + \R_\al\ra
w \otimes_A \dif (\R^\al\ra a)\nn\\
&=&(\oR^\be\R^\al\ra a)\cdot \oR_\be\ra(\nabla(\R_\al\ra
w))+\R_\al\ra w\otimes_A\R^\al\ra \dif a\nn\\
&=&(\oR^\be\oR^\al\ra a)\cdot \oR_\be\ra(\nabla(S(\oR_\al)\ra
w))+\R_\al\ra w \otimes_A\R^\al\ra \dif a\nn\\
&=&(\oR^\be\ra a)\cdot (\oR_\be\RA \nabla)( w)+ \R_\al\ra w \otimes_A\R^\al\ra \dif a~,
\ena
where in the second line we have used that the $\MMMod{H}{A}{A}$-module $W\otimes_A
\Omega$ is quasi-commutative (since both $W$ and $\Omega$ are,
cf. Proposition \ref{qc1qc2thenqc12}),  and in the third line the
property (\ref{SidR}) of the $\RR$-matrix.

The more general expression for $\xi\RA\nabla$ follows from (\ref{eqn:Hcon}).
\end{proof}

\begin{Theorem}
\label{theo:conplus}
 Let $(H,\RR)$ be a quasitriangular Hopf algebra,
  $A$ be a quasi-commutative $\AAAlg{H}{}{}$-algebra,
   $W$ be a quasi-commutative $\MMMod{H}{A}{A}$-module and
 $\bigl(\Omega^\bullet,\wedge,\dif\bigr)$ be a graded
 quasi-commutative left $H$-covariant differential calculus over $A$.
 For any $\MMMod{H}{}{A}$-module $V$ and any $\dd_V\in \Con_A(V)$,
 $\dd_W\in\Con_A(W)$ we consider the $\bfK$-linear map
$\dd_V {\;\widehat{\oplus}_\RR\,} \dd_W \,:\, V\otimes W\to V\otimes_A
W\otimes_A\Omega$ defined by
\begin{flalign}\label{SIGMA}
 \dd_V{\;\widehat{\oplus}_\RR\,} \dd_W:=\tau^{-1}_{\RR\,23}\circ\pi\circ(\dd_V\otimes_\RR\id) 
+\pi\circ( \id\otimes_\RR\dd_W) ~,
\end{flalign}
where the $\RR$-tensor product of $\bfK$-linear maps was defined in
(\ref{eqn:Rtensor}), and where
$\tau^{-1}_{\RR\,23}=\id\otimes_\RR\tau_\RR^{-1}$ is the inverse braiding map
acting on the second and third entry of the tensor product
$V\otimes_A\Omega\otimes_A W$ (the $\RR$-tensor product in
$\id\otimes_\RR \tau^{-1}_\RR$ is defined in Theorem \ref{propo:RtensorAhom}).

This map induces a 
connection on the quotient module $ V\otimes_A W$,
\eq \label{eqn:tensorcon}
\dd_V \oplus_\RR \dd_W : V\otimes_A W\to V\otimes_A W\otimes_A\Omega~,
\en
defined by, for all $v\in V, w\in W$, $\,(\dd_V \oplus_\RR \dd_W)
(v\otimes_A w):=(\dd_V{\;\widehat{\oplus}_\RR\,} \dd_W)(v\otimes w)$, i.e.,
\eq
\,(\dd_V \oplus_\RR \dd_W)
(v\otimes_A w)=
 \tau^{-1}_{\RR\,23}\bigl(\dd_V (v)\otimes_A w\bigr)
+ (\oR^\be\ra v)\otimes_A (\oR_\be\RA\dd_W)(w)~, 
\en
and extended by $\bbK$-linearity to all $V\otimes_AW$.
\end{Theorem}
\begin{proof}
The definition of $\dd_V{\;\widehat{\oplus}_\RR\,}\dd_W$ slightly simplifies if we observe that $\dd_V\otimes_\RR\id =\dd_V\otimes \id$,
$\id\otimes_\RR\tau_\RR^{-1}=\id\otimes\tau_\RR^{-1}$,
since $\id$ and $\tau_\RR^{-1}$ are  $H$-equivariant.
We first prove that the $\bfK$-linear map $\dd_V \oplus_\RR \dd_W$ is
well-defined by showing that ${\rm ker}(\pi)\subseteq {\rm{ker}(\dd_V{\;\widehat{\oplus}_\RR\,}\dd_W)}$.
In order to simplify the notation in the proof we denote
both connections $\dd_V$ and $\dd_W$ by $\dd$.
For all $v\in V$, $w\in W$ and $a\in A$,
\begin{flalign}
\dd{\;\widehat{\oplus}_\RR\,}\dd(v\cdot a\otimes w)&= \tau^{-1}_{\RR\,23}\bigl(\dd (v\cdot a)\otimes_A w\bigr)
+ (\oR^\be\ra (v\cdot a))\otimes_A (\oR_\be\RA\dd)(w) \nn\\
\nn &= \tau^{-1}_{\RR\,23}\bigl((\dd v)\otimes_A a\cdot w + v\otimes_A \dif a\otimes_A w\bigr) \\
\nn & \qquad \qquad \qquad ~~~~~~~~~~~+ (\oR^\be\ra v) \otimes_A (\oR^\al\ra a)\cdot (\oR_\al \oR_\be\RA\dd)(w)\\
\nn&=\tau^{-1}_{\RR\,23}\bigl((\dd v)\otimes_A a\cdot w\bigr)
+ (\oR^\alpha\ra v) \otimes_A (\oR_\alpha\RA\dd)(a\cdot w)\\
&= \dd {\;\widehat{\oplus}_\RR\,}  \dd
(v\otimes a\cdot w)~,
\end{flalign}
where in the second line we used the right Leibniz rule (\ref{eqn:rightcon}) and the property 
(\ref{Ral1Ral21first})  of the $\RR$-matrix. In the third line we used
Proposition \ref{twistedleft connection} and the
normalization property of the $\RR$-matrix (\ref{Rnorm}).

The $\bfK$-linear map  $\dd_V \oplus_\RR \dd_W$ defined in
(\ref{eqn:tensorcon}) is a connection because
it satisfies the right Leibniz rule (\ref{eqn:rightcon}), for all $v\in V$, $w\in W$ and $a\in A$,
\begin{flalign}
\nn (\dd\oplus_\RR\dd)(v\otimes_A w\cdot a) &= \tau^{-1}_{\RR\,23}\bigl((\dd v)\otimes_A w\cdot a\bigr)
+ (\oR^\alpha\ra v)\otimes_A (\oR_\alpha\RA\dd)(w\cdot a)\\
\nn&=(\dd\oplus_\RR\dd)(v\otimes_A w)\cdot a + \varepsilon(\oR_\alpha)~(\oR^\alpha\ra v)\otimes_A w\otimes_A \dif a\\
&=(\dd\oplus_\RR\dd)(v\otimes_A w)\cdot a + v\otimes_A w\otimes_A \dif a~,
\end{flalign}
where in the second line we have used that $\tau_\RR$ is an $\MMMod{H}{A}{A}$-module isomorphism and (\ref{eqn:Hcon}),
and in the last line the normalization property of the $\RR$-matrix
(\ref{Rnorm}). 
\end{proof}

The sum of connections is compatible with the Hopf algebra action, for
all $\xi\in H$,
\begin{flalign} 
\label{xiVW}
\xi\RA (\dd_V\oplus_\RR\dd_W) = (\xi\RA\dd_V) \oplus_\RR (\xi\RA\dd_W)~.
\end{flalign}
This property easily follows recalling property
(\ref{eqn:RtensorHmod}) and observing that all the maps in (\ref{SIGMA}), but the
connections $\dd_V$ and $\dd_W$, are $H$-equivariant. 
\sk

We end this subsection showing  that the sum  of connections 
is associative.
\begin{Theorem}\label{theo:tensorcon}
  Let $(H,\RR)$ be a quasitriangular Hopf algebra,
  $A$ be a quasi-commutative $\AAAlg{H}{}{}$-algebra, 
  $W,Z$ be two quasi-commutative $\MMMod{H}{A}{A}$-modules and
  $\bigl(\Omega^\bullet,\wedge,\dif\bigr)$ be a graded quasi-commutative left $H$-covariant differential calculus over $A$.
  Then for any $\MMMod{H}{}{A}$-module $V$, $\dd_V\in \Con_A(V)$, $\dd_W\in\Con_A(W)$ 
  and $\dd_Z\in\Con_A(Z)$,
\begin{flalign}\label{eqn:tensorconass}
 \bigl(\dd_V\oplus_\RR \dd_W\bigr)\oplus_\RR \dd_Z = \dd_V\oplus_\RR\bigl(\dd_W\oplus_\RR\dd_Z\bigr)~.
\end{flalign}
\end{Theorem}
\begin{proof}
It is convenient to denote by $\tau_{\RR\,i~i+1}^{-1} = \id\otimes\dots\otimes\tau_\RR^{-1}\otimes\dots\otimes\id$ the inverse
braiding map acting on the $i$-th and $(i+1)$-th entry of a tensor product.
The inverse braiding map exchanging the $i$-th entry with the $(i+1)$-th and $(i+2)$-th entry is denoted by
$\tau_{\RR\,i~(i+1~i+2)}^{-1}$ and similarly $\tau_{\RR\,(i~i+1)~i+2}^{-1}$ is the inverse braiding map exchanging
the $i$-th and $(i+1)$-th entry with the $(i+2)$-th entry.
For example, $\tau_{\RR\,(12)3}^{-1}(a\otimes b\otimes c) = (\R_\alpha\ra c) \otimes \R^\alpha\ra (a\otimes b)$
and $\tau_{\RR\,1(23)}^{-1}(a\otimes b \otimes c) = \R_\alpha\ra(b\otimes c)\otimes(\R^\alpha\ra a)$.

Associativity of  $\oplus_\RR$ is implied by associativity of
${\;\widehat{\oplus}}_\RR\,$. The latter follows from the associativity of
$\otimes_\RR$, 
the $H$-equivariance of the $\id, \tau_\RR^{-1}$ and $\pi$ maps, the composition
property (\ref{eqn:Rtensorcirc}),
the composition of projections property 
$\pi\circ (\id\otimes_\RR\pi)=\pi\circ (\pi\otimes_\RR\id)$, and the
braid relation (\ref{hexagonA1}). 

We here present a less abstract proof by evaluating
(\ref{eqn:tensorconass}) on $V\otimes_A W\otimes_A Z$.
Due to $\bbK$-linearity it is
enough to show associativity on elements $v\otimes_A w\otimes_A z\in
V\otimes_A W\otimes_A Z$. We denote all connections by $\nabla$ in order to simplify the notation and find
\eqa
\nn ((\dd\!\!\!\! &\oplus_\RR&\!\!\!\!\dd)\oplus_\RR \dd) (v\otimes_A\otimes w\otimes_A z) \nn\\[0.4em]
&=&\tau^{-1}_{\RR\,34}\Big( (\dd\oplus_\RR\dd)(v\otimes_A w)\otimes_A z\Big) \,
+ \oR^\al\ra(v\otimes_A w)\otimes_A (\oR_\al\RA \dd)(z)\nn\\
\nn&=&\tau^{-1}_{\RR\,34}\Big(\tau^{-1}_{\RR\,23} (\dd(v)\otimes_A
w)\otimes_A z + (\oR^\al\ra v)\otimes_A(\oR_\al\RA\dd)(w)\otimes_A z\Big) 
\nn\\[-0.2em] & &~~~~~~~~~~~~~~~~~~~~~~~~~~~~~~~~~~~~~~~~~\;+ (\oR^\al\ra
v)\otimes_A (\oR^\be\ra w)\otimes_A (\oR_\be\oR_\al\RA \dd)(z)\nn\\[0.5em]
\nn &=&\tau^{-1}_{\RR\,34}\Big( \tau^{-1}_{\RR\,23}(\dd(v)\otimes_A w\otimes_A
z) \Big)\,+ (\oR^\al\ra v)\otimes_A\tau^{-1}_{\RR\,23}\big(  (\oR_\al\RA
\dd)(w)\otimes_A z\big)
\nn\\ & &~~~~~~~~~~~~~~~~~~~~~~~~~~~~~~~~~~~~~~~~~~+
 (\oR^\al\ra v)\otimes_A (\oR^\be\ra w)\otimes_A (\oR_\be\oR_\al\RA \dd)(z)\nn\\[0.5em]
\nn &=&\tau^{-1}_{\RR\,2 (34)} (\dd(v)\otimes_A w\otimes_A z) +
(\oR^\al\ra v)\otimes_A (\oR_a\RA (\dd\oplus_\RR\dd)) (w\otimes_A z)\nn\\[0.5em]
&=& (\dd\oplus_\RR(\dd\oplus_\RR\dd)) (v\otimes_A\otimes_A w\otimes_A z) ~,
\ena
where in the fourth line we used 
$\tau^{-1}_{\RR\,2(34)} =
\tau^{-1}_{\RR\,34}\circ\tau^{-1}_{\RR\,23}$, which follows from 
the braid relation (\ref{hexagonA1}), and 
property (\ref{xiVW}).
\end{proof}

\sk
\subsection{Deformation of the sum of connections}
We study the deformation of the sum of two connections.
We need a preliminary

\begin{Lemma}\label{inductions}
Let $(H,\RR)$ be a quasitriangular Hopf algebra with twist $\FF\in H\otimes H$,
 $A$ be an $\AAAlg{H}{}{}$-algebra, $V$ be an $\MMMod{H}{}{A}$-module and $W,\Omega$ be two 
 $\MMMod{H}{A}{A}$-modules.
The commutative diagram (\ref{eqn:higheriotast}) induces the
commutative diagram (\ref{diagram}), where the inner diagonal map $\varphi_{V_\st,W_\st,\Omega_\st}:
V_\star\otimes_{A_\star} W_\star\otimes_{A_\star} \Omega_\star
\to (V\otimes_A W\otimes_A \Omega)_\star$, that is
defined as the composition
$\varphi_{V_\st,W_\st,\Omega_\st}:=\varphi_{V_\st,(W\otimes_A\Omega)_\st}\circ\big(
  \id_{V_\st}\otimes_{\RR^\FF}\varphi_{W_\st,\Omega_\st}\big)$, is an $\MMMod{H^\FF}{}{A_\star}$-module  isomorphism.

\begin{sidewaystable}
\begin{flalign}
\!\!\!\!\xymatrix{
V_\st\otimes_\st W_\st\otimes_\st \Omega_\st
\ar[ddddd]^-{\varphi_{V,W}\otimes_{_{\RR^{_\FF}}} \id_{\Omega_\st}} \ar[rrrrrr]^-{\id_{V_\st}\otimes_{_{\RR^{_\FF}}}\varphi_{W,\Omega}} \ar[dr]^-{\pi_{\st_{23}}}& & & & & &
V_\st\otimes_\st (W \otimes \Omega)_\st \ar[dl]_-{\pi_{23}} \ar[ddddd]^-{\varphi_{V,W\otimes\Omega}}\\
& V_\st\otimes_\st (W_\st\otimes_{A_\st} \Omega_\st)
\ar[rrrr]^-{\id_{V_\st}\otimes_{_{\RR^{_\FF}}}\varphi_{W_\st,\Omega_\st}}
\ar[dr]^{\pi_{\st}}
& & & &
V_\st\otimes_\st (W \otimes_A \Omega)_\st \ar[dl]_{\pi_\st} \ar[ddd]^-{\varphi_{V,W\otimes_A\Omega}}&\\
& &V_\st\otimes_{A_\st} W_\st\otimes_{A_\st} \Omega_\st
\ar[rr]^-{\id_{V_\st}\otimes_{_{\RR^{_{\FF}}}}   \varphi_{W_\st,\Omega_\st}}
\ar[d]_-{\varphi_{V_\st,W_\st}\otimes_{_{\RR^{_{\FF}}}}\id_{\Omega_\st}}
\ar[drr]_-{\varphi_{V_\st,W_\st,\Omega_\st}}& &
\,\,V_\st\otimes_{A_\st}(W\otimes_A\Omega)_\st\,\,\ar[d]^-{\varphi_{V_\st,(W\otimes_A\Omega)_\st}}& &\\
& & (V\otimes_A W)_\st\otimes_{A_\st}\Omega_\st
\ar[rr]_-{\varphi_{(V\otimes_A W)_\st,\Omega_\st}}  & & (V\otimes_A W \otimes_A\Omega)_\st & &\\
& (V\otimes W)_\st\otimes_{A_\st}\Omega_\st 
\ar[ru]^{\pi_{12}}& & & & (V\otimes (W\otimes_A\Omega))_\st \ar[lu]_{\pi}& \\
(V\otimes W)_\st\otimes_\st \Omega_\st \ar[rrrrrr]_-{\varphi_{V\otimes W,\Omega}} \ar[ru]^-{\pi_{\st}}& & & & & & (V\otimes W\otimes \Omega)_\st \ar[lu]_-{\pi_{23}}
}\nn
\end{flalign}
\sk
\eq\label{diagram}
\en
\end{sidewaystable}
\end{Lemma}
\begin{proof}
The diagram (\ref{diagram}) is composed of the commutative trapezia
defining the induced maps $\varphi$ on the quotients. The outer
rectangle is the commutative diagram (\ref{eqn:higheriotast}); its
commutativity implies that all subdiagrams of (\ref{diagram}) are commutative.
Since all maps in the diagram are $\MMMod{H^\FF}{}{A_\star}$-module 
homomorphisms, so is $\varphi_{V_\st,W_\st,\Omega_\st}$. Since the
$\varphi$-maps in the outer rectangle are isomorphisms, so are all the
induced $\varphi$-maps, hence in particular $\varphi_{V_\st,W_\st,\Omega_\st}$.
\end{proof}

\sk
The deformation of the sum of connections (\ref{eqn:tensorcon}), 
$D_\FF(\dd_V\oplus_\RR\dd_W):(V\otimes_A W)_\star \to
(V\otimes_A W\otimes_A\Omega)_\star$, where $D_\FF$ is the deformation isomorphism
discussed in Theorem \ref{PDPHom}, is by construction a $\bfK$-linear
map. We show that up to $\varphi$-isomorphisms it is a connection.

\begin{Theorem}\label{theo:prodcondef}
  Let $(H,\RR)$ be a quasitriangular Hopf algebra with twist $\FF\in H\otimes H$,
$A$ be a quasi-commutative $\AAAlg{H}{}{}$-algebra, 
$W$ be a quasi-commutative $\MMMod{H}{A}{A}$-module and
 $\bigl(\Omega^\bullet,\wedge,\dif\bigr)$ be a graded quasi-commutative left $H$-covariant differential 
 calculus over $A$.
Then for any $\MMMod{H}{}{A}$-module  $V$,
$\dd_V\in \Con_A(V)$ and $\dd_W\in\Con_A(W)$ the following diagram commutes
\begin{flalign}\label{eqn:quantprodcon}
\xymatrix{
V_\star \otimes_{A_\star} W_\star \ar[rrrr]^-{\widetilde{D}_\FF(\dd_V)\oplus_{\RR^\FF} \widetilde{D}_\FF(\dd_W)} \ar[d]_-{\varphi_{V_\st,W_\st}}& & & & V_\star \otimes_{A_\star} W_\star \otimes_{A_\star} \Omega_\star \ar[d]^-{\varphi_{V_\st,W_\st,\Omega_\st}}\\
(V\otimes_A W)_\star \ar[rrrr]_-{D_\FF(\dd_V\oplus_\RR\dd_W)} & & & & (V\otimes_A W\otimes_A\Omega)_\star 
} 
\end{flalign}
\end{Theorem}
\begin{proof}
We recall  that the sum of connections $\dd_V\oplus_\RR \dd_W$ is canonically induced from the
$\bbK$-linear map $\dd_V\,\widehat\oplus_\RR\, \dd_W : V\otimes
W\to V\otimes_A W\otimes_A\Omega$ defined in  (\ref{SIGMA}). 
We therefore first  show commutativity of the following diagram
\begin{flalign}\label{eqn:quantprodconhat}
\xymatrix{
V_\star \otimes_{\star} W_\star \ar[rrrr]^-{\widetilde{D}_\FF(\dd_V)\widehat\oplus_{\RR^\FF} \widetilde{D}_\FF(\dd_W)} \ar[d]_-{\varphi_{V,W}}& & & & V_\star \otimes_{A_\star} W_\star \otimes_{A_\star} \Omega_\star \ar[d]^-{\varphi_{V_\st,W_\st,\Omega_\st}}\\
(V\otimes W)_\star \ar[rrrr]_-{D_\FF(\dd_V\widehat\oplus_\RR\dd_W)} & & & & (V\otimes_A W\otimes_A\Omega)_\star 
} 
\end{flalign}
To simplify the notation we denote all connections by $\dd$. 
 Using $\bfK$-linearity
of $D_\FF$ and that $\tau_\RR^{-1}$, $\id$ and $\pi$ are $H$-equivariant, we obtain
\begin{flalign}
\nn D_\FF(\dd\,\widehat{\oplus}_\RR\dd) &= D_\FF\bigl(\tau^{-1}_{\RR\,23}\circ\pi\circ  (\dd\otimes_\RR \id) + \pi\circ \id\otimes_\RR \dd\bigr)\\
 &= \tau^{-1}_{\RR\,23}\circ\pi\circ  D_\FF(\dd\otimes_\RR \id) + \pi\circ D_\FF(\id\otimes_\RR\dd)~.
\end{flalign}
Theorem \ref{theo:promodhomdef} and $H$-equivariance of $\id$ implies
\begin{flalign}
\nn D_\FF(\dd\,\widehat{\oplus}_\RR\dd) &=  \tau^{-1}_{\RR\,23}\circ \pi\circ \varphi_{V\otimes_A\Omega,W} \circ \bigl(D_\FF(\dd)\otimes_{\RR^\FF}\id\bigr)\circ \varphi_{V,W}^{-1}\\
\nn  &\qquad\qquad\qquad~~~~~ +
\pi\circ \varphi_{V,W\otimes_A\Omega}\circ \bigl(\id\otimes_{\RR^\FF}D_\FF(\dd)\bigr)\circ \varphi_{V,W}^{-1} \\
\nn &= \tau^{-1}_{\RR\,23}\circ \pi\circ \varphi_{V\otimes_A\Omega,W} \circ (\varphi_{V_\st,\Omega_\st}\otimes_{\RR^\FF}\id)\circ \bigl(\widetilde{D}_\FF(\dd)\otimes_{\RR^\FF}\id\bigr)\circ \varphi_{V,W}^{-1}\\
\nn & \qquad\qquad\qquad~~~~~  + \pi\circ \varphi_{V,W\otimes_A\Omega}\circ (\id\otimes_{\RR^\FF}\varphi_{W_\st,\Omega_\st})\circ \bigl(\id\otimes_{\RR^\FF}\widetilde{D}_\FF(\dd)\bigr)\circ \varphi_{V,W}^{-1}\\
\label{eqn:quantprodcontmp} & =\varphi_{V_\st,W_\st,\Omega_\st}\circ \Big(\tau^{-1}_{\RR^\FF\,23}\circ \pi_\st \circ \bigl(\widetilde{D}_\FF(\dd)\otimes_{\RR^\FF}\id\bigr) + \pi_\st \circ \bigl(\id\otimes_{\RR^\FF}\widetilde{D}_\FF(\dd)\bigr)\Big)\circ \varphi_{V,W}^{-1}~,
\end{flalign}
where in the last equality we have used: For the second addend commutativity of the
sub-diagram
{\footnotesize
\begin{flalign*}
\xymatrix{
{\color{black}{ V_\st\otimes_\st (W_\st\otimes_{A_\st} \Omega_\st)}}
\ar[rrrr]^-{\id_{V_\st}\otimes_{_{\RR^{_\FF}}}\varphi_{W_\st,\Omega_\st}}
\ar[dr]^{\pi_{\st}}
& & & &
\color{black}{V_\st\otimes_\st (W \otimes_A \Omega)_\st}  \ar[ddd]^-{\varphi_{V,W\otimes_A\Omega}}\\
 &\color{black}{V_\st\otimes_{A_\st} W_\st\otimes_{A_\st} \Omega_\st}
\ar[drr]_-{\varphi_{V_\st,W_\st,\Omega_\st}}& &
\color{white}{\,\,V_\st\otimes_{A_\st}(W\otimes_A\Omega)_\st\,\,}
& \\
 & \color{white}{(V\otimes_A W)\st\otimes_{A_\st}\Omega_\st}
 & &\color{black}{(V\otimes_A W \otimes_A\Omega)_\st} & \\
& & & & \color{black}{(V\otimes (W\otimes_A\Omega))_\st}\ar[lu]_{\pi} 
}
\end{flalign*}
}

\noi of the diagram (\ref{diagram}). 
For the first addend  we have used 
$\tau^{-1}_{\RR^\FF\,23}=(\id\otimes_{\RR^\FF} \varphi^{-1}
_{W_\st, \Omega_\st})\circ \tau^{-1}_{\RR\,{23}}\circ (\id\otimes_{\RR^\FF}
  \varphi_{\Omega_\st,W_\st})$ (recall (\ref{tauRFmap}) and the text above
  (\ref{tauRFmap})),
$\varphi_{V_\st, (W\otimes_A\Omega)_\st}\circ
  \tau^{-1}_{\RR\,23}=\tau^{-1}_{\RR\, 23}\circ\varphi_{V_\st,
    (\Omega\otimes_A W)_\st}$ (because of $H$-equivariance of  $\tau^{-1}_{\RR\,23}$), and 
$\varphi_{V_\star,\Omega_\star,W_\star}\circ \pi_\star=\pi \circ \varphi_{V\otimes_A \Omega,W} \circ \big(\varphi_{V_\star,\Omega_\star}\otimes_{\RR^\FF}\id \big)  
$. This last equality follows from a similar diagram as in
  Lemma \ref{inductions}. (Hint: start with the outer diagram drawn in (\ref{diagram})
  and use in the upper left corner the projection $\pi_{\star_{12}}$ instead
  of $\pi_{\star_{23}}$.)
\sk

Equation (\ref{eqn:quantprodcontmp}) shows commutativity of the diagram (\ref{eqn:quantprodconhat}).
Since the map $\varphi_{V,W}$ canonically induces on the quotient
$V_\st\otimes_{A_\st}W_\st$ the
map $\varphi_{V_\st, W_\st}$ (see (\ref{eqn:lemma4})), and since
 the map $ D_\FF(\dd\,\widehat{\oplus}_\RR\dd)  $ canonically
induces on the quotient $(V\otimes_A W)_\st$  the map
$D_\FF(\dd\,{\oplus}_\RR\dd) $ (in fact, the actions $\RA$ and $\ra$ present in $D_\FF$
are canonically induced on the quotient),
the commutative diagram (\ref{eqn:quantprodconhat}) induces the commutative diagram 
(\ref{eqn:quantprodcon}).
\end{proof}


\subsection{\label{sec:leftrightcon}From right to left connections}
Let $(H,\RR)$ be a quasitriangular Hopf algebra and $A$ be a  quasi-commutative $\AAAlg{H}{}{}$-algebra.
We showed in Theorem \ref{theo:qclriso} that the map $D_\RR$, i.e.~the deformation isomorphism
with twist $\FF=\RR$, provides an isomorphism between right and left 
$A$-linear maps of strong quasi-commutative $\MMMod{H}{A}{A}$-modules.
The aim of this subsection is to prove a similar statement for right and left connections
on a strong quasi-commutative $\MMMod{H}{A}{A}$-module $V$.

\sk
We notice that due to Theorem \ref{theo:qclriso} we have an $\MMMod{H}{A}{A}$-module isomorphism
\begin{flalign}
D_\RR : \Hom_\bbK(V,V\otimes_A\Omega) \to \big(\Hom_\bbK(V,V\otimes_A\Omega)\big)^\op~,~~P\mapsto D_\RR(P) = (\oR^\alpha\RA P)\circ \oR_\alpha\ra
\end{flalign}
between $\big(\Hom_\bbK(V,V\otimes_A\Omega),\cdot,{\RA}\big)$
and $\big(\Hom_\bbK(V,V\otimes_A\Omega),\cdot^\op,{\RA^\cop}\big)$, which restricts to an $\MMMod{H}{A}{A}$-module
 isomorphism (denoted by the same symbol)
\begin{flalign}
D_\RR :\Hom_A(V,V\otimes_A\Omega) \to \big({}_A\Hom(V,V\otimes_A\Omega)\big)^\op~.
\end{flalign}
Composing this map with the inverse braiding map $\tau_\RR^{-1} : V\otimes_A \Omega \to \Omega\otimes_A V$
we obtain another $\MMMod{H}{A}{A}$-module isomorphism
\begin{flalign}\label{leftrightconnectioniso}
\widetilde{D}_\RR : \Hom_\bbK(V,V\otimes_A\Omega) \to \big(\Hom_\bbK(V,\Omega\otimes_A V)\big)^\op~,~~P\mapsto \widetilde{D}_\RR(P) = \tau_\RR^{-1}\circ D_\RR(P)~,
\end{flalign}
which also restricts to an $\MMMod{H}{A}{A}$-module isomorphism (denoted by the same symbol)
\begin{flalign}\label{leftrighthomomorphismiso}
\widetilde{D}_\RR :\Hom_A(V,V\otimes_A\Omega) \to \big({}_A\Hom(V,\Omega \otimes_A V)\big)^\op~.
\end{flalign}

\begin{Theorem}\label{theo:leftrightconiso}
  Let $(H,\RR)$ be a quasitriangular Hopf algebra, $A$ be a quasi-commutative $\AAAlg{H}{}{}$-algebra,
   $V$ be a strong quasi-commutative $\MMMod{H}{A}{A}$-module and 
 $\bigl(\Omega^\bullet,\wedge,\mathrm{d}\bigr)$ be a graded 
quasi-commutative left $H$-covariant differential calculus over $A$.
Then the $\MMMod{H}{A}{A}$-module isomorphism (\ref{leftrightconnectioniso}) restricts to 
the affine space isomorphism
\begin{flalign}
 \widetilde{D}_\RR: \Con_A(V)\to {_A}\Con(V)~,\quad \dd \mapsto \widetilde{D}_\RR(\dd) =\tau_\RR^{-1} \circ  \,(\oR^\alpha\RA\dd)\circ \oR_\alpha\ra~,
\end{flalign}
where $\Con_A(V)$ and ${}_A\Con(V)$ are respectively affine spaces over the isomorphic $\MMMod{H}{A}{A}$-modules
 (cf.~(\ref{leftrighthomomorphismiso})) $\Hom_A(V,V\otimes_A\Omega)$
and $\big({}_A\Hom(V,\Omega \otimes_A V)\big)^\op$.
\end{Theorem}
\begin{proof}
 By Lemma \ref{HAVop} we have $A_{\star_\RR} = A^\op$, $\Omega_{\star_\RR} = \Omega^{\op}$ and $V_{\star_\RR}=V^\op$.
 These equalities and the equality  (\ref{Dnabvsta}) (see Theorem \ref{theo:condef}) imply, for all $\dd\in\Con_A(V)$, $a\in A$ and $v\in V$,
\begin{flalign}
 \nn D_\RR(\dd)(a\cdot v) = D_\RR(\dd)(v\star_{\RR} a) &=
 D_\RR(\dd)(v)\star_{\RR} a + 
(\oR^\alpha\trgl v)\otimes_A(\oR_\alpha\trgl\mathrm{d} a)~\\
&= a\cdot D_\RR(\dd)(v) + (\oR^\alpha\trgl v)\otimes_A(\oR_\alpha\trgl\mathrm{d} a)~.
\end{flalign}
Applying the $\MMMod{H}{A}{A}$-module isomorphism $\tau_\RR^{-1}$ we find
\begin{flalign}
 \widetilde{D}_\RR(\dd)(a\cdot v) = a\cdot \widetilde{D}_\RR(\dd)(v) + \mathrm{d}a\otimes_A v~.
\end{flalign}
Thus $\widetilde{D}_\RR(\dd)\in{_A}\Con(V)$ is a left connection on $V$. The map $\widetilde{D}_\RR$ is invertible via
\begin{flalign}
 \widetilde{D}_\RR^{-1}: {_A}\Con(V)\to \Con_A(V)~,\quad \dd \mapsto \widetilde{D}_\RR^{-1}(\dd) =
 \tau_\RR\circ \bigl(D_\RR^{-1}(\dd)\bigr)~,
\end{flalign}
where $D_\RR^{-1}(\dd) 
=(\R^\alpha\RA^\cop
\dd)\circ\R_\alpha\ra\;$,
recall $\RA^\cop=\RA_\RR$, and Remark \ref{rem:Dinv}.

Finally, $\widetilde{D}_\RR$ is an affine space isomorphism because, for all $\dd\in \Con_A(V)$
and $P\in \Hom_A(V,V\otimes_A\Omega)$, $\widetilde{D}_\RR(\dd + P) = \widetilde{D}_\RR(\dd) + \widetilde{D}_\RR(P)$,
where $\widetilde{D}_\RR(P) \in \big({}_A\Hom(V,\Omega\otimes_A
V)\big)^\op$ is the image of the right $A$-linear map $P$ under the $\MMMod{H}{A}{A}$-module 
 isomorphism (\ref{leftrighthomomorphismiso}).
\end{proof}
\begin{Remark}
Similarly to Theorem \ref{theo:rldefcomp} one can show that 
the right to left isomorphism for connections is compatible with twist deformation:
Let $\FF\in H\otimes H$ be a twist of the quasitriangular Hopf algebra $(H,\RR)$.
 Then, due to $H$-equivariance of the $\varphi$ and $\tau_\RR$ maps, the following diagram of isomorphisms of affine
spaces over $\MMMod{H^\FF}{A_\star}{A_\star}$-modules commutes
\begin{flalign}
\xymatrix{
\ar[d]_-{\widetilde{D}_\RR}\Con_A(V)_\star \ar[rr]^-{\widetilde{D}_\FF} & & \Con_{A_\star}(V_\star)\ar[d]^-{\widetilde{D}_{\RR^\FF}}\\
{}_A\Con(V)_\star \ar[rr]^-{\widetilde{D}_\FF^\cop}& & {}_{A_\star}\Con(V_\star)
}
\end{flalign}
\end{Remark}


\subsection{Connections induced on dual modules}
A connection $\dd$ on a finitely generated and projective $\MMMod{}{}{A}$-module $V$
induces in a canonical way a connection
$\dd^\prime$
on the dual $\MMMod{}{A}{}$-module 
$V^\prime=\Hom_A(V,A)$.
 In the beginning of this subsection 
we review in detail this construction.
It  will be needed in the remaining part of the subsection,
when we consider connections $\dd$ on $\MMMod{H}{}{A}$-modules
(and $\MMMod{H}{A}{A}$-modules), the associated  dual connections
$\dd'$, and their twist deformation. 
\sk

There are various equivalent definitions of  finitely generated and projective modules over a ring, see e.g.~the monograph
\cite{LamBook}. For our purposes it is enough to use their  convenient 
characterization in terms of a pair of dual
bases (this is the so-called  {\it Dual Basis Lemma}).
\begin{Lemma}
 Let $A$ be an algebra. An $\MMMod{}{}{A}$-module $V$ is
finitely generated and projective if and only if there exists a family of elements
$\lbrace v_i\in V:i=1,\dots,n\rbrace$ and right $A$-linear maps
$\lbrace v_i^\prime\in V^\prime:=\Hom_A(V,A):i=1,\dots,n \rbrace$
with $n\in\mathbb{N}$, such that for any $v\in V$ we have
\begin{flalign}
 v=\sum\limits_{i=1}^n v_i\cdot v_i^\prime(v)~.
\end{flalign}

An $\MMMod{}{A}{}$-module $V$ is
finitely generated and projective if and only if there exists a family of elements
$\lbrace v_i\in V:i=1,\dots,n\rbrace$ and left $A$-linear maps $\lbrace v^\prime_i \in V^\prime
:={_A}\Hom(V,A):i=1,\dots,n \rbrace$
with $n\in\mathbb{N}$, such that for any $v\in V$ we have
\begin{flalign}
 v=\sum\limits_{i=1}^n v^\prime_i(v)\cdot v_i~.
\end{flalign}
\end{Lemma}
The set $\lbrace v_i,v^\prime_i:i=1,\dots,n \rbrace$ is loosely referred to ``pair of dual bases'' for $V$,
even though $\lbrace v_i\rbrace$ is just a generating set of $V$ and not necessarily a basis.
\sk

The dual $V^\prime$ of a finitely generated and projective module $V$ is also finitely
generated and projective, moreover the dual of the dual is canonically
identified with the original module $V$. We state  these properties without proof since they
can be found in standard textbooks, e.g.~\cite{LamBook}.
\begin{Proposition}\label{propo:fgpmodprop}
 Let $V$ be a finitely generated and projective $\MMMod{}{}{A}$-module with
a pair of dual bases $\lbrace v_i,v_i^\prime:i=1,\dots,n \rbrace$. For any $v\in V$ let
$v^{\prime\prime} \in V^{\prime\prime}:= {_A}\Hom(V^\prime,A)$ be defined by
$v^{\prime\prime}(v^\prime) := v^\prime(v)$, for all $v^\prime \in V^\prime$. We have
\begin{enumerate}
 \item $\lbrace v_i^\prime,v_i^{\prime\prime}:i=1,\dots,n\rbrace$ is a pair of dual bases for $V^\prime$
\item $V^\prime$ is a finitely generated and projective $\MMMod{}{A}{}$-module
\item The canonical map $V\to V^{\prime\prime}\,,~v\mapsto v^{\prime\prime}$ is an isomorphism
of $\MMMod{}{}{A}$-modules
\end{enumerate}
\end{Proposition}
\sk

We now show that for  finitely generated and projective $\MMMod{}{}{A}$-modules $V$ there is an isomorphism 
$\Hom_A(V,\Omega)\simeq \Omega\otimes_A V^\prime$. It is this property
that  allows us to construct connections on $V^\prime$ by a canonical dualization procedure.
\sk
\begin{Proposition}\label{propo:projhomiso}
 Let $V$ be a finitely generated and projective $\MMMod{}{}{A}$-module and let $\Omega$ be any $\MMMod{}{A}{A}$-module. 
Then there exists an $\MMMod{}{A}{}$-module isomorphism (evaluation map)
\eq\label{iotadef}
\iota : \Omega\otimes_A V^\prime \to \Hom_A(V,\Omega)
\en defined by,
for all $v\in V$, $\omega\in \Omega$ and $v^\prime\in V^\prime$,
$\bigl(\iota(\omega\otimes_A v^\prime)\bigr)(v):= \omega\cdot v^\prime(v)$.

 If in addition $V$ is an $\MMMod{H}{}{A}$-module and $\Omega$ is an $\MMMod{H}{A}{A}$-module,
  then $\iota$ is an $\MMMod{H}{A}{}$-module
 isomorphism.
\end{Proposition}
\begin{proof}
The map $\iota$ is an $\MMMod{}{A}{}$-module homomorphism, 
for all $a\in A$, $\omega\in \Omega$, $v^\prime\in V^\prime$ and $v\in V$,
\begin{flalign}
\bigl(\iota(a\cdot \omega\otimes_A v^\prime)\bigr)(v) = a\cdot \omega \cdot v^\prime(v) =  
a\cdot \bigl(\iota(\omega\otimes_A v^\prime)\bigr)(v) = \bigl(a\cdot \iota(\omega\otimes_A v^\prime)\bigr)(v)~.
\end{flalign}
Using a pair of dual bases $\{v_i,v^\prime_i : i=1,\ldots n\}$ for $V$, we
can invert the map $\iota$ via
\begin{flalign}\label{eqn:projhomiso}
\iota^{-1}:\Hom_A(V,\Omega)\to \Omega\otimes_A V^\prime~,\quad P\mapsto \iota^{-1}(P)
 = \sum\limits_{i=1}^n P(v_i)\otimes_A v_i^\prime~. 
\end{flalign}

Let now $V$ be an $\MMMod{H}{}{A}$-module and $\Omega$ be an $\MMMod{H}{A}{A}$-module, 
then $V^\prime,\,\Omega\otimes_A V^\prime,\,\Hom_A(V,\Omega)$ are $\MMMod{H}{A}{}$-modules.
 The isomorphism $\iota$ respects the $\MMMod{H}{}{}$-module structure,
since for all $\xi\in H$, $\omega\in \Omega$, $v^\prime\in V^\prime$ and $v\in V$,
\begin{flalign}
 \nn \bigl(\iota(\xi\ra (\omega\otimes_A v^\prime))\bigr)(v) &= (\xi_1\ra \omega)\cdot (\xi_2\RA v^\prime)(v) 
= (\xi_1\ra \omega)\cdot \bigl(\xi_2\ra v^\prime(S(\xi_3)\ra v)\bigr)\\
&= \xi_1\ra\bigl(\omega\cdot v^\prime(S(\xi_2)\ra v)\bigr) = \bigl(\xi\RA \bigl(\iota(\omega\otimes_A v^\prime)\bigr)\bigr)(v)~.
\end{flalign}
\end{proof}

In order to construct the dual connection $\dd^\prime\in {}_A\Con(V^\prime)$ of
a connection $\dd\in\Con_A(V)$  we first investigate a certain
dualization of $\bbK$-linear maps and $A$-linear maps (like the
difference of two connections on $V$). 
Let $A$ be an algebra and let $\bigl(\Omega^\bullet,\wedge,\dif\bigr)$
be a differential calculus over $A$.
Let further $V$ be an $\MMMod{}{}{A}$-module and consider the $\bbK$-module of $\bbK$-linear
maps $\Hom_\bbK(V,V\otimes_A \Omega)$. We can associate to every $P\in \Hom_\bbK(V,V\otimes_A \Omega)$
a $P^\Hom \in {}_{A}\Hom(V^\prime,\Hom_\bbK(V,\Omega))$
by defining, for all $v^\prime\in V^\prime$ and $v\in V$,
\begin{subequations}\label{eqn:Phom2}
\begin{flalign}
P^\Hom(v^\prime)(v) := - (v^\prime\otimes \id)P(v)~,
\end{flalign}
i.e., \begin{flalign}\label{eqn:Phom}
P^\Hom(v^\prime) := -\wedge\circ (v^\prime\otimes \id)\circ P~.
\end{flalign}
\end{subequations}
In (\ref{eqn:Phom2}) $v^\prime\otimes \id : V\otimes_A\Omega \to A\otimes_A \Omega$ denotes the right $A$-linear map
defined by, for all $v\in V$ and $\omega\in\Omega$, $(v^\prime\otimes \id)(v\otimes_A\omega) = v^\prime(v)\otimes_A\omega$
and $\wedge: A\otimes_A\Omega \to \Omega$ is the $\MMMod{}{A}{A}$-module isomorphism induced from
the product in $\Omega^\bullet$,  i.e.~$\wedge(a\otimes_A \omega) = a\wedge \omega = a\cdot \omega$.
The map $P^\Hom$ is left $A$-linear, for all $a\in A$ and $v^\prime\in V^\prime$,
\begin{flalign}
\nn P^\Hom(a\cdot v^\prime) &= -\wedge\circ \,( a\cdot v^\prime \otimes \id) \circ P 
= -\wedge \circ\, l_a \circ (v^\prime\otimes \id)\circ P \\
&= -_{\,} l_a\circ  \wedge  \circ (v^\prime\otimes \id)\circ P  = a\cdot P^\Hom(v^\prime)~.
\end{flalign}

The association of $P^\Hom$ to $P$ can be regarded as a $\bbK$-linear map
\begin{flalign}\label{eqn:Phommap}
{}^\Hom : \Hom_\bbK(V,V\otimes_A \Omega) \to \big({}_A\Hom(V^\prime,\Hom_\bbK(V,\Omega))\big)^\op~,~~P\mapsto P^\Hom~.
\end{flalign}
Furthermore, if $P\in\Hom_A(V,V\otimes_A\Omega)$ then
$P^\Hom \in\big( {}_A\Hom(V^\prime, \Hom_A(V,\Omega))\big)^\op$, since
for all $v^\prime\in V^\prime$,
 all maps in (\ref{eqn:Phom}) are right $A$-linear. Thus, (\ref{eqn:Phommap}) restricts to a $\bbK$-linear map
 \begin{flalign}\label{eqn:Phommaprestr}
{}^\Hom : \Hom_A(V,V\otimes_A \Omega) \to \big( {}_A\Hom(V^\prime,\Hom_A(V,\Omega))\big)^\op~.
 \end{flalign}
We now  explain the notation $({~\;~})^\op$ used in these two last equations. Later we are going to consider
an additional $\MMMod{H}{}{}$-module structure on $V$ and $\Omega$, and hence also 
on $V^\prime$, $\Hom_\bbK(V,\Omega)$
and ${}_A\Hom(V^\prime,\Hom_\bbK(V,\Omega))$.
There are two different induced $\MMMod{H}{}{}$-module structures on a $\bbK$-module of $\bbK$-linear maps
between two $\MMMod{H}{}{}$-modules, they are given by the adjoint actions $\RA$ and $\RA^\cop$. We denote by 
$\big({}_{A}\Hom(V^\prime,\Hom_\bbK(V,\Omega))\big)^\op$ the
$\MMMod{H}{}{}$-module with the $\RA^\cop$ adjoint action, which is the relevant one
 for left $A$-linear maps,
cf. Proposition \ref{ABVAEnd}. 

In order to avoid confusion when we later enrich the theory with $\MMMod{H}{}{}$-module structures,
 we already write in this part $\big({}_{A}\Hom(V^\prime,\Hom_\bbK(V,\Omega))\big)^\op$ rather than 
${}_{A}\Hom(V^\prime,\Hom_\bbK(V,\Omega))$,
even if as $\bbK$-modules they coincide.
\sk 
Notice that for a finitely generated and projective $\MMMod{}{}{A}$-module $V$ we can use
 the $\MMMod{}{A}{}$-module isomorphism $\iota^{-1}$ of Proposition \ref{propo:projhomiso}  to define the
 $\bbK$-linear map
  \begin{flalign}\label{eqn:primemaphom}
{}^\prime  : \Hom_A(V,V\otimes_A \Omega) \to \big( {}_A\Hom(V^\prime,\Omega\otimes_A V^\prime)\big)^\op~,~~P\mapsto P^\prime =  \iota^{-1}\circ P^\Hom~.
 \end{flalign}
We denoted this map by a prime because, as we will see later, 
it  maps the difference of two connections on $V$ in the
difference of their dual connections on $V^\prime$.
\sk 
 
 Let now $\nabla\in\Con_A(V)\subset \Hom_\bbK(V,V\otimes_A\Omega)$, then
 $\nabla^\Hom$ is by construction a left $A$-linear map and thus can not be a candidate for a
  connection on the  dual module $V^\prime$.
 This is why we associate to $\nabla\in \Con_A(V)$ the $\bbK$-linear map
 \begin{flalign}
\widehat{\nabla} : V^\prime \to\Hom_A(V,\Omega) ~,~~ v^\prime \mapsto \widehat{\nabla}(v^\prime) = \dif \circ v^\prime + \nabla^\Hom(v^\prime)~.
 \end{flalign}
 We observe that, for all $v^\prime\in V^\prime$, we have the right $A$-linearity property of
 $\widehat{\nabla}(v^\prime)$, for all $a\in A$ and $v\in V$,
\begin{flalign}
\nn\widehat{\dd}(v^\prime)(v\cdot a) &= \dif\bigl(v^\prime(v\cdot a)\bigr) -  (v^\prime\otimes\id)\dd(v\cdot a)\\
\nn&=\dif (v^\prime(v))\cdot a + v^\prime(v)\cdot \dif a - (v^\prime\otimes\id)\dd (v )\cdot a - (v^\prime\otimes\id)(v\otimes_A \dif a)\\
&=\big(\widehat{\dd}(v^\prime)(v)\big) \cdot a~.
\end{flalign}
Hence we have constructed a map
\begin{flalign}\label{eqn:hatmap}
\widehat{\,~}\, : \Con_A(V) \to \big( \Hom_\bbK(V^\prime,\Hom_A(V,\Omega)) \big)^\op~,~~\nabla\mapsto \widehat{\nabla} ~.
\end{flalign}
With the help of the $\MMMod{}{A}{}$-module  isomorphism $\iota^{-1}$ of
Proposition \ref{propo:projhomiso} we also  define
the map (with slight abuse of notation we denote it with the same
prime symbol used 
on homomorphisms in (\ref{eqn:primemaphom}))
\begin{flalign}\label{eqn:primemapcon}
{}^\prime: \Con_A(V) \to \big(\Hom_\bbK(V^\prime,\Omega\otimes_A V^\prime)\big)^\op~,~~\nabla \mapsto \nabla^\prime = \iota^{-1}\circ \widehat{\nabla}~.
\end{flalign}
Using (\ref{eqn:projhomiso}), the map $\dd^\prime$ acting on an
arbitrary $v^\prime\in V^\prime$ explicitly  reads 
\begin{flalign}
\label{eqn:nabprimeexplicit}
 \dd^\prime(v^\prime) =\iota^{-1}\big(\;\!\widehat{\dd}(v^\prime)\big) 
= \sum\limits_{i=1}^n\left(\dif (v^\prime(v_i)) - (v^\prime\otimes\id)\dd (v_i)\right) \otimes_Av_i^\prime ~.
\end{flalign}

We now show that $\dd^\prime$ is a connection on $V^\prime$.
\begin{Theorem}\label{theo:dualconn}
Let $V$ be a finitely generated and projective $\MMMod{}{}{A}$-module and let
 $\bigl(\Omega^\bullet,\wedge,\dif\bigr)$ be a  differential 
 calculus over $A$.
The $~^\prime$ maps in (\ref{eqn:primemapcon}) and in (\ref{eqn:primemaphom}) provide an affine space isomorphism
\begin{flalign}\label{eqn:conprimetheo}
{}^\prime: \Con_A(V) \to {}_A\Con(V^\prime)~,~~\nabla\mapsto \nabla^\prime~,
\end{flalign}
where $\Con_A(V)$ and ${}_A\Con(V^\prime)$ are affine spaces
over the isomorphic $\bbK$-modules $\Hom_A(V,V\otimes_A\Omega)$ and $\big({}_A\Hom(V^\prime,\Omega\otimes_A V^\prime)\big)^\op$, 
respectively.
\end{Theorem}
\begin{proof}
We first have to show that, for all $\nabla\in \Con_A(V)$, $\nabla^\prime$ defined by   (\ref{eqn:primemapcon})
is a connection on the $\MMMod{}{A}{}$-module $V^\prime$.
Using (\ref{eqn:nabprimeexplicit}) and left $A$-linearity of
$\dd^\Hom$ we have, for all $a\in A$ and $v^\prime\in V^\prime$,
\begin{flalign}
 \nn \dd^\prime(a\cdot v^\prime) &= \sum\limits_{i=1}^n \Bigl(\dif(a\,v^\prime(v_i)) - a\cdot \bigl((v^\prime\otimes\id)\dd (v_i)\bigr)\Bigr)\otimes_A v_i^\prime\\
\nn &=\sum\limits_{i=1}^n \Bigl(\dif a \cdot v^\prime(v_i) + a\cdot \dif( v^\prime(v_i)) - a\cdot \bigl((v^\prime\otimes\id)\dd (v_i)\bigr)\Bigr)\otimes_A v_i^\prime\\
\nn&=\dif a\otimes_A \sum\limits_{i=1}^n v_i^{\prime\prime}(v^\prime)\cdot v_i^\prime + a\cdot \dd^\prime(v^\prime) \\
&= \dif a\otimes_A v^\prime + a\cdot \dd^\prime(v^\prime)~,
\end{flalign}
where in the last two steps we have used also Proposition  \ref{propo:fgpmodprop}.

Next we show that the $~^\prime$ maps in  (\ref{eqn:conprimetheo}) and 
(\ref{eqn:primemaphom}) are an affine space map, 
for all $\nabla\in \Con_A(V)$, $P\in\Hom_A(V,V\otimes_A\Omega)$ and
$v^\prime\in V^\prime$,
\begin{flalign}
\nn (\nabla+P)^\prime(v^\prime) &= \iota^{-1}\big(\dif\circ v^\prime + (\nabla+P)^\Hom(v^\prime)\big) \\
& = 
\iota^{-1}\big(\dif\circ v^\prime + \nabla^\Hom(v^\prime) + P^\Hom(v^\prime)\big)  = (\nabla^\prime + P^\prime )(v^\prime)~,
\end{flalign}
where we have used that ${~}^\Hom$ is $\bbK$-linear.

Finally, the  $~^\prime$ maps in  (\ref{eqn:conprimetheo}) and 
(\ref{eqn:primemaphom}) are isomorphisms because the map $\,{}^\Hom\,$
in (\ref{eqn:Phommap}) is invertible when the $\MMMod{}{}{A}$-module
$V$ is finitely
generated and projective. Explicitly, for all 
$T\in \big({}_A\Hom(V^\prime,\Hom_\bbK(V,\Omega))\big)^\op$
and $v\in V$,  $\,T^{\,\Hom^{-1\!}}(v)=-\sum_{i=1}^n v_i\otimes_A T(v^\prime_i)(v)$,
where $\{v_i,v^\prime_i : i =1,\ldots,
n\}$ is a pair of dual bases for $V$. 
\end{proof}
\sk
\subsubsection*{Deformation of dual connections}
The above detailed preliminary discussion on dual connections is
helpful in order to understand their twist deformation.
In this case we have a Hopf algebra $H$ that acts on 
all the modules encountered in the beginning of this subsection.
We study these actions, in particular the  maps
${}^{\Hom}\,$, $\iota$, and their composition
${}^{\prime}$, are $H$-equivariant.
We then twist deform the
modules and consider the corresponding maps ${}^{\Hom_\star}\,$, 
$\iota_\star$, and their composition ${\,}^{\prime_\star}$. 
Their properties imply that deformation and dualization are
compatible operations: given a connection $\dd$, the dual of the 
twist deformed connection $\widetilde{D}_\FF(\dd)$ is equal to the 
twist deformation of the dual connection $\dd^\prime$.

\sk
Let $H$ be a Hopf algebra,   
$A$ be an $\AAAlg{H}{}{}$-algebra, $V$ an $\MMMod{H}{}{A}$-module and
$(\Omega^\bullet,\wedge,\dif)$ a left $H$-covariant
differential calculus over $A$. 
We recall that in this case $V\otimes_A\Omega$ is an $\MMMod{H}{}{A}$-module and
  $V^\prime,\,\Omega\otimes_A V^\prime,\,\Hom_\bbK(V,\Omega)$ are $\MMMod{H}{A}{}$-modules.
Also the modules appearing in (\ref{eqn:Phommap}) are $\MMMod{H}{}{}$-modules, in detail:
The $H$-action on $\Hom_\bbK(V,V\otimes_A\Omega)$ is given by the
$H$-adjoint action $\RA$ 
(obtained lifting the $H$-actions $\ra$ on $V$
and  $V\otimes_A\Omega$). 
The $H$-action on $\big({}_A\Hom(V^\prime,\Hom_\bbK(V,\Omega))\big)^\op$ is the
$\RA^\cop$ adjoint action, as required by the left
$A$-linearity of these homomorphisms (cf. Proposition
\ref{ABVAEnd}). It is obtained via a lift of the $H$-actions on $V^\prime$ and
$\Hom_\bbK(V,\Omega)$, these latter are themselves adjoint actions $\RA$,
obtained lifting the $H$-actions $\ra$ on  $V$ and $\Omega$.

Throughout this subsection we assume $V$ to be finitely generated and
projective as an $\MMMod{}{}{A}$-module. As proven in the end of
Theorem \ref{theo:dualconn}, the $\bbK$-linear map $\,{}^\Hom\,$ is then invertible.

\begin{Lemma}\label{lem:equivarianceHommap}
The $\bbK$-linear map $~^\Hom\,$ in (\ref{eqn:Phommap}) is an $\MMMod{H}{}{}$-module isomorphism, i.e.~for all $P\in\Hom_\bbK(V,V\otimes_A\Omega)$ and
all $\xi\in H$, $(\xi\RA P)^\Hom = \xi\RA^\cop P^\Hom$.
\end{Lemma}
\begin{proof}
Recalling definition
(\ref{eqn:Phom}) we obtain, for all $v^\prime\in V^\prime$ 
\begin{flalign}
\nn \big(\xi\RA^\cop P^\Hom\big)(v^\prime) &= \xi_2\RA\Big(P^\Hom\big(S^{-1}(\xi_1)\RA v^\prime \big)\Big)\\
\nn &= -\xi_2\RA\Big( \wedge\circ  \big(S^{-1}(\xi_1)\RA v^\prime\otimes \id\big) \circ P\Big)\\
\nn &=- \wedge\circ  \big( (\xi_2 S^{-1}(\xi_1))\RA v^\prime\otimes \id\big) \circ \xi_3\RA P \\
 &=- \wedge\circ  (v^\prime\otimes \id) \circ \xi\RA P = (\xi\RA P)^\Hom~.
\end{flalign}
In line three we used that $\wedge$ is $H$-equivariant.
\end{proof}
\begin{Corollary}\label{cor:equivarianceprimehatmaps}
Since $\iota$ is an $\MMMod{H}{A}{}$-module isomorphism (cf.~Proposition
\ref{propo:projhomiso}) 
it follows that the ~$^\prime$~map (see (\ref{eqn:primemaphom})) is also an $\MMMod{H}{}{}$-module isomorphism.
\end{Corollary}

We now consider a twist $\FF\in H\otimes H$ and deform the modules
relevant  in the study of dual connections. In particular, we 
have the $\MMMod{H^\FF}{A_\star}{}$-module isomorphisms
$D_\FF: V^\prime{\,}_\star\to  {V_\star\,}^\prime$ and
$D_\FF : \Hom_\bbK(V,\Omega)_\star, \to \Hom_\bbK(V_\star,\Omega_\star)$
that intertwine between the $\RA$ actions of $H^\FF$ on 
$\Hom_\bbK(V,\Omega)_\star$ and $V^\prime{\,}_\star$, and the  $\RA_\FF$
actions of $H^\FF$ on $ {V_\star\,}^\prime$ and
$\Hom_\bbK(V_\star,\Omega_\star)$.
These $H^\FF$-actions lift (cf. Proposition \ref{ABVAEnd}) to a
$\RA^\cop$ adjoint action of $H^\FF$ 
on
$\big(\Hom_\bbK({V^\prime}_{\,\star}, \Hom_\bbK(V,\Omega)_\star)\big)^\op$
and to a 
${\RA_{\!\FF}}^{\:\cop}$ adjoint action of $H^\FF$ 
on  
$\big( \Hom_\bbK({V_\star}^{\,\prime}, \Hom_\bbK(V_\star,\Omega_\star))\big)^\op$.

The two $\MMMod{H^\FF}{}{}$-modules 
$\big(\Hom_\bbK({V^\prime}_{\,\star}, \Hom_\bbK(V,\Omega)_\star)\big)^\op$
and 
$\big( \Hom_\bbK({V_\star}^{\,\prime}, \Hom_\bbK(V_\star,\Omega_\star))\big)^\op$
are isomorphic:
\begin{Lemma}\label{lem:AdDF}
Let $H$ be a Hopf algebra with twist $\FF\in H\otimes H$,
 $A$ be an $\AAAlg{H}{}{}$-algebra,
  $V$ be an $\MMMod{H}{}{A}$-module and $\Omega$ be an $\MMMod{H}{A}{A}$-module.
The $\bbK$-linear map 
 \begin{flalign}
 \nn \mathrm{Ad}_{D_\FF}:\big(\Hom_\bbK({V^\prime}_{\,\star}, \Hom_\bbK(V,\Omega)_\star)\big)^\op
 \quad &\longrightarrow \quad 
\big( \Hom_\bbK({V_\star}^{\,\prime}, \Hom_\bbK(V_\star,\Omega_\star))\big)^\op~,\\
 T\quad &\longmapsto  \quad \mathrm{Ad}_{D_\FF}(T) = D_\FF\circ T\circ D_\FF^{-1}~,
 \end{flalign}
 is an $\MMMod{H^\FF}{}{}$-module isomorphism.

Furthermore, it restricts  to
an $\MMMod{H^\FF}{}{}$-module isomorphism
\begin{flalign}
\mathrm{Ad}_{D_\FF}: \big({}_{A_\star}\Hom({V^\prime}_{\,\star}, \Hom_\bbK(V,\Omega)_\star)\big)^\op
 \quad &\longrightarrow \quad 
\big( {}_{A_\star}\Hom({V_\star}^{\,\prime}, \Hom_\bbK(V_\star,\Omega_\star))\big)^\op~.
\end{flalign}
\end{Lemma}
\begin{proof}
For all $T\in \big(\Hom_\bbK({V^\prime}_{\,\star},
\Hom_\bbK(V,\Omega)_\star)\big)^\op$ and $\xi\in H$ we have
\begin{flalign}\label{ADintertwi}
\nn \xi \,{\RA_{\!\FF}}^{\:\cop}\,\big( \mathrm{Ad}_{D_\FF}(T) \big)&= \xi_{2_\FF}\RA_\FF\, \circ \, \mathrm{Ad}_{D_\FF}(T)  \circ\ S^{\FF\,-1}(\xi_{1_\FF})\RA_\FF\, \\
\nn &= \xi_{2_\FF}\RA_\FF\,\circ \,D_\FF\circ  T  \circ D_\FF^{-1}\circ S^{\FF\,-1}(\xi_{1_\FF})\RA_\FF\, \\
\nn &=D_\FF\circ \xi_{2_\FF}\RA\,\circ \, T  \circ  S^{\FF\,-1}(\xi_{1_\FF})\RA\, \circ \,D_\FF^{-1}\\
&= \mathrm{Ad}_{D_\FF}\big(\xi\, {\RA}^{\cop}\, T\big)~,
\end{flalign}
where we used that $D_\FF$ intertwines between the $\RA$ and  $\RA_\FF$ 
adjoint actions,
$D_\FF\circ \xi\RA~ = \xi\RA_\FF~\circ D_\FF$ (cf.~Theorem
\ref{isoAAstDAA}). The map $\mathrm{Ad}_{D_\FF}$ is obviously invertible.

 Let now $T\in \big({}_{A_\star}\Hom({V^\prime}_{\,\star}, \Hom_\bbK(V,\Omega)_\star)\big)^\op$, i.e.~for
all $a\in A_\star$ and $v^\prime \in {V^\prime}_{\,\star}$, $T(a\star v^\prime)  = a\star T(v^\prime)$.
Then, for all $v^\prime _\star\in {V_\star}^{\,\prime} $,
\begin{flalign}
\nn \big(\mathrm{Ad}_{D_\FF}(T)\big)(a\cdot v^\prime_\star) &= D_\FF\Big(T\big(D_\FF^{-1}(a\cdot v^\prime_\star)\big)\Big)=
D_\FF\Big(T\big(a\star D_\FF^{-1}(v^\prime_\star)\big)\Big)\\
\nn &=
D_\FF\Big(a\star\big( T\big( D_\FF^{-1}(v^\prime_\star)\big)\big)\Big)=
a\cdot D_\FF\Big(T\big( D_\FF^{-1}(v^\prime_\star)\big)\Big)\\
&=a\cdot \Big(\big(\mathrm{Ad}_{D_\FF}(T)\big)(v^\prime_\star)\Big)~,
\end{flalign}
where  in the second and fourth equality we used  that $D_\FF$ is an $\MMMod{H^\FF}{A_\star}{}$-module isomorphism.
\end{proof}
\begin{Remark}\label{rem:abuseAdDF}
We later consider a slight variation of the map $\mathrm{Ad}_{D_\FF}$, namely the $\bbK$-linear map (denoted  with
 abuse of notation also by $\mathrm{Ad}_{D_\FF}$)
\begin{flalign}
 \nn\mathrm{Ad}_{D_\FF}:  \big(\Hom_\bbK({V^\prime}_{\,\star}, \Omega_\star\otimes_{A_\star} {V^\prime}_{\,\star}) \big)^\op
 \quad &\longrightarrow \quad 
 \big(\Hom_\bbK({V_\star}^{\,\prime}, \Omega_\star\otimes_{A_\star} {V_\star}^{\,\prime})\big)^\op~,\\
 T\quad &\longmapsto  \quad\mathrm{Ad}_{D_\FF}(T)=  (\id\otimes_\star D_\FF)\circ T\circ D_\FF^{-1}~.
 \end{flalign}
 With a similar calculation as in Lemma \ref{lem:AdDF} one shows 
that this map is an $\MMMod{H^\FF}{}{}$-module isomorphism 
and  that it restricts to the  $\MMMod{H^\FF}{}{}$-module isomorphism (denoted by the same symbol)
 \begin{flalign}
 \mathrm{Ad}_{D_\FF}: \big({}_{A_\star}\Hom({V^\prime}_{\,\star}, \Omega_\star\otimes_{A_\star} {V^\prime}_{\,\star})\big)^\op
 \quad &\longrightarrow \quad 
\big( {}_{A_\star}\Hom({V_\star}^{\,\prime}, \Omega_\star\otimes_{A_\star} {V_\star}^{\,\prime})\big)^\op~.
 \end{flalign}
\end{Remark}
\sk

As a first step towards the twist deformation of dual connections we
study the map ${~}^{\Hom_\star}\,$ between
homomorphisms of deformed modules.  Let $\FF\in H\otimes H$ be a twist of the Hopf algebra $H$. For
all $P_\star\in \Hom_\bbK(V_\star,V_\star\otimes_{A_\star}\Omega_\star)$ we define
the map $P_\star^{\Hom_\star}\in \big({}_{A_\star}\Hom({V_\star}^{\,\prime},\Hom_\bbK(V_\star,\Omega_\star))\big)^\op$
as in  (\ref{eqn:Phom}) by, for all $v_\star^\prime\in {V_\star}^{\,\prime}$,
\begin{flalign}\label{eqn:Phomstar}
P_\star^{\Hom_\star}(v_\star^\prime) := -\wedge_\star\circ\,( v_\star^\prime\otimes_\star\id)\circ P_\star~.
\end{flalign}
As in Lemma \ref{lem:equivarianceHommap} we have an $\MMMod{H^\FF}{}{}$-module isomorphism
\begin{flalign}
{}^{\Hom_\star}: \Hom_\bbK(V_\star,V_\star\otimes_{A_\star}\Omega_\star)\longrightarrow \big({}_{A_\star}\Hom({V_\star}^{\,\prime},\Hom_\bbK(V_\star,\Omega_\star))\big)^\op~,~~P_\star\longmapsto P_\star^{\Hom_\star}~.
\end{flalign}
We compare this map with the map $\;^\Hom\,$, that, due to its
$H$-equivariance, can be considered as the $\MMMod{H^\FF}{}{}$-module isomorphism 
\begin{flalign}
{}^\Hom\,:\Hom_\bbK(V,V\otimes_A\Omega)_\star \longrightarrow
\big({}_{A}\Hom(V^\prime,\Hom_\bbK(V,\Omega))\big)^\op_{~~\star} ~\,,~~P
\longmapsto P^{\Hom}~.
\end{flalign}
 \begin{Proposition}\label{propo:HomHomstarcom}
 Let $\FF\in H\otimes H$ be a twist of the Hopf algebra $H$.  Then the following diagram of $\MMMod{H^\FF}{}{}$-module
  isomorphisms commutes:
 \begin{flalign}
 \xymatrix{
 \Hom_\bbK(V,V\otimes_A\Omega)_\star \ar[dd]_-{\widetilde{D}_\FF}\ar[rrr]^-{{}^\Hom} & & & \big({}_{A}\Hom(V^\prime,\Hom_\bbK(V,\Omega))\big)^\op_{~~\star} \ar[d]^-{D_\FF^\cop}\\
 &&& \big({}_{A_\star}\Hom({V^\prime}_{\,\star}, \Hom_\bbK(V,\Omega)_\star) \big)^\op\ar[d]^-{\mathrm{Ad}_{D_\FF}}\\
 \Hom_\bbK(V_\star,V_\star\otimes_{A_\star}\Omega_\star) \ar[rrr]^-{{}^{\Hom_\star}} &&& \big( {}_{A_\star}\Hom({V_\star}^{\,\prime}, \Hom_\bbK(V_\star,\Omega_\star)) \big)^\op
 }
 \end{flalign}
 \end{Proposition}
 \begin{proof}
We consider any $P\in \Hom_\bbK(V,V\otimes_A\Omega)_\star$, follow the
upper path in the diagram and act on an arbitrary element
$v_\star^\prime\in {V_\star}^{\,\prime}$, 
 \begin{flalign}\label{6.80}
\nn \Big( \mathrm{Ad}_{D_\FF}\big(D_\FF^\cop(P^\Hom)\big)\Big)(v_\star^\prime) &= D_\FF\Big(D_\FF^\cop(P^\Hom)\big(D_\FF^{-1}(v_\star^\prime)\big)\Big)\\
\nn &=D_\FF\Big((\of_\alpha\RA^\cop P^\Hom)\big(\of^\alpha\RA D_\FF^{-1}(v_\star^\prime)\big)\Big)\\
\nn &=D_\FF\Big((\of_\alpha\RA P)^\Hom\,\big(\of^\alpha\RA D_\FF^{-1}(v_\star^\prime)\big)\Big)\\
\nn &=- D_\FF\Big(\wedge \circ_{\,} \big((\of^\alpha\RA D_\FF^{-1}(v_\star^\prime))\otimes\id\big)\circ (\of_\alpha\RA P)\Big)\\
\nn &= -D_\FF\Big(\wedge\circ_\star \big(D_\FF^{-1}(v_\star^\prime)\otimes\id\big)\circ_\star P\Big)\\
\nn &= -\wedge\circ\, D_\FF\Big(D_\FF^{-1}(v_\star^\prime)\otimes\id\Big)\circ D_\FF(P)\\
\nn &=-\wedge\circ \,\varphi \circ \big(v_\star^\prime\otimes_\star \id\big)\circ\varphi^{-1}\circ D_\FF(P) \\
&= -\wedge_\star 
\circ \,\big(v_\star^\prime\otimes_\star \id\big)\circ \widetilde{D}_\FF(P)~.
 \end{flalign}
 In the third line we have used Lemma \ref{lem:equivarianceHommap}, in the fourth line (\ref{eqn:Phom}), in the fifth line that
 $\wedge$ is $H$-equivariant and in the sixth line this property and that
  $D_\FF(Q\circ_\star \check Q) = D_\FF(Q)\circ D_\FF(\check Q)$, which holds for arbitrary composable $\bbK$-linear 
  maps $Q$ and $\check Q$.
  Then in line seven we used Theorem \ref{theo:promodhomdefA} and that $\id$ is $H$-equivariant.
  The last passage follows by noticing that $\wedge_\star =
  \wedge\circ \varphi$ and recalling the definition
   of $\widetilde{D}_\FF$ (\ref{DtildeFHOM}).

We now follow the lower path in the diagram, from
(\ref{eqn:Phomstar}) we immediately obtain (\ref{6.80}), and hence
commutativity of the diagram.
\end{proof}
 \begin{Remark}\label{rem:homcomdiag}
By restricting to right $A$-linear maps $\Hom_A(V,V\otimes_A\Omega)_\star$, we also obtain 
the following commutative diagram of $\MMMod{H^\FF}{}{}$-module isomorphisms:
 \begin{flalign}
 \xymatrix{
 \Hom_A(V,V\otimes_A\Omega)_\star \ar[dd]_-{\widetilde{D}_\FF}\ar[rrr]^-{{}^\Hom} & & & \big({}_{A}\Hom(V^\prime,\Hom_A(V,\Omega))\big)^\op_{~~\star} \ar[d]^-{D_\FF^\cop}\\
 &&& \big({}_{A_\star}\Hom({V^\prime}_{\,\star}, \Hom_A(V,\Omega)_\star) \big)^\op\ar[d]^-{\mathrm{Ad}_{D_\FF}}\\
 \Hom_{A_\star}(V_\star,V_\star\otimes_{A_\star}\Omega_\star) \ar[rrr]^-{{}^{\Hom_\star}} &&& \big({}_{A_\star}\Hom({V_\star}^{\,\prime}, \Hom_{A_\star}(V_\star,\Omega_\star))\big)^\op
 }
 \end{flalign}
 \end{Remark} 
 \sk
 
We now consider the map
 \begin{flalign}
 \widehat{{}{~~}}^\star : \Con_{A_\star}(V_\star) \longrightarrow\big(\Hom_\bbK({V_\star}^{\,\prime}, \Hom_{A_\star}(V_\star,\Omega_\star)) \big)^\op
\end{flalign}
defined by,
 for all $\nabla_{\!\star}\in\Con_{A_\star}(V_\star)$ and
 $v_\star^\prime\in {V_\star}^{\,\prime}$,
\begin{flalign}\label{eqn:hatmapstar}
 {\widehat{\,\nabla_{\!\star}}}^{\!\star} (v_\star^\prime) := \dif\circ v_\star^\prime + \nabla^{\Hom_\star}(v_\star^\prime)~.
 \end{flalign}
The corresponding  map $\widehat{{}{~~}}$ , defined
in (\ref{eqn:hatmap}), 
can be seen as a map
 \begin{flalign}\
\widehat{~~\,}: \Con_{A}(V)_\star\longrightarrow  \big(\Hom_\bbK(V^\prime,\Hom_A(V,\Omega))\big)^\op_{~~\star} ~,
\end{flalign}
where we recall that $\Con_{A}(V)_\star$ differs from $\Con_{A}(V)$ just because it
is seen as an affine space over the  $\MMMod{H^\FF}{}{}$-module $
\Hom_A(V,V\otimes_A\Omega)_\star $ rather than over the $\MMMod{H}{}{}$-module  $
\Hom_A(V,V\otimes_A\Omega)$.

It is easy to see that the $\widehat{~~\,}$ and $\widehat{~~\,}^\star$
maps close a commutative diagram, where the vertical
 arrows are given by deformation maps $D_\FF$, $\widetilde{D}_\FF$ and $D_\FF^\cop$.
\begin{Corollary}\label{cor:widehatcomdiag}
  Let $\FF\in H\otimes H$ be a twist of the Hopf algebra $H$.  Then the following diagram commutes:
 \begin{flalign}
 \xymatrix{
 \Con_{A}(V)_\star \ar[dd]_-{\widetilde{D}_\FF}\ar[rrr]^-{\widehat{~~\,}} & & & \big(\Hom_\bbK(V^\prime,\Hom_A(V,\Omega))\big)^\op_{~~\star} \ar[d]^-{D_\FF^\cop}\\
 &&& \big(\Hom_\bbK({V^\prime}_{\,\star}, \Hom_A(V,\Omega)_\star)\big)^\op \ar[d]^-{\mathrm{Ad}_{D_\FF}}\\
 \Con_{A_\star}(V_\star) \ar[rrr]^-{\widehat{~~\,}^\star} &&& \big(\Hom_\bbK({V_\star}^{\,\prime}, \Hom_{A_\star}(V_\star,\Omega_\star)) \big)^\op
 }
 \end{flalign}
 \end{Corollary}
 \begin{proof}
From Proposition \ref{propo:HomHomstarcom} we know that this diagram
holds for the second addend in
the $\widehat{~~}$ and $\widehat{~~}^\star$ maps in
(\ref{eqn:hatmap}) and (\ref{eqn:hatmapstar}). Since the vertical
arrows $D^\cop_\FF$ and $\mathrm{Ad}_{D_\FF}$
are $\bbK$-linear maps, it remains to check the first addend. 
Use of $H$-equivariance of the differential $\dif$ then implies,  for all $v_\star^\prime\in {V_\star}^{\,\prime}$,
 \begin{flalign}
 \Big( \mathrm{Ad}_{D_\FF}\big(D_\FF^\cop(\widehat{\nabla} - \nabla^\Hom)\big)\Big)(v_\star^\prime)  = \dif \circ v_\star^\prime
 =\big(\widehat{\,\widetilde{D}_\FF(\nabla)\,}^{\!\,\!\!\!\!\star} - \big(\widetilde{D}_{\FF}(\nabla)\big)^{\Hom_\star}\big)(v_\star^\prime)~.
 \end{flalign}
 \end{proof}
 
 The second step in the study of the deformation of the dual
 connection $\nabla^\prime=\iota^{-1}\circ \widehat{\nabla}$ of Theorem \ref{theo:dualconn}
 is the study of the relation between the $\iota$ map of Proposition
 \ref{propo:projhomiso} and the corresponding one $\iota_\star$ between deformed modules.
It is here that finitely generated and projective modules are
needed. We denote by 
\eq\label{iotastar}
\iota_\star: \Omega_\star \otimes_{A_\star}
{V_\star}^{\,\prime}\to\Hom_{A_\star}(V_\star, \Omega_\star)
\en
the $\MMMod{H^\FF}{A_\star}{}$-module isomorphism defined by, for all $v\in V_\star$, $\omega\in \Omega_\star$,
 $v_\star^\prime\in {V_\star}^{\,\prime}$,
$\bigl(\iota_\star(\omega\otimes_{A_\star}v_\star^\prime)\bigr)(v):= \omega\star v_\star^\prime(v)$.
The maps $\iota$ and $\iota_\star$ close a commutative diagram where the vertical
 arrows are given by deformation maps $D_\FF$ and $\varphi$.
\begin{Lemma}\label{lem:varphistar}
Let $\FF\in H\otimes H$ be a twist of the Hopf algebra
$H$. Then the following diagram of  $\MMMod{H^\FF}{A_\star}{}$-module isomorphisms commutes:
\begin{flalign}\label{eqn:varphistar}
 \xymatrix{
\Omega_\star\otimes_{A_\star} {V_\star}^{\,\prime} \ar[rr]^-{\iota_\star} \ar[d]_-{\id\otimes_\star D_\FF^{-1}} & &\Hom_{A_\star}(V_\star,\Omega_\star)\ar[dd]^-{D_\FF^{-1}}\\
 \Omega_\star\otimes_{A_\star} {V^\prime}_{\,\star }\ar[d]_-{\varphi} & &\\
 (\Omega\otimes_A V^\prime)_\star \ar[rr]^-{\iota_{}} & & (\Hom_A(V,\Omega))_\star 
}
\end{flalign}
\end{Lemma}
\begin{proof}
 Using the explicit expression for $D_\FF^{-1}$ (see Remark \ref{rem:Dinv}), we find, when following the upper path,
for all $v\in V_\star$, $\omega\in \Omega_\star$ and $v_\star^\prime\in {V_\star}^{\,\prime}$,
\begin{flalign}
\nn \Bigl(D_\FF^{-1}\bigl(\iota_\star(\omega \otimes_{A_\star}v_\star^\prime)\bigr)\Bigr)(v) &= \of^\alpha\ra \Bigl(\omega\star v_\star^\prime \bigl(\chi S(\of_\alpha)\ra v\bigr)\Bigr)\\
\nn &= (\of^\alpha_1\of^\beta\ra \omega) \cdot
\of^\alpha_2\of_\beta\ra \Big( v_\star^\prime\bigl(\chi S(\of_\alpha)\ra v\bigr)\Big)\\
\nn &= (\of^\alpha\ra \omega) \cdot \of_{\alpha_1}\of^\beta\ra \Big( v_\star^\prime\bigl(\chi S(\of_\beta) S(\of_{\alpha_2})\ra v\bigr)\Big)\\
&= (\of^\alpha\ra \omega) \cdot \bigl(\of_\alpha\RA D_\FF^{-1}(v_\star^\prime)\bigr)(v)~,
\end{flalign}
where in line three we have used the twist cocycle property (\ref{propF1}).
Following the lower path we find the same expression, for all $v\in V_\star$, $\omega\in \Omega_\star$ and $v_\star^\prime\in 
 {V_\star}^{\,\prime}$,
\begin{flalign}
\nn \iota\Bigl(\varphi\bigl(\omega\otimes_{A_\star}D_\FF^{-1}(v_\star^\prime)\bigr) \Bigr) (v) 
&=\iota\Bigl((\of^\alpha\ra \omega)\otimes_{A}\bigl(\of_\alpha\RA D_\FF^{-1}(v_\star^\prime)\bigr)  \Bigr)(v) \\
&=(\of^\alpha\ra \omega) \cdot \bigl(\of_\alpha\RA D_\FF^{-1}(v_\star^\prime)\bigr)(v)~.
\end{flalign}
\end{proof} 
It is convenient to consider the following $\MMMod{H}{}{}$-module isomorphism induced by 
the $\MMMod{H}{A}{}$-module isomorphism $\iota^{-1}$ (and denoted with
a slight abuse of notation by the same symbol)
\begin{flalign}\label{TiotaT92}
\iota^{-1}:\big( \Hom_\bbK(V^\prime,\Hom_A(V,\Omega) )  \big)^\op \to \big(\Hom_\bbK(V^\prime,\Omega\otimes_A V^\prime)\big)^\op~,~~T\mapsto \iota^{-1}\circ T~.
\end{flalign}
With this definition the maps $\,^\prime$ in
(\ref{eqn:primemaphom}) and (\ref{eqn:primemapcon}) are just
the compositions 
$\,^\prime=\iota^{-1}\circ\, ^\Hom$ and  $\,^\prime=\iota^{-1}\circ\,
\widehat{~~\,}\:$. 
Since the map $\iota^{-1}$ is $H$-equivariant,  it can  also be seen as an
isomorphism  between $\MMMod{H^\FF}{}{}$-modules, 
$
\iota^{-1}:\big( \Hom_\bbK(V^\prime,\Hom_A(V,\Omega) )  \big)^\op_{~~\star} \to \big(\Hom_\bbK(V^\prime,\Omega\otimes_A V^\prime)\big)^\op_{~~\star}\,.
$

Similarly the isomorphism $\iota_\star$ in (\ref{iotastar}) induces
the  $\MMMod{H^\FF}{}{}$-module  isomorphism 
\begin{flalign}
\iota_\star^{-1}: \big(\Hom_\bbK( {V_\star}^{\,\prime},\Hom_{A_\star}(V_\star,\Omega_\star) )  \big)^\op \to \big( \Hom_\bbK( {V_\star}^{\,\prime},\Omega_\star \otimes_{A_\star}  {V_\star}^{\,\prime}) \big)^\op ~,~~T_\star\mapsto \iota^{-1}_\star\circ T_\star~.
\label{iotastar-1}\end{flalign}
\begin{Proposition}\label{propo:iotacomdiag}
Let $\FF\in H\otimes H$ be a twist of the Hopf algebra
$H$. Then the following diagram of $\MMMod{H^\FF}{}{}$-module isomorphisms
commutes: 
\begin{flalign}
\xymatrix{
\ar[d]_-{D_\FF^\cop}\big(\Hom_\bbK(V^\prime,\Hom_A(V,\Omega))\big)^\op_{~~\star} \ar[rrr]^-{\iota_{}^{-1}} & &  & \big(\Hom_\bbK(V^\prime,\Omega\otimes_A V^\prime)\big)^\op_{~~\star} \ar[d]^-{\widetilde{D}_\FF^\cop}\\
\big(\Hom_\bbK({V^\prime}_{\,\star},\Hom_A(V,\Omega)_\star) \big)^\op\ar[d]_-{\mathrm{Ad}_{D_\FF}}& & & \big(\Hom_\bbK({V^\prime}_{\,\star}, \Omega_\star\otimes_{A_\star} {V^\prime}_{\,\star}) \big)^\op\ar[d]^-{\mathrm{Ad}_{D_\FF}}\\
\big(\Hom_\bbK( {V_\star}^{\,\prime},\Hom_{A_\star}(V_\star,\Omega_\star)) \big)^\op \ar[rrr]^-{\iota_\star^{-1}} &&&  \big(\Hom_\bbK( {V_\star}^{\,\prime}, \Omega_\star\otimes_{A_\star} {V_\star}^{\,\prime}) \big)^\op
}
\end{flalign}
(where we recall the abuse of notation in the definition of
$\mathrm{Ad}_{D_\FF}$ in Remark \ref{rem:abuseAdDF}).
\end{Proposition}
\begin{proof}
We consider an arbitrary $T\in\big(\Hom_\bbK(V^\prime,\Hom_A(V,\Omega))\big)^\op_{~~\star}$, 
follow the upper path in the diagram
and act on an arbitrary $v_\star^\prime\in {V_\star}^{\,\prime}$, 
\begin{flalign}
\nn \mathrm{Ad}_{D_\FF}\Big(\widetilde{D}_\FF^\cop\big(\iota^{-1}(T)\big)\Big)(v_\star^\prime) &= (\id\otimes_\star D_\FF)\Big(\widetilde{D}_\FF^\cop\big(\iota^{-1}(T)\big)\big(D_\FF^{-1}(v_\star^\prime)\big)\Big)\\
\nn &= \big((\id\otimes_\star D_\FF)\circ \varphi^{-1}\big)\Big(\big(\of_\alpha\RA^\cop \iota^{-1}(T)\big)\big(\of^\alpha\RA D_\FF^{-1}(v_\star^\prime)\big)\Big)\\
\nn &=\big((\id\otimes_\star D_\FF)\circ \varphi^{-1}\circ \iota^{-1}\big)\Big(\big(\of_\alpha\RA^\cop T\big)\big(\of^\alpha\RA D_\FF^{-1}(v_\star^\prime)\big)\Big)\\
\nn &=\big(\iota_\star^{-1}\circ D_\FF\big)\Big(D_\FF^\cop( T)\big(D_\FF^{-1}(v_\star^\prime)\big)\Big)\\
&=\iota_\star^{-1}\Big(\mathrm{Ad}_{D_\FF}\big(D_\FF^\cop(T)\big)\Big)(v_\star^\prime)~,
\end{flalign}
which is exactly the expression we obtain by following the lower path in the diagram.
In line two we have used the definition of $\widetilde{D}_\FF^\cop$ (cf.~Theorem \ref{theo:condefleft}),
in line three that $\iota^{-1}$ is $H$-equivariant and then
(\ref{TiotaT92}), finally in line four Lemma \ref{lem:varphistar} in the
form $(\id\otimes_\star D_\FF)\circ \varphi^{-1}\circ \iota^{-1}=\iota_\star^{-1}\circ D_\FF$.
\end{proof}
The   ${~}^{\prime_\star}$ maps between deformed modules corresponding
to the prime maps in (\ref{eqn:primemapcon}) and
(\ref{eqn:primemaphom}) (that with a slight abuse of notation are denoted with the same symbol) are defined by
\begin{subequations}
\begin{flalign}
{}^{\prime_\star}: \Con_{A_\star}(V_\star) \to {}_{A_\star}\Con(
{V_\star}^{\,\prime})~,~~\nabla_\star\mapsto
\nabla_\star^{\,\prime_\star} = \iota^{-1}_\star\circ 
 {\widehat{\,\nabla_{\!\star}}}^{\!\star}
\end{flalign}
and
\begin{flalign}
{}^{\prime_\star}: \Hom_{A_\star}(V_\star,V_\star\otimes_{A_\star}\Omega_\star) 
\to\big( {}_{A_\star}\Hom( {V_\star}^{\,\prime},\Omega_\star\otimes_{A_\star} {V_\star}^{\,\prime})\big)^\op~,~~P_\star\mapsto P_\star^{\,\prime_\star} = \iota^{-1}_\star\circ P_\star^{\Hom_\star}~,
\end{flalign}
\end{subequations}
or, using the $\iota^{-1}_\star$ map in (\ref{iotastar-1}), simply by the
compositions
$~^\prime{}^\star=\iota_\star^{-1}\circ\, ^{\Hom_\star}\,$ and $~^\prime{}^\star=\iota_\star^{-1}\circ \widehat{~~\,}^\star\;$.

Theorem \ref{theo:dualconn} and  Lemma \ref{lem:equivarianceHommap}
imply that these maps constitute an affine space isomorphism, where
$\Con_{A_\star}(V_\star)$ and $ {}_{A_\star}\Con(
{V_\star}^{\,\prime})$ are affine spaces respectively over the
$\MMMod{H^\FF}{}{}$-modules of right $A_\star$-linear maps
$\Hom_{A_\star}(V_\star,V_\star\otimes_{A_\star}\Omega_\star) $ and of left $A_\star$-linear maps
$\big( {}_{A_\star}\Hom( {V_\star}^{\,\prime},\Omega_\star\otimes_{A_\star} {V_\star}^{\,\prime})\big)^\op$.
\sk
We can now finally prove the main theorem of this subsection stating  that deforming the dual connection is 
equivalent to dualizing the deformed connection.
\begin{Theorem}\label{theo:dualcondef}
 Let $H$ be a Hopf algebra with twist $\FF\in H\otimes H$, $A$ be an $\AAAlg{H}{}{}$-algebra,
 $V$ be a finitely generated and projective $\MMMod{H}{}{A}$-module and
$\bigl(\Omega^\bullet,\wedge,\dif\bigr)$ be a left $H$-covariant differential calculus over $A$.
Then the following diagram of isomorphisms between affine spaces of
connections over $\MMMod{H^\FF}{}{}$-modules  commutes:
\begin{flalign}
\xymatrix{
\Con_A(V) _\star\ar[rrr]^-{^\prime}\ar[dd]_-{\widetilde{D}_\FF} & & & {_A}\Con(V^\prime)_\star \ar[d]^-{\widetilde{D}_\FF^\cop} \\
&&& {_{A_\star}}\Con({V^\prime}_{\,\star}) \ar[d]^-{\mathrm{Ad}_{D_\FF}}\\
\Con_{A_\star}(V_\star) \ar[rrr]^-{^{\prime_\star}} & & & {_{A_\star}}\Con({V_\star}^{\,\prime})
} 
\end{flalign}
\end{Theorem}
\begin{proof}
Combining the commutative diagrams of Corollary \ref{cor:widehatcomdiag} and Proposition \ref{propo:iotacomdiag}
we obtain commutativity of the above diagram regarded as maps. The
commutativity of the diagram regarded as affine space maps follows by
combining the  diagrams of Remark
\ref{rem:homcomdiag} and Proposition \ref{propo:iotacomdiag}. The
vertical maps are obviously invertible, the horizontal ones are also
isomorphisms (recall  Theorem \ref{theo:dualconn}).
\end{proof}
\sk
\subsubsection*{Dual connections on $\MMMod{H}{A}{A}$-modules} 
Let us consider the case where $V$ is an $\MMMod{H}{A}{A}$-module, then
all the modules encountered in this subsection are also
$\MMMod{H}{A}{A}$-modules, we list them:
$V^\prime,\,\Omega\otimes_A V^\prime ,\,
\Hom_\bbK(V,\Omega),\,\Hom_A(V,\Omega)$, 
 $\Hom_\bbK(V,V\otimes_A\Omega)$, $\Hom_A(V,V\otimes_A\Omega)$,
$\big(\Hom_\bbK(V^\prime,\Omega\otimes_A V^\prime)\big)^\op$,
$\big({}_A\Hom(V^\prime,\Omega\otimes_A V^\prime)\big)^\op$ as well as
$\big(\Hom_\bbK({V^\prime}, \Hom_\bbK(V,\Omega))\big)^\op$ and
$\big({}_A\Hom({V^\prime}, \Hom_\bbK(V,\Omega))\big)^\op$. For example it follows from Proposition
\ref{ABVAEnd} that for all $T\in
\big(\Hom_\bbK({V^\prime}, \Hom_\bbK(V,\Omega))\big)^\op$ and $a\in A$
the right $A$-module structure is given by  $T\cdot^\op a= \boldsymbol{r}_a\circ T$, where $\boldsymbol{r}:A\to
{}_A\End\big(\Hom_\bbK(V,\Omega)\big)$ and, for all $P\in
\Hom_\bbK(V,\Omega) $,
$\boldsymbol{r}_a(P)=P\circ l_a$.
Similarly the left $A$-module structure is given by $a\cdot ^\op T=T\circ \boldsymbol{r}_a$ where now $\boldsymbol{r}:A\to
{}_A\End(V^\prime)$.

It is now straightforward to prove that in this case the isomorphisms $\iota$ in (\ref{iotadef}), 
${}^\Hom$ in (\ref{eqn:Phommap}) and (\ref{eqn:Phommaprestr}), and hence ${~}^\prime$
in (\ref{eqn:primemaphom}) are  $\MMMod{H}{A}{A}$-module
isomorphisms. 
Thus we have
\begin{Corollary}\label{coro:dualconnbi}
If $V$ is an $\MMMod{H}{A}{A}$-module, the isomorphism ${~}^\prime: \Con_A(V) \to {}_A\Con(V^\prime)$
of Theorem \ref{theo:dualconn} is  an affine space isomorphism,
where  $\Con_A(V)$ and ${}_A\Con(V^\prime)$ are 
affine spaces over the isomorphic $\MMMod{H}{A}{A}$-modules $\Hom_A(V,V\otimes_A\Omega)$ and $\big({}_A\Hom(V^\prime,\Omega\otimes_A V^\prime)\big)^\op$, respectively.
\end{Corollary}

\sk
Similarly, twist deformation of connections on $\MMMod{H}{A}{A}$-modules leads to
Theorem \ref{theo:dualcondef},  where now all the maps in the
commutative diagram are isomorphisms between affine spaces of
connections over $\MMMod{H^\FF}{A_\star}{A_\star}$-modules.
Indeed, all the maps  in the commutative diagram 
of $\MMMod{H^\FF}{A_\star}{A_\star}$-modules underlying the diagram in 
Theorem \ref{theo:dualcondef} 
are  $\MMMod{H^\FF}{A_\star}{A_\star}$-module
isomorphisms. (The map $\mathrm{Ad}_{D_\FF}$ is an
 $\MMMod{H^\FF}{A_\star}{A_\star}$-module  isomorphisms 
simply because all the
other maps in this commutative diagram are 
 $\MMMod{H^\FF}{A_\star}{A_\star}$-module isomorphisms).

\subsection{Connections on
 tensor products of modules and dual modules}
Combining the results of 
Corollary \ref{coro:dualconnbi}  on dual connections and of  Theorem
\ref{theo:leftrightconiso} on right to left  connections we immediately
obtain the following
 \begin{Corollary}\label{cor:rightdualcon}
 Let $(H,\RR)$ be a quasitriangular Hopf algebra, $A$  be a quasi-commutative $\AAAlg{H}{}{}$-algebra,
 $V$ be a strong quasi-commutative $\MMMod{H}{A}{A}$-module and
$(\Omega^\bullet,\wedge,\dif)$ be  a graded quasi-commutative left $H$-covariant differential calculus over $A$.
 Let further $V$ be finitely generated and projective
 as an $\MMMod{}{}{A}$-module and denote by $V'=\Hom_A(V,A)$ the dual $\MMMod{H}{A}{A}$-module of $V$.
  Then there is an affine space isomorphism
 \begin{flalign}
 \xymatrix{
 \Con_A(V)\ar[rr]^-{{}^\prime} & & {}_A\Con(V^\prime) \ar[rr]^-{\widetilde{D}_\RR^{-1}} & & \Con_A(V^\prime)~
}~,
 \end{flalign}
where the sets of right connections $\Con_A(V)$ and $\Con_A(V^\prime)$ are affine spaces over  the isomorphic
$\MMMod{H}{A}{A}$-modules  $\Hom_A(V,V\otimes_A\Omega)$ and
$\Hom_A(V^\prime,V^\prime\otimes_A\Omega)$, respectively.
Explicitly, given a right connection $\dd\in \Con_A(V)$ the right connection
on the dual module is $\widetilde{D}_\RR^{-1}\big(\nabla^\prime\big)\in\Con_A(V^\prime)$.
\end{Corollary}

This result allows  to induce from a connection on $V$ a connection on
the tensor algebra generated by $V$ and $V^\prime$. In other words, given a connection
on vector fields (elements of $V$) we extend it to a connection on
covariant and contravariant tensor fields (elements of the tensor
algebra generated by $V$ and $V'$).

\begin{Corollary}\label{coro:conntensoralgebra}
In the hypotheses of Corollary \ref{cor:rightdualcon}, 
given a right connection $\nabla\in \Con_A(V)$ we induce the right
connection 
$\widetilde{D}_\RR^{-1}\big(\nabla^\prime\big)\in\Con_A(V^\prime)$. 
Sums of the $\dd$ and
$\widetilde{D}_\RR^{-1}\big(\nabla^\prime\big)$
connections as in Theorem \ref{theo:conplus} allow  to extend the
connection $\dd$ to  the tensor algebra generated by $V$ and
$V^\prime$.
\end{Corollary}
\begin{proof}
This construction is unambiguous due to the associativity of the sum
of connections, which is proven in Theorem \ref{theo:tensorcon}.
\end{proof}
\begin{Remark}
The deformation of the sum of a connection and its dual connection is
canonical, in fact  we have already shown that sum,
dualization and the left to right isomorphism $\widetilde{D}_{\RR}^{-1}$  are compatible with deformation, 
hence their composition is also compatible.
\end{Remark}

\section{\label{sec:curvature}Curvature}
After reviewing the basic definitions we investigate the behaviour of
the curvature of a connection under twist deformation. The twist
deformed curvature in general differs from the curvature
of the twist deformed connection, hence flat connections are twisted in
non flat ones and vice versa. We then explicitly calculate the curvature of the sum of
two connections, $\nabla_V\oplus_\RR\nabla_W$, and find that it is in
general not simply given by the sum of the individual curvatures.

\subsection{Definitions}
A connection on an $\MMMod{}{}{A}$-module $V$ extends as usual to a
$\bbK$-linear map from $V\otimes_A\Omega^\bullet$ to $V\otimes_A\Omega^\bullet$
(cf. (\ref{eqn:nabliftinduced})), where $(\Omega^\bullet,\wedge,\dif)$
is a differential calculus over $A$. 
For later purposes we analyse in subsequent steps this construction. 
We then define the curvature of a connection.

\begin{Lemma}\label{lem:connectionlift}
Let $V$ be an 
$\MMMod{}{}{A}$-module,
$(\Omega^\bullet,\wedge,\dif)$ a differential calculus over the
algebra $A$, and 
$\nabla \in\Con_A(V)$. We define the $\bbK$-linear map
 $\nabla^\bullet : V\otimes \Omega^\bullet \to V\otimes_A \Omega^\bullet$
by
\begin{flalign}\label{eqn:nablift}
\nabla^\bullet  := (\id\otimes \wedge) \circ (\nabla\otimes\id) + \pi\circ (\id\otimes\dif)~.
\end{flalign}
Here $\pi:V\otimes \Omega^\bullet \to V\otimes_A \Omega^\bullet$ is the canonical projection
and $\id\otimes \wedge: (V\otimes_A \Omega^\bullet)\otimes \Omega^\bullet \to V\otimes_A\Omega^\bullet$
is the right $A$-linear map defined by, for all $v\in V$ and $\omega,\tilde \omega\in\Omega^\bullet$,
\begin{flalign}
(\id\otimes\wedge)\big((v\otimes_A\omega)\otimes\tilde\omega\big) := (v\otimes_A\omega)\wedge \tilde\omega:= v\otimes_A \omega\wedge\tilde\omega~,
\end{flalign} 
and extended to all $ (V\otimes_A \Omega^\bullet)\otimes \Omega^\bullet $ by $\bbK$-linearity.

The map $\nabla^\bullet$ induces a well-defined map (still denoted by $\nabla^\bullet$ for ease of notation)
on the quotient, $\nabla^\bullet: V\otimes_A\Omega^\bullet \to V\otimes_A\Omega^\bullet$. Explicitly, for
all $v\in V$ and $\omega\in \Omega^\bullet$,
\begin{flalign}\label{eqn:nabliftinduced}
\nabla^\bullet ( v\otimes_A\omega ) = \nabla v\wedge \omega + v\otimes_A \dif \omega ~.
\end{flalign}

The map $\nabla^\bullet: V\otimes_A\Omega^\bullet \to V\otimes_A\Omega^\bullet$ satisfies,
for all $v\in V$, $a\in A$ and $\omega\in\Omega^\bullet$ of
homogeneous degree,
\begin{flalign}\label{eqn:nabliftleibnitz}
\nabla^\bullet\big((v\otimes_A \omega) a\big) 
= \big(\nabla^\bullet (v\otimes_A\omega) \big) a + (-1)^{\deg(\omega)}\,v\otimes_A \omega\wedge \dif a~.
\end{flalign}

\end{Lemma}
\begin{proof}
The map (\ref{eqn:nablift}) induces a well-defined map on the quotient,
since it annihilates the submodule $\mathcal{N}_{V,\Omega^\bullet}={\ker}(\pi)$, 
i.e.~$\nabla^\bullet\big(\mathcal{N}_{V,\Omega^\bullet}\big)=\{0\}$. Indeed, for
all $a\in A$, $v\in V$ and $\omega\in\Omega^\bullet$,
\begin{flalign}
\nn\nabla^\bullet(v\cdot a\otimes \omega) &= \nabla(v\cdot a)\wedge \omega+ v\cdot a\otimes_A \dif\omega\\
\nn&=  (\nabla v)_{\,}a\wedge \omega + v\otimes_A\dif a\wedge \omega + v\cdot a\otimes_A \dif\omega\\
\nn &= \nabla v\wedge a \,\omega + v\otimes_A(\dif a \wedge \omega + a\,\dif\omega)\\
\nn &=  \nabla v\wedge a \,\omega + v\otimes_A\dif (a \,\omega)\\
&=\nabla^\bullet (v\otimes a\,\omega)~.
\end{flalign}

Property (\ref{eqn:nabliftleibnitz}) hold because, for all
$v\in V$, $a\in A$ and $\omega\in\Omega^\bullet$ of homogeneous degree,
\begin{flalign}
\nn \nabla^\bullet\big(v\otimes_A \omega a\big) 
&= \nabla v\wedge \omega\,a + v\otimes_A\dif(\omega\,a)\\
\nn&= \nabla v\wedge \omega  a  + v\otimes_A\Big((\dif\omega)\,a+ (-1)^{\deg(\omega)}\,\omega\wedge\dif a\Big)\\
&= \big( \nabla^\bullet(v\otimes_A \omega)\big) a
+(-1)^{\deg(\omega)} \,v\otimes_A \omega\wedge\dif a~.
\end{flalign}
\end{proof}
\begin{Definition}\label{defi:curvature}
Let $\nabla\in\Con_A(V)$. The {\bf curvature} of $\nabla$ is the $\bbK$-linear map defined by
\begin{flalign}
R_\nabla:=\nabla^\bullet \circ\nabla \;:\; V\to V\otimes_A\Omega^2~.
\end{flalign}
\end{Definition}
\sk
The curvature is right $A$-linear,
i.e.~$R_\nabla\in\Hom_A(V,V\otimes_A\Omega^2)$, indeed,
for all $a\in A$ and $v\in V$,
\begin{flalign}
\nn R_\nabla(v\cdot a) &= \nabla^\bullet\big(\nabla(v\cdot a)\big) = \nabla^\bullet\big((\nabla v) a + v\otimes_A\dif a\big)\\
\nn &=\nabla^\bullet\big(\nabla v \big) \, a - \nabla v \wedge \dif a + \nabla v \wedge \dif a + v\otimes_A \dif{\:\!}\dif a \\
&=\nabla^\bullet\big(\nabla v\big)\, a = R_\nabla(v)\, a~.
\end{flalign}


\subsection{Curvature of deformed connections}
 Let $H$ be a Hopf algebra with twist $\FF\in H\otimes H$,
 $A$ be an $\AAAlg{H}{}{}$-algebra, $V$ be an $\MMMod{H}{}{A}$-module
 and
 $(\Omega^\bullet,\wedge,\dif)$ be a left $H$-covariant differential calculus over $A$.
We denote by  $H^\FF$, $A_\st$, $V_\st$ and  $(\Omega^\bullet,\wedge_\st,\dif)$
the deformations of $H$, $A$, $V$  and $(\Omega^\bullet,\wedge,\dif)$ 
(obtained  applying  Theorems \ref{TwistedHopfAlg},
\ref{Theorem1}, \ref{Theorem2} and  Proposition \ref{lem:dcdef}, respectively).
\sk

Consider an arbitrary connection $\nabla_\st\in\Con_{A_\st}(V_\st)$, then,
following Lemma \ref{lem:connectionlift}, we have a well-defined
extension  $\nabla_\st^{\bullet_\star}:
 V_\st \otimes_{A_\st} \Omega^\bullet_\st \to
 V_\st\otimes_{A_\st}\Omega^\bullet_\st$. The $\bbK$-linear map 
 $\nabla_\star^{\bullet_\star} : V_\st \otimes_{\star} \Omega^\bullet_\st \to
 V_\st\otimes_{A_\st}\Omega^\bullet_\st$ is defined by
\eq
\nabla_\st^{\bullet_\star} := (\id\otimes_\st \wedge_\st) \circ (\nabla_\st\otimes_\st\id) 
+ \pi_\st\circ (\id\otimes_\st\dif)\,,
\en
where $\pi_\st : V_\st\otimes_\st \Omega^\bullet_\st \to V_\st \otimes_{A_\st}\Omega^\bullet_\st$
 is the canonical projection, and $\id\otimes_\st\wedge_\st : (V_\st \otimes_{A_\st}\Omega^\bullet_\st)\otimes_\st \Omega^\bullet_\st \to  V_\st \otimes_{A_\st}\Omega^\bullet_\st$ is defined by, for all
 $v\in V_\st$ and $\omega,\tilde\omega\in\Omega^\bullet_\st$,
 \begin{flalign}
 (\id\otimes_\st\wedge_\st)\big((v\otimes_{A_\st}\omega)\otimes_\st\tilde\omega\big) := (v\otimes_{A_\st} \omega)\wedge_\star \tilde\omega := v\otimes_{A_\st}\omega\wedge_\st\tilde\omega ~.
 \end{flalign}
 The induced map $\nabla_\st^{\bullet_\star}:
 V_\st \otimes_{A_\st} \Omega^\bullet_\st \to
 V_\st\otimes_{A_\st}\Omega^\bullet_\st$ reads explicitly, for all $v\in V_\star$ and $\omega\in \Omega^\bullet_\star$,
$ \dd^{\bullet_\star}_\st(v\otimes_{A_\star} \omega)=(\nabla_\st v)\wedge_\st
\omega + v\otimes_{A_\st} \dif \omega$.
\begin{Lemma}\label{lem:conextdef}
 Let $H$ be a Hopf algebra with twist $\FF\in H\otimes H$,
 $A$ be an $\AAAlg{H}{}{}$-algebra, $V$ be an $\MMMod{H}{}{A}$-module and 
 $(\Omega^\bullet,\wedge,\dif)$  be a left $H$-covariant differential calculus over $A$. 
 For any $\dd\in  \Con_{A}(V)$
 we have the commutative diagram
\begin{flalign}\label{eqn:connbullet}
\xymatrix{
{V_\star\otimes_{A_\st}\Omega_\st^\bullet}\ar[d]_-{\varphi_{V_\star,\Omega^\bullet_\star}}
\ar[rrr]^-{{\widetilde{D}_\FF(\dd)}^{\,\bullet_\st}} & & &
 { V_\star\otimes_{A_\st}\Omega^\bullet_\star} \ar[d]^-{\varphi_{V_\star,\Omega^\bullet_\star}}\\
(V\otimes_A\Omega^\bullet)_\star
\ar[rrr]^-{D_{\FF}(\dd^\bullet)}& & &(V\otimes_A\Omega^\bullet)_\star
}
\end{flalign}
\end{Lemma}
\begin{proof}
$H$-equivariance of all maps in (\ref{eqn:nablift}), but $\nabla$, implies
\begin{flalign}
\nn D_\FF(\nabla^\bullet) &= D_\FF\Big( (\id\otimes \wedge) \circ (\nabla\otimes\id) + \pi\circ (\id\otimes\dif)\Big)\\
&=(\id\otimes\wedge)\circ D_\FF(\nabla\otimes\id) + \pi\circ (\id\otimes\dif)~.
\end{flalign} 
Theorem \ref{theo:promodhomdef} implies
\begin{flalign}
\nn D_\FF(\nabla\otimes\id) &= \varphi_{V\otimes_A\Omega^\bullet,\Omega^\bullet}\circ \big(D_\FF(\nabla)\otimes_\st\id\big) \circ \varphi^{-1}_{V,\Omega^\bullet} \\
&=\varphi_{V\otimes_A\Omega^\bullet,\Omega^\bullet}\circ (\varphi_{V_\st,\Omega^\bullet_\st}\otimes_\st\id)\circ \big(\widetilde{D}_\FF(\nabla)\otimes_\st\id\big) \circ \varphi^{-1}_{V,\Omega^\bullet}~.
\end{flalign}
The proof follows from the two
 commutative diagrams
 \begin{subequations}
 \begin{flalign}
 \xymatrix{
(V_\st\otimes_{A_\st}\Omega^\bullet_\st )\otimes_\st \Omega^\bullet_\st \ar[d]_-{ \varphi_{V_\st,\Omega_\st^\bullet} \otimes_\st\id}
\ar[rrr]^-{\id\,\otimes_\st\,\wedge_\st} & &  & V_\st\otimes_{A_\st} \Omega^\bullet_\st \ar[dd]^-{\varphi_{V_\st,\Omega_\st^\bullet}}
\\
( V\otimes_A\Omega^\bullet)_\st\otimes_\st \Omega^\bullet_\st  \ar[d]_-{\varphi_{V\otimes_A\Omega^\bullet,\Omega^\bullet}}
& & & \\
\big((V\otimes_A\Omega^\bullet)\otimes \Omega^\bullet\big)_\st \ar[rrr]_-{\id\,\otimes\,\wedge}& & & (V\otimes_A\Omega^\bullet)_\st
 }
\end{flalign}
 and
 \begin{flalign}
 \xymatrix{
 V_\st\otimes_{\st} \Omega^\bullet_\st \ar[d]_-{\varphi_{V,\Omega^\bullet}}\ar[rr]^-{\id\,\otimes_\st \dif} &&  V_\st\otimes_{\st} \Omega^\bullet_\st \ar[rr]^-{\pi_\st} \ar[d]_-{\varphi_{V,\Omega^\bullet}}  && V_\st\otimes_{A_\st}\Omega^\bullet_\st \ar[d]^-{\varphi_{V_\st,\Omega^\bullet_\st}}\\
 (V\otimes \Omega^\bullet)_\st \ar[rr]_-{\id\,\otimes\dif}&&  (V\otimes \Omega^\bullet)_\st \ar[rr]_-{\pi}&& (V\otimes_A\Omega^\bullet)_\st
 }
 \end{flalign}
 \end{subequations}
which are a consequence of Theorem \ref{theo:promodhomdef}, the $H$-equivariance of the maps and Lemma \ref{lemma4}.
\end{proof}

The curvature of a connection $\nabla_\st\in\Con_{A_\st}(V_\st)$ is
defined by $R_{\nabla_\st} := \nabla_\st^{\bullet_\st}\circ \nabla_\st \in
\Hom_{A_\st}(V_\st,V_\st\otimes_{A_\st}\Omega_\st^2)$.  Due to the
isomorphism  $\Con_A(V)\simeq Con_{A_\st}(V_\st)$ (cf. Theorem
\ref{theo:condef} and the text after the theorem) any
connection $\nabla_\st\in\Con_{A_\st}(V_\st)$ is the image $\widetilde
D_\FF(\dd)$ of a connection  $\nabla\in\Con_{A}(V)$.

 \begin{Theorem}\label{theo:curvaturedeform} 
 Let $H$ be a Hopf algebra with twist $\FF\in H\otimes H$,
 $A$ be an $\AAAlg{H}{}{}$-algebra, $V$ be an $\MMMod{H}{}{A}$-module and 
 $(\Omega^\bullet,\wedge,\dif)$  be a left $H$-covariant differential
 calculus over $A$. 
Consider an arbitrary connection
$\widetilde{D}_\FF(\nabla)\in\Con_{A_\st}(V_\st),$ (where $\nabla\in\Con_A(V)$). 
The curvature of the deformed connection $\widetilde{D}_\FF(\nabla)$ satisfies the identity
 \begin{flalign}
 R_{\widetilde{D}_\FF(\nabla)}  = \varphi^{-1}_{V_\st,\Omega^2_\st}\circ D_\FF\big(\nabla^\bullet\circ_\st \nabla\big) = \widetilde{D}_\FF\big(\nabla^{\bullet}\circ_\st \nabla\big)~.
 \end{flalign}
 \end{Theorem}
 \begin{proof}
The proof follows from  Lemma \ref{lem:conextdef} and the property 
$D_\FF(Q\circ_\star \check Q) = D_\FF(Q)\circ D_\FF(\check Q)$, which holds for
 any two composable $\bbK$-linear maps $Q,\check Q$.
\end{proof}
Notice that the deformed curvature $\widetilde D_\FF(R_\dd)=\widetilde
D_\FF(\dd^\bullet\circ\dd)$ differs from the curvature of the deformed
connection $R_{\widetilde D_\FF(\dd)}=\widetilde
D_\FF(\dd^{\bullet}\circ_\st\dd)$. The right hand side of this latter 
equality can be recognized as the deformation of the
curvature of the connection $\dd$ seen as an element of the affine space
$\Con_{A}(V)_\st$ (rather than $\Con_A(V)$). Indeed, in this case
$\dd\in\Con_{A}(V)_\st$ can be seen as a morphism in the category
$\rep^H{}_{\,\st}$, which is characterized by the $\star$-composition $\circ_\st$.
\sk
The difference between $\widetilde D_\FF(R_\dd) $ and $R_{\widetilde
  D_\FF(\dd)} $ immediately leads to the following corollary on flat connections.
 \begin{Corollary}
 In the hypotheses of Theorem \ref{theo:curvaturedeform}, consider an
 arbitrary connection $\widetilde{D}_\FF(\nabla)\in\Con_{A_\st}(V_\st)$,
 (where $\nabla\in\Con_A(V)$). Then the connection $\widetilde{D}_\FF(\nabla)$ is flat
 (i.e.~$R_{\widetilde{D}_\FF(\nabla)} =0$) if and only if $\nabla^\bullet \circ_\st \nabla =0$.
  In general, the deformation
 of a flat connection $\nabla\in\Con_A(V)$ (i.e.~$R_\nabla=0$) is not flat (i.e.~$R_{\widetilde{D}_\FF(\nabla)}\neq 0$). 
 Notice, however, that in the case the flat connection $\nabla\in\Con_A(V)$ is additionally $H$-equivariant,
  then the deformed connection $\widetilde{D}_\FF(\nabla)$ is also flat.
 \end{Corollary} 
 \sk
 The study of the cohomology of flat deformed connections could lead to new cohomology invariants 
 or interesting combinations of undeformed ones.


\subsection{Curvature of the sum of connections}
We conclude this section by calculating the curvature of the sum of
two connections.
Let $(H,\RR)$ be a quasitriangular Hopf algebra,  $A$ be a quasi-commutative $\AAAlg{H}{}{}$-algebra,
$W$ be a quasi-commutative $\MMMod{H}{A}{A}$-module and
 $\bigl(\Omega^\bullet,\wedge,\dif\bigr)$ be a graded
 quasi-commutative left $H$-covariant differential calculus over $A$.
 For an arbitrary $\MMMod{H}{}{A}$-module $V$ and arbitrary connections $\dd_V\in \Con_A(V)$,
 $\dd_W\in\Con_A(W)$,
 the sum of connections $\nabla_V\oplus_\RR \nabla_W\in\Con_A(V\otimes_A W) $ is given 
 by (cf.~Theorem \ref{theo:conplus}), for
 all $v\in V$ and $w\in W$,
 \begin{flalign}\label{eqn:sumconincurv}
 (\dd_V \oplus_\RR \dd_W)
(v\otimes_A w)=
 \tau^{-1}_{\RR\,23}\bigl(\dd_V (v)\otimes_A w\bigr)
+ (\oR^\be\ra v)\otimes_A (\oR_\be\RA\dd_W)(w)~.
 \end{flalign}
 We use the following formal notation
 \begin{flalign}\label{eqn:sumconformal}
 \dd_V \oplus_\RR \dd_W = \tau_{\RR\,23}^{-1}\circ (\dd_V \otimes_\RR \id_W) +  \id_V\otimes_\RR \dd_W~,
 \end{flalign}
 which is understood to give (\ref{eqn:sumconincurv}) when acting on generating elements $v\otimes_A w$.
 Here all maps are considered to act on the tensor product over $A$ (this is why there are no projections in this expression), and the 
 notation is formal in the sense that the individual terms in (\ref{eqn:sumconformal}) are not well-defined maps on 
 $V\otimes_A W$, but only their sum is.
The reason for introducing this compact notation is that it simplifies
the lengthy calculation in  Proposition \ref{prop7.5}.
The equations appearing in the proof of this proposition should be
understood as holding true when acting on 
 generating elements $v\otimes_A w$ of $V\otimes_A W$.

The canonical extension (\ref{eqn:nabliftinduced}) of the connection 
$\dd_V \oplus_\RR \dd_W$ reads, for all $v\in V$, $w\in W$ and $\omega\in \Omega^\bullet$,
 \begin{flalign}\label{eqn:liftsumconincurv}
 (\dd_V \oplus_\RR \dd_W)^\bullet(v\otimes_A w\otimes_A\omega) = \big((\dd_V \oplus_\RR \dd_W)(v\otimes_A w)\big)
 \wedge \omega 
 + v\otimes_A w\otimes_A\dif \omega~.
 \end{flalign}
Also here we use  a formal notation similar to (\ref{eqn:sumconformal}), 
\begin{flalign}\label{eqn:liftsumconformal}
(\dd_V \oplus_\RR \dd_W)^\bullet = (\id_{V\otimes_A W}\otimes_\RR \wedge) \circ 
\big((\dd_V \oplus_\RR \dd_W) \otimes_\RR\id_{\Omega^\bullet}\big) + \id_{V\otimes_A W}\otimes_\RR \dif~.
\end{flalign}
Again, the individual terms in this expression are not well-defined maps on $V\otimes_A W\otimes_A \Omega^\bullet$,
 but their sum is.
 
 We express  the curvature $R_{\nabla_V\oplus_\RR\nabla_W}$
in terms of the curvatures $R_{\nabla_V}$ and $R_{\nabla_W}$.
 \begin{Proposition}\label{prop7.5}
Let $(H,\RR)$ be a quasitriangular Hopf algebra, $A$ be a quasi-commutative $\AAAlg{H}{}{}$-algebra,
 $W$ be a quasi-commutative $\MMMod{H}{A}{A}$-module and
 $\bigl(\Omega^\bullet,\wedge,\dif\bigr)$ be a graded
 quasi-commutative left $H$-covariant differential calculus over $A$.
 Then for any $\MMMod{H}{}{A}$-module $V$ and arbitrary connections $\dd_V\in \Con_A(V)$,
 $\dd_W\in\Con_A(W)$,
 the curvature $R_{\nabla_V\oplus_\RR\nabla_W} \in\Hom_A(V\otimes_AW,V\otimes_AW\otimes_A\Omega^2)$ satisfies the identity
 \begin{multline}\label{eqn:curvaturesum}
 R_{\nabla_V\oplus_\RR\nabla_W}= \tau^{-1}_{\RR\,23}\circ (R_{\nabla_V}\otimes_\RR \id_W \big)
 + \id_V \otimes_\RR R_{\nabla_W} \\
 + (\id_{V\otimes_A W}\otimes_\RR\wedge)\circ \tau^{-1}_{\RR\,23} \circ \Big( \nabla_V \otimes_\RR \nabla_W - (\oR^\al\RA\nabla_V) \otimes_\RR (\oR_\al\RA \nabla_W )\Big)~,
 \end{multline}
 where $R_{\nabla_V}\in\Hom_A(V,V\otimes_A\Omega^2)$ and $R_{\nabla_W}\in\Hom_A(W,W\otimes_A\Omega^2)$ 
 are the curvatures of $\nabla_V$ and $\nabla_W$, respectively.
 The second line in (\ref{eqn:curvaturesum}) is understood in the same formal notation as used in (\ref{eqn:sumconformal}) 
 and (\ref{eqn:liftsumconformal}).
 \end{Proposition}
 \begin{proof}
 We have to calculate $R_{\nabla_V\oplus_\RR\nabla_W} = (\dd_V \oplus_\RR \dd_W)^\bullet \circ (\dd_V \oplus_\RR \dd_W)$.
 We use $\id_{V\otimes_A W} = \id_V\otimes_\RR \id_W =
 \id_V\otimes\id_W$. 
In order to simplify the notation
 we drop all module indices on the identity maps and write $\id_{V\otimes_A W}\otimes \wedge 
 = \id\otimes\id\otimes \wedge = \wedge_{34}$. We also recall that $P\otimes_\RR Q = P\otimes Q$, 
 whenever the map $Q$ is $H$-equivariant.
 
 Using the notation introduced in  (\ref{eqn:sumconformal}) and (\ref{eqn:liftsumconformal}), the curvature is given by the sum of the
  following four terms
 \begin{subequations} \label{eqn:sumcurvintermediate}
 \begin{flalign}
 \label{I1}  R_{\nabla_V\oplus_\RR\nabla_W} &=\Big(\wedge_{34} \circ\tau_{\RR\,23}^{-1}\circ (\nabla_V\otimes\id\otimes\id) + \id\otimes\id\otimes\dif\Big) \circ \tau_{\RR\,23}^{-1}\circ (\nabla_V\otimes\id)\\
 \label{I2}&~~~+\Big(\wedge_{34}\circ (\id\otimes_\RR\nabla_W\otimes_\RR\id) + \id\otimes\id\otimes\dif\Big)\circ (\id\otimes_\RR\nabla_W)\\
\label{I3}&~~~+ \wedge_{34}\circ \tau_{\RR\,23}^{-1} \circ (\nabla_V\otimes\id\otimes\id) \circ (\id\otimes_\RR\nabla_W)\\
 \label{I4} &~~~+\wedge_{34} \circ (\id\otimes_\RR\nabla_W \otimes_\RR\id) \circ \tau_{\RR\,23}^{-1}\circ (\nabla_V\otimes\id)~.
 \end{flalign}
 \end{subequations}
 We now simplify the individual terms in (\ref{eqn:sumcurvintermediate}). For the first term we find
 \begin{flalign}
\nn (\text{\ref{I1}}) &= \Big(\wedge_{34} \circ\tau_{\RR\,23}^{-1}\circ (\nabla_V\otimes\id\otimes\id) \circ\tau_{\RR\,23}^{-1} + (\id\otimes\id\otimes\dif) \circ \tau_{\RR\,23}^{-1}\Big) \circ (\nabla_V\otimes\id)\\
 \nn &=\Big(\wedge_{34} \circ\tau_{\RR\,23}^{-1}\circ \tau_{\RR\,34}^{-1}\circ (\nabla_V\otimes\id\otimes\id) +\tau_{\RR\,23}^{-1}\circ  (\id\otimes\dif\otimes\id) \Big) \circ (\nabla_V\otimes\id)\\
 \nn &=\Big(\wedge_{34} \circ\tau_{\RR\,(23)4}^{-1}\circ (\nabla_V\otimes\id\otimes\id) +\tau_{\RR\,23}^{-1}\circ  (\id\otimes\dif\otimes\id) \Big) \circ (\nabla_V\otimes\id)\\
\nn &=\tau_{\RR\,23}^{-1}\circ \Big(\wedge_{23}\circ  (\nabla_V\otimes\id\otimes\id) + \id\otimes\dif\otimes\id \Big) \circ (\nabla_V\otimes\id)\\
\nn &=\tau_{\RR\,23}^{-1}\circ \Big(\big((\wedge_{23}\circ  (\nabla_V\otimes\id) + \id\otimes\dif) \circ \nabla_V\big)\otimes\id\Big)\\
&= \tau_{\RR\,23}^{-1}\circ (R_{\nabla_V}\otimes\id)~.
 \end{flalign}
 For the second term we find
 \begin{flalign}
 \nn (\text{\ref{I2}}) &= \Big(\id\otimes_\RR \big(\wedge_{23} \circ (\nabla_W\otimes\id) + \id\otimes\dif\big)\Big) \circ (\id\otimes_\RR \nabla_W) \\
 \nn &= \id\otimes_\RR \Big(\big(\wedge_{23} \circ (\nabla_W\otimes\id) + \id\otimes\dif\big)\circ \nabla_W\Big)\\
 &= \id\otimes_\RR R_{\nabla_W}~.
 \end{flalign}
 The third term simplifies to
 \begin{flalign}
 \nn (\text{\ref{I3}}) &= \wedge_{34}\circ \tau_{\RR\,23}^{-1} \circ \big(\nabla_V\otimes_\RR (\id\otimes\id)\big) \circ (\id\otimes_\RR\nabla_W)\\
 &= \wedge_{34}\circ \tau_{\RR\,23}^{-1}\circ \big(\nabla_V\otimes_\RR\nabla_W\big)~,
 \end{flalign}
 and the last term to
 \begin{flalign}
 \nn (\text{\ref{I4}}) &= \wedge_{34}\circ \Big(\id\otimes_\RR\big((\nabla_W\otimes\id)\circ \tau_\RR^{-1}\big)\Big)\circ (\nabla_V\otimes\id)\\
 \nn &= \wedge_{34}\circ \Big(\id\otimes_\RR\big(\tau_{\RR\,1(23)}^{-1}\circ (\id\otimes_\RR\nabla_W)\big)\Big)\circ (\nabla_V\otimes\id)\\
\nn &=\wedge_{34}\circ \tau_{\RR\,2(34)}^{-1}\circ \big(\id\otimes_\RR\id\otimes_\RR\nabla_W\big)\circ (\nabla_V\otimes\id)\\
\nn &=\wedge_{34}\circ \tau_{\RR\,2(34)}^{-1}\circ \big((\oR^\alpha\RA \nabla_V)\otimes_\RR (\oR_\alpha\RA \nabla_W)\big)\\
\nn &=\wedge_{34}\circ \tau_{\RR\,34}^{-1}\circ \tau_{\RR\,23}^{-1}\circ \big((\oR^\alpha\RA \nabla_V)\otimes_\RR (\oR_\alpha\RA \nabla_W)\big)\\
&=- \wedge_{34}\circ\tau_{\RR\,23}^{-1}\circ \big((\oR^\alpha\RA \nabla_V)\otimes_\RR (\oR_\alpha\RA \nabla_W)\big)~.
 \end{flalign}
The sum of these four terms gives (\ref{eqn:curvaturesum}).
 \end{proof}
 \begin{Remark}
The first line in  (\ref{eqn:curvaturesum}) is the sum (of the lift to
$V\otimes_A W$) of the curvatures $R_{\nabla_V}$ and $R_{\nabla_{W}}$. 
The curvature $R_{\nabla_V\oplus_\RR\nabla_W}$ is not simply the sum of 
 $R_{\nabla_V}$ and $R_{\nabla_{W}}$; the second line in
 (\ref{eqn:curvaturesum}) gives an additional contribution due to the
 non $H$-equivariance of the connections. 
 In the special case that either $\nabla_V$ or $\nabla_W$ are
  $H$-equivariant,
 the second line in (\ref{eqn:curvaturesum}) vanishes and the curvature $R_{\nabla_V\oplus_\RR\nabla_W}$
 is simply the sum of the individual curvatures. 
\end{Remark}


\section*{Acknowledgements}
We would like to thank Ugo Bruzzo, Tomasz Brzezi{\'n}ski, 
Leonardo Castellani, Branislav Jur{\v c}o, Thorsten Ohl, Christoph F.~Uhlemann  
and Walter van Suijlekom for useful discussions and comments.

The work of P.A.~is in part supported by the exchange grant 2646 of the ESF Activity {\it Quantum Geometry 
and Quantum Gravity}, by the Deutsche Forschungsgmeinschaft through the Research Training Group GRK 1147 {\it Theoretical Astrophysics and Particle Physics} and by a
short term visit grant of CERN TH Division.

The work of A.S.~is in part supported by the Deutsche Forschungsgmeinschaft through the Research Training Group 
GRK 1147 {\it Theoretical Astrophysics and Particle Physics}, by the short visit grant 3267 of the ESF Activity {\it Quantum Geometry 
and Quantum Gravity} and by INFN sezione di Torino, gruppo collegato di Alessandria.



\end{document}